\documentclass{amsart}

\usepackage{amssymb, amsfonts, amsmath, amsthm,bm, amscd}
\usepackage{mathrsfs}
\usepackage{graphicx}
\usepackage{caption}
\usepackage{subcaption}
\usepackage{wrapfig}
\usepackage{enumitem, multicol}
\usepackage{tikz}
\usetikzlibrary{cd}
\usepackage{nicefrac, bigints}
\usepackage{array}
\usepackage{stmaryrd}
\usepackage{mathtools} 
\usepackage{extarrows} 
\usepackage[margin=1.00 in]{geometry}

\usepackage[colorlinks,citecolor=cyan,linkcolor=teal]{hyperref}

\newcommand{\para}[1]{\medskip\noindent\textbf{#1.}}

\theoremstyle{definition}
\newtheorem{theorem}{Theorem}[section]
\newtheorem{maintheorem}{Theorem}

\newtheorem{definition}[theorem]{Definition}
\newtheorem{lemma}[theorem]{Lemma}

\newtheorem{corollary}[theorem]{Corollary}
\newtheorem{proposition}[theorem]{Proposition}
\newtheorem{remark}[theorem]{Remark}

\newtheorem{claim}[theorem]{Claim}
\newtheorem{fact}[theorem]{Fact}

\newtheorem{construction}[theorem]{Construction}
\newtheorem{example}[theorem]{Example}

\makeatletter
\DeclareRobustCommand{\cev}[1]{%
  {\mathpalette\do@cev{#1}}%
}
\newcommand{\do@cev}[2]{%
  \vbox{\offinterlineskip
    \sbox\z@{$\m@th#1 x$}%
    \ialign{##\cr
      \hidewidth\reflectbox{$\m@th#1\vec{}\mkern4mu$}\hidewidth\cr
      \noalign{\kern-\ht\z@}
      $\m@th#1#2$\cr
    }%
  }%
}
\makeatother

\newcommand{\ZZ}{\mathbb{Z}}
\newcommand{\RR}{\mathbb{R}}

\newcommand{\CC}{\mathbb{C}}
\newcommand{\HH}{\mathbb{H}}

\newcommand{\inverse}{^{-1}}

\newcommand{\cN}{\mathcal{N}}
\newcommand{\cE}{\mathcal{E}}

\DeclareMathOperator{\Mod}{Mod}

\newcommand{\MF}{\mathcal{MF}}

\newcommand{\ML}{\mathcal{ML}}

\newcommand{\GL}{\mathsf{GL}}
\DeclareMathOperator{\GLcr}{\mathcal{GL}^{cr}}

\newcommand{\SL}{\mathsf{SL}}
\DeclareMathOperator{\im}{im}

\newcommand{\T}{\mathcal{T}}
\newcommand{\M}{\mathcal{M}}

\newcommand{\cH}{\mathcal{H}}

\newcommand{\PT}{\mathcal{PT}}
\newcommand{\PM}{\mathcal{PM}}
\newcommand{\PoT}{\mathcal{P}^1\mathcal{T}}
\newcommand{\PoM}{\mathcal{P}^1\mathcal{M}}
\newcommand{\QT}{\mathcal{QT}}
\newcommand{\QM}{\mathcal{QM}}
\newcommand{\QoT}{\mathcal{Q}^1\mathcal{T}}
\newcommand{\QoM}{\mathcal{Q}^1\mathcal{M}}
\newcommand{\sing}{\underline{\kappa}}

\newcommand{\SH}{\mathcal{SH}}
\DeclareMathOperator{\Sp}{\mathsf{Sp}}

\DeclareMathOperator{\PSL}{\mathsf{PSL}}

\DeclareMathOperator{\Area}{Area}
\DeclareMathOperator{\diam}{diam}

\newcommand{\tlambda}{\widetilde{\lambda}}
\newcommand{\tX}{\widetilde{X}}

\newcommand{\tS}{\widetilde{S}}

\newcommand{\taua}{\tau_{\arc}}

\newcommand{\sigl}{\sigma_\lambda}
\newcommand{\ac}{\mathfrak{s}}

\newcommand{\cO}{\mathcal{O}}

\DeclareMathOperator{\Ol}{\mathcal{O}_\lambda}

\newcommand{\arc}{\underline{\alpha}}
\newcommand{\arcwt}{\underline{A}}
\newcommand{\arcb}{\underline{\beta}}

\newcommand{\arcc}{\underline{\gamma}}

\newcommand{\Arcfill}{\mathscr{A}_\text{fill}}

\newcommand{\Fol}{\mathcal W}

\newcommand{\Base}{\mathscr{B}(S \setminus \lambda)}

\DeclareMathOperator{\Eq}{Eq}

\renewcommand{\Re}{\operatorname{Re}}
\renewcommand{\Im}{\operatorname{Im}}

\newcommand{\cQ}{\mathcal{Q}}

\DeclareMathOperator{\hol}{hol}
\DeclareMathOperator{\ImAnn}{ImAnn}

\colorlet{lgray}{gray!40}

\newcommand{\gvarc}{\arc_{\bullet}(X,\lambda,\delta)}
\DeclareMathOperator{\injrad}{injrad}
\DeclareMathOperator{\Thick}{Thick}
\DeclareMathOperator{\Thin}{Thin}
\DeclareMathOperator{\Dil}{Dil}
\DeclareMathOperator{\sys}{sys}
\DeclareMathOperator{\Hor}{Hor}
\DeclareMathOperator{\Hom}{Hom}

\newcommand{\AIS}{\mathcal L}

\newcommand{\Addresses}{{
  \bigskip
  \footnotesize
  \noindent Aaron Calderon, \textsc{Department of Mathematics, University of Chicago}\par\nopagebreak
  \textit{E-mail address}: \texttt{aaroncalderon@uchicago.edu}
  
  \noindent James Farre, \textsc{Max Planck Institute for Mathematics in the Sciences, Leipzig}\par\nopagebreak
  \textit{E-mail address}: \texttt{james.farre@mis.mpg.de}
  }}

\begin{document}

\title[Continuity of the orthogeodesic foliation]{Continuity of the orthogeodesic foliation \\ and ergodic theory of the earthquake flow}
\author{Aaron Calderon}
\author{James Farre}

\setcounter{tocdepth}{1}

\begin{abstract}
In a previous paper, the authors extended Mirzakhani's (almost-everywhere defined) measurable conjugacy between the earthquake and horocycle flows to a measurable bijection. 
In this one, we analyze the continuity properties of this map and its inverse, proving that both are continuous at many points and in many directions.
This lets us transfer measure convergence between the two systems, allowing us to pull back results from Teichm{\"u}ller dynamics to deduce analogous statements for the earthquake flow.
\end{abstract}

\maketitle
\vspace{-5ex}
\thispagestyle{empty}

\section{Introduction}

Let $P$ be the upper triangular subgroup of $\SL_2 \RR$ and let $S$ be a closed surface of genus $g\ge 2$.
We consider two $P$-actions on the cotangent bundle $T^*\M_g$ to the moduli space $\M_g$ of genus $g$ Riemann surfaces.
From the complex analytic theory, $T^*\M_g$ is identified with the bundle $\QM_g$ of holomorphic quadratic differentials, or certain foliated singular flat metrics, on $S$.
From the perspective of hyperbolic geometry, $T^*\M_g$ can be viewed as the bundle $\PM_g$ of measured geodesic laminations on hyperbolic surfaces homeomorphic to $S$.

There is a natural action of $P$ on the space $\QM_g$ of singular flat structures by affine deformations of the metric.
The diagonal and unipotent subgroups
\[A=\left\{g_t = \begin{pmatrix}
e^{t} & 0 \\
0 & e^{-t}
\end{pmatrix}\right\}
\text{ and }
U=\left\{u_s = \begin{pmatrix}
1 & s \\
0 & 1
\end{pmatrix}\right\}
\]
of $P$ are the Teichm\"uller and horocycle flows, respectively.
The $P$-action has been used to study interval exchange transformations, giving a solution to the Keane conjecture \cite{Masur_IETsMF, Veech_IETs}, 
to study the dynamics of polygonal billiards
\cite{KMS:billiard}, 
and to count saddle connections and square-tiled surfaces \cite{EM:asymptotic_counts,DGZZ_freq}.

Thurston introduced the space $\ML_g$ of measured geodesic laminations as a useful completion of the space of weighted simple (multi-)curves on $S$ \cite{Thurston:bulletin}. The earthquake flow $\{\Eq_s\}$ on $\PM_g$ continuously extends the Fenchel-Nielson twist flows in closed curves. 
Thurston also defined stretch rays on the 
subbundle of maximal geodesic laminations \cite{Th_stretch}; 
these normalize earthquakes, and combining the two defines a Borel action of $P$ on $\PM_g$, defined almost everywhere with respect to the Lebesgue measure class, that preserves the hyperbolic length function. 
The geometry of the earthquake flow was used to resolve the Nielson Realization Problem \cite{Kerckhoff_NR}, and its dynamics are intricately related to the growth rate of simple closed geodesics on hyperbolic surfaces 
\cite{Mirz_horo, AH_compcount, Liu_horo, spine}.

While the geometry of horocycle and earthquake orbits are different \cite{MWnondiv,Fu:earthquake_excursions}, the ergodic theory of the two flows is essentially the same.
Using the ``horocyclic foliation'' construction of Thurston \cite{Th_stretch} and further developed by Bonahon \cite{Bon_SPB}, Mirzakhani defined a bijection $H$ between hyperbolic surfaces equipped with a maximal measured geodesic lamination and quadratic differentials in the principal stratum without horizontal saddle connections that exchanges the $P$ actions \cite{MirzEQ}.
Using the ``orthogeodesic foliation'' construction, the authors extended $H$ to an everywhere-defined bijection $\cO:\PM_g\to \QM_g$ swapping the earthquake flow with the horocycle flow \cite{shshI}.
The resulting $P$ action on $\PM_g$ is Borel
and extends the $P$ action coming from stretch rays and earthquakes.

Mirzakhani already observed that $H$ does not admit a continuous extension (hence that $\cO$ is not continuous): discontinuities arise along sequences in $\ML_g$ where the support of the limit in the measure topology and the limit of the supports in the Hausdorff topology disagree.
Arana-Herrera and Wright proved that there can be no continuous conjugation between the earthquake flow and horocycle flow \cite{AHW:continuous}.
Since the maps $H$ and $\cO$ are both defined geometrically using the supports of measured geodesic laminations, it is natural to wonder if they are continuous in a finer topology requiring convergence of laminations in measure and geometrically in the Hausdorff topology.  
We prove that this is indeed the case.
\begin{maintheorem}\label{maintheorem:O and Oinv continuous}
Suppose that $\lambda_n\to \lambda$ in the measure topology and their supports converge in the Hausdorff topology.
Then $(X_n,\lambda_n) \to (X,\lambda)$ in $\PM_g$ if and only if $\cO(X_n,\lambda_n) \to \cO(X,\lambda)$ in $\QM_g$.
\end{maintheorem}

Our proof is quantitative (see Theorems \ref{thm:contcell}, \ref{thm:cont_inverse}, and \ref{thm:cont_inverse_stable}).
For example, we show that for a fixed hyperbolic surface $X$, the vertical foliation of the quadratic differential $\cO(X, \lambda)$ varies H{\"o}lder continuously (in a fixed family of period coordinate charts) as the lamination $\lambda$ varies (in the Hausdorff metric on $X$).

It has long been known that there is a combinatorial-topological connection between geodesic laminations on hyperbolic surfaces and the horizontal foliations of quadratic differentials.
In the proof of Theorem \ref{maintheorem:O and Oinv continuous}, we show that this true even on a geometric level: 
there is an explicit duality between certain geometrically constructed train tracks that approximate $\lambda$ on $X$ and cellulations of $\cO(X,\lambda)$ by saddle connections whose cells interact nicely with the horizontal foliation
(see \S\ref{sec:dual cellulations}).
These techniques allow us to identify points and directions in $\QM_g$ where convergence to such a point along such a direction implies Hausdorff convergence of horizontal foliations. As a consequence, we prove that $\PM_g$ and $\QM_g$ can be decomposed into a countable union of nice Borel sets on which $\cO$ is in fact a homeomorphism (Theorem \ref{thm:Boreliso}).

In particular, for Lebesgue--almost every point $q$ of a stratum $\mathcal Q$ of quadratic differentials, 
$q$ is a point of continuity for $\cO\inverse|_{\cQ}$ (see Theorem \ref{thm:continuity for only short horizontals} for a more general statement).
Continuity almost everywhere with respect to a limiting measure is well-known to preserve weak-$*$ convergence of measures \cite[Main Theorem]{B:measures}.
Thus, Theorem \ref{maintheorem:O and Oinv continuous} has immediate dynamical consequences.
Let $\PoM_g$ and $\QoM_g$ denote the unit-length and -area loci inside $\PM_g$ and $\QM_g$, respectively. For a Borel measure $\nu$ on $\QM_g$ or $\QoM_g$ we use $\cO^*\nu$ to denote the pushforward of $\nu$ by $\cO\inverse$.\label{ind:O*}

\begin{maintheorem}\label{mainthm:measure convergence strata}
Let $\cQ^1\subset \QoM_g$ be a stratum of unit area quadratic differentials.  Suppose $\nu_n$ is a sequence of Borel probability measures on $\cQ^1$ converging weak-$*$ to a Borel probability measure $\nu$ that gives zero measure to the set of differentials with a horizontal saddle connection.  
Then $\cO^*\nu_n\to \cO^*\nu$ weak-$*$ on $\PoM_g$.
\end{maintheorem}

\begin{remark}\label{rmk:P invariant measures horizontal null}
By \cite{EM}, every $P$-invariant ergodic probability measure on $\QoM_g$ is $\SL_2\RR$-invariant and locally affine. The set of differentials with a horizontal saddle connection is null for such measures, hence Theorem \ref{mainthm:measure convergence strata} applies whenever $\nu$ is $P$-invariant.
\end{remark}

\begin{remark}
Our techniques in fact apply to a much larger class of limiting measures, including every known $U$-invariant measure on $\QoM_g$.
In such situations, one must constrain the sequence $\nu_n$ to ensure that the generic points of $\nu_n$ all have ``the same'' horizontal saddle connections as the generic point of $\nu$.
These results rely on the full power of Theorem \ref{maintheorem:O and Oinv continuous}; see \S \ref{sec:continuous a.e.} for precise statements.
\end{remark}

Most of our immediate applications (discussed in the next section and in \cite{shshIII}) import recent deep results in Teichm\"uller dynamics on strata in $\QoM_g$ to the earthquake flow and the $P$-action on $\PoM_g$.
In principle, Theorem \ref{maintheorem:O and Oinv continuous} also allows for the exchange of information in the other direction.

\begin{maintheorem}\label{mainthm:push measures from P to Q}
Suppose $\mu_n$ are Borel probability measures on $\PoM_g$ converging weak-$*$ to a measure $\mu$ that gives zero measure to the set of pairs $(X, \lambda)$ where $\lambda$ is not maximal.
Then $\cO_* \mu_n \to \cO_* \mu$ on $\QoM_g$.
\end{maintheorem}

\begin{remark}
The main example of such a limiting measure $\mu$ is the {\bf Mirzakhani measure},\label{ind:Mirzmeasure} which is a $P$-invariant measure in the class of Lebesgue that is locally induced by Thurston measure on $\ML_g$ and Weil-Petersson measure on $\M_g$.\footnote{The pushforward under the projection $\PoM_g\to \M_g$ is also often called the Mirzakhani measure.}
\end{remark}

There is a $P$-invariant {\bf Masur--Smillie--Veech probability measure} $\nu^1_{\cQ}$ in the class of Lebesgue on every component of every stratum $\mathcal Q^1$ of unit area quadratic differentials.  
The pullback of the Masur--Smillie--Veech measure of the principal stratum of quadratic differentials is the Mirzakhani measure \cite{MirzEQ}.
More generally, the $\cO^*\nu_\cQ^1$-typical point in $\PoM_g$ is a pair $(X,\lambda)$ where $\lambda$ is a minimal geodesic lamination whose support cuts $X$ into a union of regular ideal $(\kappa_i+2)$-gons, where $\sing=(\kappa_1, ..., \kappa_n)$ is the partition of $4g-4$ defining $\cQ^1$.
Using the aforementioned duality between train tracks and cellulations,
we develop ``train track coordinate charts'' for subsets of $\PoM_g$ by analogy with period coordinate charts for strata and their boundaries in $\QoM_g$. 
These tools are useful for understanding the topological structure of the pullbacks of strata in $\PoM_g$
and will be used in future work to elucidate the structure of the measures $\cO^*\nu_\cQ^1$.

Along the way towards Theorem \ref{maintheorem:O and Oinv continuous}, we establish other basic structural properties of the maps $\cO$ and $\cO\inverse$.  For example, in \S\ref{sec:thin_parts}, we show that co-bounded sets in $\PoM_g$ do not necessarily map to co-bounded sets in $\QoM_g$ (Example \ref{example:thicktothin}).
In contrast, $\cO\inverse$ does map co-bounded sets in $\QoM_g$ to co-bounded sets in $\PoM_g$ (Proposition \ref{prop:thinparts} and Corollary \ref{cor:thinparts}).
This addresses Remark 5.10 in Wright's notes on Mirzakhani's work on the earthquake flow \cite{Wright_MirzEQ}.  

\section{Applications}

We now use Theorem \ref{mainthm:measure convergence strata} to derive many new results on the ergodic theory of the earthquake flow and the structure of spaces of ergodic measures.
In forthcoming work \cite{shshIII}, we use Theorem \ref{mainthm:measure convergence strata} to address Mirzakhani's Twist Torus Conjecture \cite[Problem 13.2]{Wright_Mirz} and establish related equidistribution results.

\subsection{Ergodic theory of stretchquake disks}

Perhaps the marquee results in Teichm{\"u}ller dynamics are the Ratner-type theorems of Eskin--Mirzakhani \cite{EM} and Eskin--Mirzakhani--Mohammadi \cite{EMM}, building off work of McMullen \cite{McMgenus2} in genus 2.
These results allow for a complete description of the invariant measures, orbit closures, and the distribution of orbits of the $P$ and $\SL_2 \RR$ actions on strata of quadratic differentials, up to understanding the structure and classification of certain special subvarieties of $\QoM_g$.3

An {\bf invariant subvariety}\label{ind:AIS} $\AIS$ of a stratum of $\QM_g$ is an immersed suborbifold that is locally cut out by $\RR$-linear equations in period coordinates; 
in \cite{Filipvarieties}, Filip proved that invariant subvarieties are algebraic varieties defined over $\overline{\mathbb Q}$. 
Let $\AIS^1$ denote the intersection of $\AIS$ with the unit-area locus $\QoM_g$; this is also referred to as an invariant subvariety. 
The affine measure on period coordinates induces an (infinite) measure (up to scale) $\nu_{\AIS}$ on $\AIS$ which in turn gives rise to a probability measure $\nu_{\AIS}^1$ called the {\bf affine measure} on $\AIS^1$.

The following is a amalgam of \cite[Theorem 1.4]{EM}, \cite[Theorems 2.1 and 2.10]{EMM}, and \cite[Theorem 1.1]{Filipvarieties}; we direct the reader to the original papers for a number of further related results.

\begin{theorem}[Eskin--Mirzakhani, Eskin--Mirzakhani--Mohammadi, Filip]\label{thm:EMM}
Let $\cQ^1$ be any stratum of $\QoM_g$.
\begin{itemize} 
    \item (Measure classification): Every $P$-invariant ergodic measure on $\cQ^1$ is also $\SL_2\RR$-invariant and is the affine measure on an invariant subvariety.
    \item (Orbit closures): For any $q \in \cQ^1$, the orbit closure $\overline{P q} = \overline{ \SL_2\RR q}$ is an invariant subvariety of $\cQ^1$.
    \item (Genericity): For any $q \in \cQ^1$, the orbit $Pq$ equidistributes in its closure. That is, for any $\phi \in C_c(\cQ^1)$ and any $r \in \mathbb{R} \setminus \{0\}$,
\begin{equation}\label{eqn:QMg P generic}
\lim_{T \to \infty}
\frac{1}{T} \int_0^T \int_0^r 
\phi\left( g_t u_s (q) \right)
\, ds \, dt
=
\int_{\AIS} \phi \,d\nu_{\AIS}^1
\end{equation}
where $\nu_{\AIS}^1$ is the affine measure on $\AIS^1 = \overline{Pq}$.
\end{itemize}
\end{theorem}

The conjugacy between $\PoM_g$ and $\QoM_g$ allows one to extend the earthquake flow to a Borel measurable (and everywhere-defined) $P$ action whose orbits are locally smooth submanifolds.
The $A$ orbits of this action are called {\bf dilation rays} $\Dil_t(X)$,\label{ind:dilrays} and are defined in terms of the shear-shape coordinates introduced in \cite{shshI}.
The generic dilation ray for any $P$-invariant measure on $\PoM_g$ is a geodesic for the Lipschitz metric on Teichm{\"u}ller space (i.e., a generalized stretch line) \cite[\S\S 1.2 and 15]{shshI}.
To delineate the $P$ action on $\PoM_g$ from the $P$ action on $\QoM_g$, we call this the {\bf stretchquake} action\label{ind:stretchquakes} on $\PoM_g$.

Our construction of the Borel isomorphism $\cO$ in \cite{shshI} allowed us to pull back the measure classification part of Theorem \ref{thm:EMM}, deducing that every stretchquake-invariant ergodic measure on $\PoM_g$ is the pullback of the affine measure on some invariant subvariety of $\QoM_g$ (see Theorem C therein).
Our results on the continuity of $\cO$ now allow us to pull back the genericity part of Theorem \ref{thm:EMM}.

\begin{theorem}[Stretchquakes equidistribute]
For any $(X, \lambda) \in \PoM_g$, any $\phi \in C_c(\PoM_g)$, and any $r \in \RR \setminus \{0\}$, we have that
\[
\lim_{T \to \infty}
\frac{1}{T} \int_0^T \int_0^r 
\phi \left(\Dil_t \Eq_s (X,\lambda) \right)
\, ds \, dt
= \int_{\AIS} \phi \, d\cO^*\nu_{\AIS}^1\]
where $\nu^1_{\AIS}$ is the affine measure on $\AIS^1 = \overline{P \cdot \cO(X, \lambda)}$.
\end{theorem}
\begin{proof}
Apply $\cO^{-1}$ to \eqref{eqn:QMg P generic}; Theorem \ref{mainthm:measure convergence strata} 
and Remark \ref{rmk:P invariant measures horizontal null}
ensure that convergence of measures is preserved.
\end{proof}

Theorem \ref{maintheorem:O and Oinv continuous} does not provide enough continuity to pull back the orbit closure part of Theorem \ref{thm:EMM}.

We can also pull back Theorem 2.3 of \cite{EMM}, which states that the condition of being $P$-ergodic defines a closed condition on the space of probability measures on $\PoM_g$.
This is analogous to a result of Mozes and Shah for unipotent flows acting on homogeneous spaces \cite{MozesShah}.

\begin{theorem}\label{thm:MozesShah_stretchquake}
The space of ergodic stretchquake-invariant probability measures on $\PoM_g$ is compact in the weak-$*$ topology.
\end{theorem}
\begin{proof}
Since $\cO$ is a Borel isomorphism the induced maps $\cO_*$ and $\cO^*$ on measures preserve ergodicity, and Theorems \ref{mainthm:measure convergence strata} and \ref{mainthm:push measures from P to Q} (together with Remark \ref{rmk:P invariant measures horizontal null}) imply that they also preserve weak-$*$ convergence along sequences of stretchquake-invariant measures.
\end{proof}

\begin{remark}
One can in fact extend the stretchquake action to a Borel action of $\SL_2\RR$ and then pull back both the ``equidistribution for sectors'' and ``equidistribution for random walks'' theorems from \cite{EMM}.
Since we do not yet have a hyperbolic-geometric description of the (measurable) action of the rotation subgroup we leave the precise formulation of these results to the reader.
\end{remark}

\subsection{Ergodic theory of the earthquake flow}
In the homogeneous setting, Ratner's theorems hold not only for actions by $\SL_2\RR$ but for any subgroup generated by unipotents; in particular, they hold for any unipotent flow.
There have been major recent efforts to investigate analogues of Theorem \ref{thm:EMM} for the Teichm{\"u}ller horocycle flow.

In \cite{CSW}, it was shown that there exist fractal orbit closures for the horocycle flow.
The classification of ergodic invariant measures is a major open problem, with complete (Ratner-like) answers in only a few special cases
\cite{BSW:Ratnerhorocycle, CW:RatnerH11, EMWM:RatnerbranchedVeech}; see also \cite{CWY:weakRatner}.
In these cases, one also usually gets corresponding genericity results for the horocycle flow.
Pulling these back (using a more general version of Theorem \ref{mainthm:measure convergence strata}), we also get corresponding statements for the earthquake flow. For example, we have the following:

\begin{theorem}\label{thm:eigenform genericity}
Let $f:S_2 \to S_2$ be an involution whose quotient is a genus 1 surface with two orbifold points of order $2$.
Suppose that $(X, \lambda) \in \PoM_2$ is symmetric with respect to $f$, i.e., lies in the orbifold locus corresponding to $f$.
Then ergodic averages along earthquake flow lines through $(X, \lambda)$ always converge.
Moreover, the limiting measure lives in one of seven explicit families.
\end{theorem}
\begin{proof}
The condition on the symmetry of $(X, \lambda)$ exactly corresponds to the statement that $\cO(X, \lambda)$ lives in the eigenform locus $\mathcal E_4(1,1)$.
Horocycle-flow invariant measures on this space were classified in \cite[Theorem 1.1]{BSW:Ratnerhorocycle} and genericity of all horocycle orbits was proved in Theorem 11.1 of the same paper.
We can pull back convergence to these limiting measures because the generic point with respect to any of them have ``the same'' horizontal saddle connections (Theorem \ref{thm:pull back special convergence}).
\end{proof}

However, it was recently shown in both \cite{CSW} and \cite{CKS} that not every point is generic for the horocycle flow.
Pulling back (a slightly modified version of) the result from \cite{CKS} yields:

\begin{theorem}\label{thm:EQ nongeneric}
For every $g \ge 2$ there is a point $(X, \lambda) \in \PoM_g$ that is not generic for the earthquake flow.
\end{theorem}

A proof follows at the end of this subsection.
The non-genericity result of \cite{CKS} is in fact a corollary of their result that the Mozes--Shah phenomenon (Theorem \ref{thm:MozesShah_stretchquake}) does not hold for the horocycle flow alone.
We conclude a similar statement for earthquakes:

\begin{theorem}\label{thm:MozesShah_EQ}
The space of ergodic earthquake-invariant subprobability measures on $\PoM_g$ is not compact.
\end{theorem}
\begin{proof}
Theorem 2.5 of \cite{CKS} constructs a sequence of closed horocycle orbits in a stratum $\cQ^1$ of unit area quadratic differentials (that are squares of abelian differentials) whose corresponding uniform (ergodic) measures $\nu_n$ converge to a non-trivial convex combination $\nu$ of $P$-invariant ergodic probability measures on $\cQ^1$. 
The pullbacks $\cO^* \nu_n$ are the uniform measures supported on a sequence of closed earthquake orbits, and are in particular ergodic.
Using Theorem \ref{mainthm:measure convergence strata} and Remark \ref{rmk:P invariant measures horizontal null}, we have that $\cO^*\nu_n\to \cO^*\nu$ where $\cO^*\nu$ is a non-trivial convex combination of stretchquake-invariant ergodic probability measures: in particular, it is not ergodic for the earthquake flow.
\end{proof}

\begin{proof}[Proof of Theorem \ref{thm:EQ nongeneric}]
Let all notation be as in the proof of Theorem \ref{thm:MozesShah_EQ}. 
Corollary 1.2 of \cite{CKS} constructs a dense $G_\delta$ in $\cQ^1$ such that for any $q$ in this set, there are sequences $r_n, s_n \to \infty$ such that in the weak-$*$ topology
\begin{equation}\label{eqn:U nongeneric}
u_{[0, r_n]} q := \frac{1}{r_n}\int_0^{r_n} \delta_{u_s q} \, ds \to \nu
\hspace{3ex}
\text{ and }
\hspace{3ex}
u_{[0, s_n]} q := \frac{1}{s_n}\int_0^{s_n} \delta_{u_s q} \, ds \to \nu_{\cQ}^1
\end{equation}
where $\nu_{\cQ}^1$ is the Masur--Smillie--Veech measure on $\cQ^1$ and $\nu$ is some other measure.
With a very slight modification to their proof,
one can also ensure that $\nu$ is the convex combination of ergodic measures from the proof of Theorem \ref{thm:MozesShah_EQ}.\footnote{As indicated to the authors \cite{Osamachat}, one can even find $q$ such that $\nu$ is ergodic for the $P$ action and singular to $\nu_{\cQ}^1$.}
We outline the argument below.

Let $\{r_n\}$ denote the periods of the closed horocycles limiting to $\nu$ and consider the sets
\[V_k:= 
\{ q \, |\,  d(u_{[0, r_n]}q, \nu) < 1/k \text{ for some } r_n > k\}
\]
where $d$ is any metric on the space of probability measures on $\cQ^1$ inducing the weak-$*$ topology. 
These sets are defined by inequalities, hence are open, while they also contain arbitrary long horocycles that limit to $\nu$, hence each $V_k$ is dense since $\nu$ has full support. Taking the intersection of the $V_{k}$'s yields a dense $G_\delta$ where any point has some subsequence of times $r_{n_k} \to \infty$ along which $u_{[0, r_{n_k}]} q \to \nu$.
On the other hand, by the ergodicity of the horocycle flow with respect to Masur--Smillie--Veech measure $\nu_{\cQ}^1$ there is another dense $G_\delta$ whose ergodic averages along a different sequence of times $\{s_n\}$ converges to $\nu_{\cQ}^1$; taking any point in the intersection of these dense $G_\delta$'s gives a $q$ satisfying \eqref{eqn:U nongeneric} for $\nu$ as desired.

We may now pull back \eqref{eqn:U nongeneric} to obtain points in $\PoM_2$ that are not generic for the earthquake flow.
To get arbitrary genus, one needs only take covers of the appropriate degree. 
\end{proof}

\subsection{Expanding horospheres}
Theorems \ref{mainthm:measure convergence strata} and \ref{mainthm:push measures from P to Q} also provide a new perspective on the equidistribution of certain interesting sets inside of $\PoM_g$ and $\QoM_g$.

In \cite{Mirz_horo}, Mirzakhani established the equidistribution of level sets for the hyperbolic length of a simple multicurve in $\PoM_g$;
these are the analogues of expanding horospheres based at a cusp in a cusped hyperbolic manifold.
In \cite{AH_horo} and \cite{Liu_horo}, Arana--Herrera and Liu independently proved similar results for the level sets of a broader class of functions of the lengths of individual components of a multicurve, and in \cite[\S5]{spine}, Arana--Herrera and the first author proved a similar result given constraints on the geometry of the complementary subsurface to the multicurve.
Via averaging and unfolding arguments in the style of Margulis \cite{Margulisthesis},
these equidistribution results can be turned into counting results for multicurves
\cite{AH_compcount, Liu_horo, spine}.
The proofs of all three of the above-mentioned ``expanding horosphere'' results rely on the ergodicity of the earthquake flow and delicate absolute continuity arguments.

On the other hand, there are also equidistribution results for the level sets of extremal length
(and more generally, pushes of pieces of the unstable foliation of $\QoM_g$ under the Teichm{\"u}ller geodesic flow), which play the role of expanding horospheres in $\QoM_g$.
See \cite[Theorem 1.6]{Forni} and \cite[Proposition 3.2]{EMM:effscc}, as well as \cite{QDspine}.
The map $\cO$ is defined leafwise on the unstable leaves to take hyperbolic length to extremal length, and so in particular we see that it takes horospheres to horospheres.
Theorems \ref{mainthm:measure convergence strata} and \ref{mainthm:push measures from P to Q} now show that the equidistribution results 
of \cite{Mirz_horo, AH_horo, Liu_horo, spine} on $\PoM_g$ are {\em equivalent} to those of \cite{Forni, EMM:effscc, QDspine} on the principal stratum of $\QoM_g$.

While expanding horospheres in $\PoM_g$ have only been studied for multicurves, the equidistribution results in $\QoM_g$ hold for any leaf of the unstable foliation.
Using Theorem \ref{mainthm:measure convergence strata}, we can deduce similar results in $\PoM_g$.
To state this result, let us first fix some notation.

For any $\lambda \in \ML$, let $H_\lambda$ denote the 1--level set of its length function.
Following \cite{Forni}, say that a measure $\mu$ on $H_\lambda$ is {\bf horospherical}\label{ind:horosphere} if it is in the Lebesgue class with continuous density and such that almost all of the conditional measures on earthquake flow lines are just restrictions of Lebesgue.
Such a measure can be obtained, for example, by picking some nice set in $H_\lambda$ using shear coordinates for a completion of $\lambda$, then restricting the natural affine measure coming from shear coordinates to that set.

Given a horospherical measure $\mu$ on $H_\lambda$, one can lift it to live on the section 
$H_\lambda \times \{\lambda\} \subset \PoT_g$
and then take the pushforward of $\mu$ under the covering map to get a measure $\widehat{\mu}$ on $\PoM_g$.

\begin{theorem}
Let $\lambda \in \ML$ and let $\mu$ be any horospherical measure on $H_\lambda$.
Then $(\Dil_t)_* \widehat{\mu}$ converges to Mirzakhani measure on $\PoM_g$ as $t \to \infty$.
\end{theorem}
\begin{proof}
The corresponding statement in $\QoM_g$ is Theorem 1.6 of \cite{Forni}.
We note that the map $\cO$ restricts to a (piecewise) real-analytic map on leaves of the unstable foliation \cite{shshI} that takes $\Eq_s$ to $u_s$ in a time-preserving way. Thus, it takes horospherical measures on $\PoM_g$ to horospherical measures on $\QoM_g$.
We can then pull back convergence using Theorem \ref{mainthm:measure convergence strata}.
\end{proof}

Work of Lindenstrauss and Mirzakhani \cite{ML:measures} implies that whenever $\lambda$ is arational (i.e., contains no simple closed curves), the projection of the level set $H_\lambda$ is dense in $\M_g$. 
This is an analogue of a result of Dani that says that any horospere in a cusped hyperbolic manifold is either closed or dense \cite{Dani:horosphere}.
In the homogeneous setting, horospheres arise as the orbits of a horospherical subgroup, and in fact every non-closed orbit equidistributes.
In $\PoM_g$ or $\QoM_g$ there is no notion of a horospherical subgroup (and these spaces are completely inhomogeneous), but we expect that one should be able to use the techniques of this paper to prove an analogous result.

The results of \cite{Forni} and \cite{EMM:effscc} hold for non-principal strata as well; pulling these back, we can deduce equidistribution results for specific slices of expanding horospheres; the limiting measures will now be stretchquake-flow invariant ergodic measures which are singular to Mirzakhani measure.
We remark more on this in a sequel paper, in which we also apply recent deep results in Teichm{\"u}ller dynamics to address Mirzakhani's twist torus conjecture \cite{shshIII}.

\section{Outline of the paper}

In Part \ref{part:prelim} of the paper, we gather preliminaries and prove some elementary estimates in hyperbolic geometry.
In \S\ref{sec:background}, we recall some basics about geodesic laminations, the orthogeodesic foliation, and shear-shape cocycles from \cite{shshI}.
In \S\ref{sec:thin_parts}, we turn our attention to establishing two basic and useful properties of $\cO$: that $\cO$ need not map bounded sets to bounded sets (Example \ref{example:thicktothin}), but $\cO\inverse$ does (Proposition \ref{prop:thinparts} and Corollary \ref{cor:thinparts}).
In \S\ref{sec:triplecenters} we consider tuples of pairwise disjoint complete geodesics in $\HH^2$, none of which separates the others, and study the geometry of the circles inscribed in this configuration and how they change under perturbation. 
These elementary estimates are used extensively throughout Parts \ref{part:continuity} and \ref{part:inverse continuity}.
\vspace{1ex}

In Part \ref{part:tts and cellulations}, we establish a duality between certain train tracks defined via hyperbolic geometry and certain cellulations by saddle connections on singular flat surfaces. 
We thus can use train tracks to coordinatize the ``strata'' of both $\PT_g$ and $\QT_g$ and analyze how they fit together.
In \S\ref{sec:geometric tts}, 
we build an important class of geometric train track neighborhoods called \emph{equilateral train track neighborhoods}.
Then, in \S\ref{sec:dual cellulations}, we show that dual to a (filling) equilateral train track $\tau$ constructed from $(X,\lambda)\in \PT_g$, there is a dual cellulation $\mathsf T$ of $\cO(X,\lambda)$ by \emph{veering} saddle connections.  When $\tau$ is not filling, we augment it by adjoining the (visible) geometric filling arc system to obtain a filling train track $\taua$ and dual cellulation. 
In the final \S\ref{sec:ttperiod coords} of this part, we explain how to view period coordinate charts for strata of quadratic differentials adapted to a cellulation by saddle connections as weights on the dual augmented train track.
Throughout the paper, horizontal saddle connections play a special role; given a component $\cQ$ of a stratum of quadratic differentials, we show how to coordinatize $\cQ^*$, which consists of the differentials in strata adjoining $\cQ$ that are obtained by opening up higher order zeros \emph{horizontally}.
Later on in \S\ref{sec:continuous a.e.}, we obtain good continuity properties of $\cO\inverse$ restricted to $\cQ^*$ (Theorem \ref{thm:continuity for only short horizontals}), hence a more general version of Theorem \ref{mainthm:measure convergence strata}.
\vspace{1ex}

In the next Part \ref{part:continuity}, we establish quantitative continuity for $\cO$ with respect to the Hausdorff + measure topology on $\ML_g$.  
In order to compare shear-shape cocycles on different measured laminations (more generally, chain recurrent laminations), we have to understand how the geometric arc system $\arc(X,\lambda)$ and the weights that describe the geometry of $X\setminus \lambda$ vary as $X$ and $\lambda$ vary.  
In \S\ref{sec:persistent}, we establish that visible arcs ``persist'' and control how the weights vary as the metric $X$ and $\lambda$ vary in a small neighborhood (Proposition \ref{prop:persistent}).
This essentially establishes continuity of the ``shape'' part of the maps $\sigl$ recording the shear-shape cocycle of the metric $X$ with respect to $\lambda$.
In the next \S\ref{sec:ties}, we study the structure of equilateral train track neighborhoods and show that ties of equilateral neighborhoods associated to $\lambda$ and $\lambda'$ on a hyperbolic surface $X$ are Hausdorff close when $\lambda$ and $\lambda'$ are Hausdorff close (Proposition \ref{prop:tiestable}).  
This analysis is the main ingredient for continuity of ``shear'' part of the maps $\{\sigl\}_\lambda$.
In the last \S\ref{sec:continuity proof}, we assemble the ingredients to prove the main Theorem \ref{thm:contcell}, establishing one direction of Theorem \ref{maintheorem:O and Oinv continuous}.
To accomplish this, we show that certain augmentations of equilateral train tracks are stable as one varies the metric and lamination (Proposition \ref{prop:ttstable} and Corollary \ref{cor:augstable}); this implies that the corresponding flat surfaces have comparable cellulations by saddle connections.
This is a crucial step in our proof that $\cO$ is continuous in that it provides a common container for the shear-shape cocycles for geometrically similar but topologically different geodesic laminations.  
\vspace{1ex}

Part \ref{part:inverse continuity} deals with continuity of $\cO\inverse$ along convergent sequences $q_n \to q$ where the horizontal measured laminations also converge in the Hausdorff topology.
In \S\ref{sec:shapeshift_est} we establish detailed, uniform estimates on the sizes of the ``shape-shifting'' deformation cocycles introduced in \cite{shshI}.
The main result of this section is Theorem \ref{thm:shapeshift_distance_small}. The estimates that go into its proof are fairly involved; fortunately, they are also independent from the rest of the paper and can safely be ignored on first reading.
In the next \S\ref{sec:inverse continuity}, Theorem \ref{thm:cont_inverse} establishes the remaining direction of the main Theorem \ref{maintheorem:O and Oinv continuous}.
Using the main result of \cite{shshI}, this essentially amounts to proving a quantified continuity result for $\cO\inverse$ along leaves of the \emph{stable} foliation consisting of differentials whose real foliations coincide, after enforcing the condition that the imaginary foliations converge Hausdorff as well as in the measure topology.  
Although the main point of the proof appeals to the estimates on shape-shifting cocycles from the previous section, it is surprisingly technical to get to a point that we can compare shear-shape cocycles coming from periods of differentials with different imaginary foliations.  This is where our work in the previous sections establishing the (geometric) duality between train tracks and cellulations by saddle connections pays off.
We point out also Proposition \ref{prop:stable_hausdorff_stratum}, which asserts that differentials in the same leaf of the stable foliation with Hausdorff-close imaginary laminations are actually in the same stratum component.  
\vspace{1ex}

Finally, in Part \ref{part:measures}, we conclude our main results on convergence of measures.
In \S\ref{sec:continuous a.e.}, 
we establish Theorem \ref{thm:continuity for only short horizontals}, which generalizes Theorem \ref{mainthm:measure convergence strata}.
We do this by identifying directions in $\ML_g$ where measure convergence implies Hausdorff convergence of supports (Lemma \ref{lem:full_convergence}); the result then follows from Theorem \ref{thm:cont_inverse} and a well-known result from probability theory.
Along the way, we make explicit the fact that $\cO$ induces an isomorphism between the Borel $\sigma$-algebras on $\PT_g$ and $\QT_g$.
This part also contains a much more general result than Theorem \ref{mainthm:measure convergence strata} that uses the full strength of the continuity theorems we prove in the paper.

\subsection*{Acknowledgements}
The authors would like to extend a special thanks to Francisco Arana-Herrera for a very useful discussion at an early stage of this project, and are also grateful to Osama Khalil, Yair Minsky, and Alex Wright for enlightening conversations and helpful comments.

\tableofcontents

\vfill
\pagebreak

\part{Preliminaries}\label{part:prelim}

\section{Background}\label{sec:background}

\begin{remark}\label{rmk: O notation}
Throughout the paper, for real valued functions $g$ and $f$, we say that $h = O(f)$ if there is a constant $C$ such that $h(x) \le C f(x)$ for all $x$ in a suitable set.  We usually try to make this suitable set explicit, e.g., that $x$ is required to be positive and smaller than some specific threshold.
The constant $C$ is referred to as the ``implicit constant.''
\end{remark}

\subsection{Geodesic laminations}
Recall that a geodesic lamination $\lambda$ on a closed hyperbolic surface $X\in \T_g$ is a closed set foliated by simple complete geodesics called its leaves.  
We denote by $\mathcal {GL}(X)$ the space of all geodesic laminations contained in $X$ equipped with the Hausdorff metric $d_X^H$\label{ind:Hausdmetric} on closed subsets of $X$.  
It is well known that $\mathcal {GL}(X)$ is closed in the Hausdorff topology, hence is itself a compact metric space.

We are mostly interested in \textbf{chain recurrent}\label{ind:chainrec} geodesic laminations, which are those that can be approximated by geodesic multi-curves; let $\mathcal {GL}^{cr}(X)$ denote the (closed) subspace of chain recurrent geodesic laminations (see \cite[\S\S6 and 8]{Th_stretch}).
We collect here a list of useful, well-known facts about geodesic laminations and their geometry.

\begin{lemma}\label{lem:lamfacts}
Let $X$ and $X'$ be any hyperbolic structures on $S$.  Then the following are true:
\begin{enumerate}
    \item \label{item:gl_topological} There is a canonical homeomorphism between the spaces of geodesic laminations on $X$ and $X'$: thus geodesic laminations are intrinsic topological objects, rather than strictly metric. 
    \item\label{item:Haus_equiv}  The Hausdorff metrics $d_X^H$ and $d_{X'}^H$ on geodesic laminations with respect to $X$ and $X'$ are H\"older equivalent \cite{BonZhu:HD}. 
    More precisely, for any compact $K \subset \T_g$ there are constants $M$ and $\alpha$ depending on $K$ such that \[\frac1M\left( d_X^H\right)^{1/\alpha} \le d_{X'}^H\le M\left( d_X^H\right)^\alpha\]
    for any $X, X' \in K$.
    \item \label{item:measurezero} For any geodesic lamination $\lambda$ on $X$ and any transverse arc $k$, the $1$-dimensional Lebesgue measure of $k\cap \lambda$ is zero \cite{BS}.   
\end{enumerate}
\end{lemma}

A transverse measure $\mu$ on a geodesic lamination $\lambda$ assigns a finite Borel measure $\mu_k$ to every arc $k$ transverse to $\lambda$ satisfying that the support of $\mu_k$ is \emph{equal to} $k\cap \lambda$.  In addition, this assignment is required to be invariant under holonomy.
Often, we abuse notation and use the symbol $\lambda$ for a geodesic lamination equipped with a transverse measure.
Likewise, for a transverse arc $k$, we let $\lambda(k)\ge 0$ denote the total mass of the measure deposited on $k$ from $\lambda$.

We denote by $\ML_g$\label{ind:ML} the space of measured geodesic laminations (with the topology of measure convergence).
There is a continuous, locally bi-linear \emph{geometric intersection form}
\[i: \ML_g \times \ML_g \to \RR_{\ge 0}\]
extending the minimal geometric intersection number between simple multi-curves.

Since any measurable lamination is chain recurrent, there is a set-theoretic map $\ML_g \to \mathcal {GL}^{cr}(X)$ remembering only the support of a measured lamination $\lambda$.
While this map is not continuous, almost every point with respect to the Lebesgue measure on $\ML_g$ is a point of continuity; see Section \ref{subsec:measure vs hausdorff} for details.

We denote by $\PT_g$\label{ind:PMg} the space $\T_g\times \ML_g$, which can (and will) be thought of either as a bundle over $\T_g$ or over $\ML_g$.  The mapping class group $\Mod_g$ acts nicely with quotient $\PM_g := \PT_g/\Mod_g$.  Thurston introduced a continuous, homogeneous length function $\PT_g \to \RR_{>0}$ that extends the total hyperbolic length of simple multi-curves \cite{Thurston:bulletin}.  Let $\PoT_g$ and $\PoM_g$ denote the unit length loci.  

\subsection{Quadratic differentials}\label{subsec:QD bckgrd}
A quadratic differential $q$ on a Riemann surface $X$ is a section of the symmetric square of its cotangent bundle. Away from its zeros, a holomorphic quadratic differential $q$ locally has the form $dz^2$, and the expression $|dz|^2$ defines a flat cone metric on $X$ whose cone points correspond to the zeros of $q$.
We denote by $\QT_g$ the bundle of holomorphic quadratic differentials over $\T_g$ minus the zero section and $\QM_g = \QT_g/\Mod_g$ the moduli space of such.\label{ind:QMg}
We let $\QoT_g$ and $\QoM_g$ be the unit area loci of holomorphic quadratic differentials, i.e., the area of the singular flat metric on $S$ induced by such a differential is $1$.

All of these spaces are naturally broken up into {\bf strata}, sub-orbifolds that parametrize those $q$ with fixed number and order of zeros.\label{ind:strata}
Strata are not necessarily connected, but their components have been classified
\cite{KZstrata, Lanneau, CMexceptional}.
Apart from a few sporadic examples in genus four,
the components of a stratum are always classified by whether or not every surface in them is hyperelliptic, whether or not they parametrize squares of abelian differentials, and by the Arf invariant of an associated spin structure (when one exists).
Throughout the paper, we will use $\cQ$ to refer to a component of a stratum and $\cQ^1$ for its unit-area locus.

Strata can be given local {\bf period coordinates} as follows.\label{ind:period coordinates}
Every quadratic differential comes with an {\em orientation cover} $\widehat q \to q$ that is a double cover of $q$, branched over the zeros of odd order.
This object is naturally an {\em abelian} differential corresponding to the square root of $q$.
One can measure the periods of all relative cycles in order to associate to $\widehat{q}$ a map 
\[H_1(\widehat q, Z(\widehat q); \ZZ) \to \mathbb{C}\]
that is naturally anti-invariant under the covering involution of $\widehat{q}$.
Performing this construction in families therefore locally models a (marked) stratum component $\cQ$ on $H^1(\widehat q; Z(\widehat q), \CC)^-$, the $-1$-eigenspace for the covering involution.
The natural affine measure on this subspace pulls back to a measure called the {\bf Masur--Smillie--Veech measure} $\nu_\cQ$ on $\cQ$\label{ind:MSV}
which can be used to induce a probability measure $\nu_\cQ^1$ on $\cQ^1$ that is ergodic for the $\SL_2(\RR)$ action \cite{Masur_IETsMF, Veech_IETs}.

\subsection{Measured foliations}\label{subsec:MF backgrd}
A (singular) measured foliation on $S$ is a $C^1$ foliation $\mathcal F$ of $S\setminus Z$, where $Z$ is a finite set (the singular points), equipped with a transverse measure $\nu$ on arcs transverse to $\mathcal F$.  The transverse measure is required to be invariant under holonomy and every singularity is modeled on a standard $k$-pronged singularity.
Isotopic measured foliations are viewed as identical.
The space $\MF_g$\label{ind:MF} of Whitehead equivalence classes of singular measured foliations is equipped with its measure topology making the geometric intersection pairing $i:\MF_g \times \MF_g \to \RR_{\ge0}$ given by integrating the product of transverse measures continuous; see  \cite{Thurston:bulletin, FLP} for details.

By straightening the non-singular leaves of a measured foliation with respect to a choice of hyperbolic metric and taking the closure, one can obtain a geodesic lamination together with a (pushforward) transverse measure.
This describes a natural map $\MF_g \to \ML_g$ which is in fact a homeomorphism commuting with geometric intersection \cite{Levitt}.  As such, we often conflate measured foliations and measured laminations.

Important examples of singular measured foliations are the directional foliations coming from quadratic differentials. Building on work of Hubbard and Masur, Gardiner and Masur proved that a holomorphic quadratic differential $q$ is determined uniquely by the Whitehead equivalence classes of its imaginary and real foliations, whose leaves are horizontal and vertical, respectively.
The imaginary and real foliations of $q$ {\bf bind} the surface, a condition that is equivalent to the fact that the corresponding measured geodesic laminations cut $S$ into compact disks.\label{ind:bind}
Moreover, every pair of measured foliations which bind $S$ are realized as the vertical and horizontal foliations of some $q\in \QT_g$, and this assignment yields a homeomorphism
\begin{equation}\label{eqn:GM QT is MFMF}
\QT_g \cong \MF_g \times \MF_g \setminus \Delta
\end{equation}
where $\Delta$ denotes the set of pairs of measured foliations that do not together bind the surface 
\cite{GM, HubMas}.

This product structure gives rise to two natural foliations $\Fol^{uu}$ and $\Fol^{ss}$\label{ind:unstablefols} of $\QT_g$: given $q \in \QT_g$, let $\Fol^{uu}(q)$ denote those $q'$ with the same (Whitehead equivalence class of) imaginary measured foliation, while $\Fol^{ss}(q)$ denotes those $q'$ with the same real foliation.
In period coordinates, these foliations simply look like changing the real and imaginary parts of periods, respectively.
This identification also gives a hint as to the meaning of the superscripts: deformations in $\Fol^{uu}$ come from either scaling the measure or are exponentially expanded under the Teichm{\"u}ller geodesic flow, while deformations in $\Fol^{ss}$ that do not come from scaling are exponentially contracted.

At a global level, \eqref{eqn:GM QT is MFMF} implies that we can identify each leaf of $\Fol^{uu}$ and each leaf of $\Fol^{ss}$ with open subsets in $\MF_g$. However, the intersection of these leaves with a stratum $\cQ$ is {\em not} described just in terms of the topological type of the foliation (in part due to complications arising from taking Whitehead equivalence classes).
A global description of $\Fol^{uu} \cap \cQ$ is a consequence of the main theorem of our previous paper (see \cite[Corollary 2.6]{shshI}).

\subsection{The orthogeodesic foliation}\label{subsec:orthofoliation}
Let $X\in \T_g$ and let $\lambda$ be a geodesic lamination.
The complement of $\lambda$ in $X$ is a (possibly disconnected) hyperbolic surface whose metric completion has totally geodesic (possibly non-compact) boundary.  For each such complementary component $Y$, there is a piecewise geodesic $1$-complex $\Sp$ called the \textbf{spine}\label{ind:spine} consisting of points that are closest to two or more components of $\partial Y$ in its universal cover.  Away from $\Sp$, there is a unique nearest point in $\partial Y$.

The fibers of the projection map $Y\setminus \Sp \to \partial Y$ form a foliation of $Y\setminus \Sp$ whose leaves extend continuously across $\Sp$ to a piecewise geodesic singular foliation $\cO_{\partial Y}(Y)$ called the \textbf{orthogeodesic foliation}.
This singular foliation has $n$-pronged singularities at the vertices of the spine of valence $n$, and
every endpoint of every leaf of $\cO_{\partial Y}(Y)$ meets $\partial Y$ orthogonally.\label{ind:orthofol}
    
The orthogeodesic foliation $\cO_{\partial Y}(Y)$ is equipped with a transverse measure: the measure assigned to  a small enough transversal is Lebesgue after projection to $\partial Y$, and 
this assignment is invariant by isotopy transverse to $\cO_{\partial Y}(Y)$. 
This construction produces a (non-smooth) singular foliation on  $X\setminus \lambda$, which extends continuously across the leaves of $\lambda$ and defines a singular measured foliation $\cO_{\lambda}(X)$ on $S$
(after smoothing out the leaves in a neighborhood of $\Sp$); see \cite[Section 5]{shshI} for more details.
An essential feature of $\Ol(X)$ is that for a geodesic segment of a leaf of $\lambda$, the length of that segment is computed by the transverse measure to $\Ol(X)$, and this is preserved under isotopy transverse to $\Ol(X)$.

We have therefore produced a map\label{ind:orthofolmap}
\[\cO_\lambda: \T_g \to \MF_g\]
recording the measure equivalence class of $\cO_\lambda(X)$.
Suppose now that $\lambda$ admits a transverse measure of full support, and define $\MF(\lambda)$ to be the subset of measured foliations that bind together with $\lambda$.\label{ind:MFlambda}
In \cite{shshI}, the authors proved the following:
\begin{theorem}\label{thm:ortho homeo}
    For $\lambda\in \ML_g$, the map $\cO_\lambda: \T_g \to \MF_g(\lambda)$ is a homeomorphism, is equivariant with respect to the stabilizer of the support of $\lambda$ in $\Mod_g$, and satisfies $i(\cO_\lambda(X), \lambda) = \ell_X(\lambda)$.
\end{theorem}

Combining Theorem \ref{thm:ortho homeo} with the Theorem of Gardiner and Masur, 
we obtain a $\Mod_g$-equivariant bijection\label{ind:O}
\[\cO: \PT_g \to \QT_g\]
mapping the unit length locus to the unit area locus,
such that the imaginary and real foliations of $\cO(X,\lambda)$ are measure equivalent to $\lambda$ and $\cO_\lambda(X)$, respectively.  
One can also give an explicit construction of $\cO(X,\lambda)$ by ``inflating'' $\lambda$ and ``deflating'' the complementary subsurfaces to $\lambda$; this viewpoint makes it clear that the real foliation of $\cO(X,\lambda)$ is \emph{isotopic} to $\cO_\lambda(X)$, not just Whitehead-equivalent \cite[Proposition 5.10]{shshI}.

For use in the sequel, we record here an elementary geometric estimate that says that leaves of $\Ol(X)$ are nearly $C^1$ in the spikes of $\lambda$. It can be proven via elementary Euclidean geometry in the plane.

\begin{fact}\label{fact:kink angle}
    Let $Y$ be a hyperbolic surface with non-compact boundary and let $\ell$ be a leaf of $\cO_{\partial Y}(Y)$ of length less than $l\le \log(3)$.  Then the angle between the two geodesic segments that comprise $\ell$ is bounded below by $2\pi/3$ and tends to $\pi$ linearly with $l$, i.e., the angle is $\pi-O(l)$.  
    
    Additionally, even without a bound on $l$ the angle is bounded below by a constant that depends only on the area of $Y$.
\end{fact}

\subsection{Arc systems}\label{subsec:arc systems}
From now on we tacitly take the metric completion of $X\setminus \lambda$, so each component is a finite area hyperbolic surface with totally geodesic boundary.  
For each component $Y$ of $X\setminus \lambda$ and for each compact edge $e$ of $\Sp$, there is a dual arc $\alpha$ properly isotopic to any non-singular leaf of $\cO_{\partial Y}(Y)$ meeting $e$.
It is not difficult to see that the union $\arc(X,\lambda)$ of all arcs obtained this way cuts $X\setminus \lambda$ into disks, i.e., is a filling arc system on $X\setminus \lambda$.
To each arc $\alpha \in \arc(X,\lambda)$ we record a {\bf weight} $c_\alpha := \cO_{\partial Y}(Y)(e)$, where $e\subset \Sp$ and $\alpha$ are dual.  That is, $c_\alpha$ is the length of the projection of $e$ to $\lambda$.\label{ind:arcsystems}

We remark that if the weight of $\arc$ is small, its shortest (orthogeodesic) representative may not necessarily be a leaf of $\Ol(X)$. This is related to the observation in \cite{Luo} that ``radius invariants'' can be negative even when ``radius coordinates'' are positive.

It turns out that the formal sum
\[\arcwt(X,\lambda) := \sum_{\alpha \in \arc(X,\lambda)} c_\alpha \alpha\]
completely determines the geometry of $X\setminus \lambda$ in the following sense.

The {\bf arc complex} $\mathscr{A}(Y)$ of a surface with boundary $Y$ is the simplicial flag complex whose vertices are essential proper isotopy classes of embedded arcs, and where two vertices are joined if they have disjoint representatives. 
The {\bf weighted filling arc complex} of $Y$ is either the cone 
\[\mathscr{A}(Y) \times \mathbb{R}_{\ge 0} / (\mathscr{A}(Y) \times \{0\})\]
when $Y$ is simply connected, or an $\mathbb{R}_{>0}$-bundle over the subspace of the arc complex corresponding to filling arc systems, when $Y$ is not.
The weighted filling arc complex $|\Arcfill(S\setminus \lambda)|_\RR$\label{ind:arccx} 
is the product over all complementary subsurfaces to $\lambda$.
Generalizing work of \cite{Do,Luo, Ushijima} for when $\lambda$ is a multi-curve, in \cite[\S6]{shshI} we proved that the map \[\arcwt : \T(S\setminus \lambda) \to |\Arcfill(S\setminus\lambda)|_\RR\]
is a piecewise analytic mapping class group equivariant homeomorphism.

This induces a map $\T_g \to |\Arcfill(S\setminus\lambda)|_\RR$ obtained by cutting $X\in \T$ along the geodesic realization of $\lambda$ and recording the geometric weighted filling arc complex; 
abusing notation,  this map is also denoted by $\arcwt$.  In general, this map is not surjective, but rather maps onto a subspace $\mathscr B(S\setminus \lambda)\subset |\Arcfill(S\setminus\lambda)|_\RR$\label{ind:base} cut out by equations that are linear on each cell (see \S7.3 of \cite{shshI}). 

\begin{example}[Arc complex of a $4$-gon]\label{example:arc complex 4-gon}
    Let $Y$ be an ideal $4$-gon.  There are exactly three filling arc systems: the empty arc system and the two arcs that join opposite edges of $Y$.  So, $|\Arcfill(Y)|_\RR$ is a union of two rays $\RR_{\ge 0 }$ glued along their endpoints and is therefore homeomorphic to $\RR$.  The special point $0$ in both rays corresponds to the empty arc system, or geometrically, the ideal $4$-gon with an order $4$ cyclic symmetry.  

    Similarly, if $Y$ is an (infinite area) region in $\HH^2$ with smooth geodesic boundary consisting of $4$ complete geodesics, we can still record the weighted arc system dual to the compact arc of $\Sp$, and the weighted filling arc complex is still homeomorphic to $\RR$.
\end{example}

\subsection{Hexagons, centers, and basepoints}\label{subsec:hexagons center and basepoints}
Since the arc system $\arc = \arc(X,\lambda)$ is filling, the complement of $\lambda\cup \arc$ in $X$ consists of topological disks.  Generically, i.e., if every vertex of $\Sp$ is trivalent, then when $\arc$ is realized orthogeodesically each component $H$ of the metric completion of  $X \setminus (\lambda\cup \arc)$ is a partially ideal (degenerate) right angled hexagon where we consider spikes of $H$ as sides of zero length.
More generally, if $H$ contains a vertex of $\Sp$ of valence $n$, then $H$ is a partially ideal  (degenerate) right angled $2n$-gon. It is a consequence of Gauss--Bonnet that there is a uniform bound (depending only on genus) on the total valence of the vertices of $\Sp$.

Abusing terminology, we will often refer to \emph{any} component of $X\setminus (\lambda \cup \arc)$ a as {\bf hexagon}\label{ind:hexagon}, regardless of its number of sides.
Every vertex $v$ of $\Sp$ corresponds to a unique hexagon $H_v$, and $v$ is equidistant to each of the geodesic boundary components of $H_v$ corresponding to leaves of $\lambda$.
In other words, $v$ is the center of a hyperbolic circle contained in $X\setminus \lambda$ tangent to $\lambda$; these points of tangency are also the first intersection points with $\lambda$ of the $n$-pronged singularity of $\cO_\lambda(X)$ emanating from $v$.
Note that this incircle is not necessarily contained in $H_v$.

More generally, to any triple of pairwise disjoint complete geodesics in $\HH^2$, none of which separates the others, there is a unique inscribed circle.
The {\bf center}\label{ind:center} of the triple is the (hyperbolic) center of this circle; equivalently, it is the unique point which is equidistant from all three geodesics.
The {\bf basepoints}\label{ind:basepts} of the triple are the points of tangency of this circle with the geodesics. Equivalently, they are the closest-point projections of the center of the triple to the constituent geodesics.

The following observation gives a universal lower bound on distance from any center to its corresponding geodesics, i.e., the radius of the incircle.

\begin{lemma}\label{lem:closethenarc}
Let $Y$ be a hyperbolic surface with totally geodesic boundary.
If two geodesics of $\partial Y$ are distance less than $\log\sqrt3$ of each other then they are connected by an arc of $\cO_{\partial Y}(Y)$.
\end{lemma}
\begin{proof}
Let $g$, $h$, and $k$ be distinct, pairwise disjoint complete geodesics in $\HH^2$, none of which separates the others.  Then $\log\sqrt 3$ is a lower bound on the distance from any of $g$, $h$, and $k$ to their center.

In the universal cover $\widetilde{Y}$, suppose $g$ and $h$ are components of $\partial \widetilde {Y}$ that are distance at most $\log\sqrt3$ from each other.
If $g$ and $h$ are asymptotic, then the conclusion is obvious.
Otherwise, consider the unique orthogeodesic $\alpha$ running from $g$ to $h$, let $r$ be its length, and let $a$ be its midpoint.
By assumption the circle of radius $r$ centered at $a$ is tangent to both $g$ and $h$.
Moreover, it can meet no other geodesics of $\partial \widetilde{Y}$ since there is no configuration of three pairwise disjoint geodesics, none of which separates the other two, and which are all at distance less than $\log\sqrt3$ from each other.

Since leaves of the orthogeodesic foliation are fibers of the closest point projection map, this implies that $\alpha$ is (the lift of) a leaf of $\cO_{\partial Y}(Y)$.
\end{proof}

If, in addition, we have an lower bound on the injectivity radius of the hyperbolic structure $X$, we can give an upper bound on this radius.

\begin{lemma}\label{lem:radius_bounded}
For any $s>0$ there is a bound $r= r(s, g)$ such that the following holds.
Let $X$ be a hyperbolic metric on a genus $g$ surface with systole bounded below by $s$ and let  $\lambda $ be a geodesic lamination on $X$; then the function $p\mapsto d(p,\lambda)$ on $X$ is bounded above by $r$.
\end{lemma}

\begin{proof}
To bound the radius of any circle inscribed in an equidistant configuration, we note that the function that sends triples $(X, x, \lambda)$ where $x \in X$ to the distance from $x$ to $\lambda$ is continuous as $\lambda$ varies in the Hausdorff topology.
The space of all geodesic laminations $\mathcal{GL}$ is compact, so the (universal curve over the) moduli space of pairs $(X, \lambda) \subset \T_g \times \mathcal{GL}$ where $X$ is $s$-thick is also compact.
\end{proof}

Lemma \ref{lem:radius_bounded} necessarily requires a bound on the thickness of $X$. For example, consider an ideal triangulation of a once-punctured surface. Opening the cusp into a boundary component, this determines a maximal arc system on the bordered surface.
The centers of each triangle all have definite injectivity radius, but the boundary is far away from the thick part of the surface. See Figure \ref{fig:nounifribbd}.

\begin{figure}[ht]
\includegraphics[width=\linewidth]{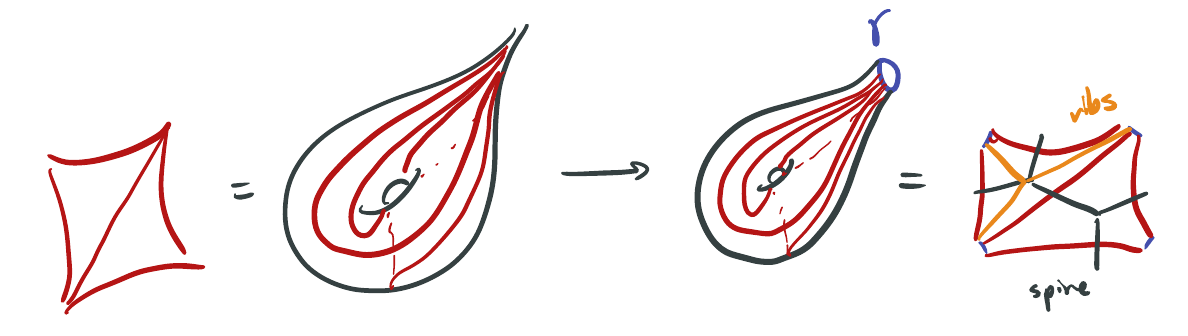}
\caption{No uniform bound on radii over the entirety of moduli space.}
\label{fig:nounifribbd}
\end{figure}

\subsection{Shear-shape cocycles}\label{subsec:shsh}
Let $\lambda\in \GLcr(S)$, and let $\arc = \cup_i \alpha_i$ be a filling arc system on $S\setminus \lambda$. Following \S7 of \cite{shshI}, we define shear-shape cocycles and discuss the structure of the space of these objects.

A shear-shape cocyle $\sigma$ for the pair $(\lambda, \arc)$ is an element of a certain cohomology theory.
We need the notion of an {\bf orientation} of $\lambda\cup \arc$,\label{ind:lamor} which is a continuous orientation of $\lambda$ and a continuous co-orientation of $\arc$ that agrees with the orientation of $\lambda$ at the endpoints of all of the arcs.
While $\lambda\cup \arc$ is not usually orientable, there is always a two-to-one orientation cover $\widehat {\lambda}\cup\widehat{\arc} \to \lambda\cup \arc$. 

Let $N_{\arc}$ be a {\bf snug neighborhood}\label{ind:snugtopnbhd} of $\lambda\cup \arc$, i.e., an open neighborhood of $\lambda\cup \arc$ on $S$ satisfying that the complement of $\lambda\cup \arc$ on $S$ and the complement of $N_{\arc}$ on $S$ have the same topological type.
Then the orientation cover $\widehat {\lambda}\cup\widehat{\arc} \to \lambda\cup \arc$ extends to a two-to-one cover $\widehat {N_{\arc}} \to N_{\arc}$.

A {\bf shear-shape cocycle}\label{ind:shshcoc} $\sigma$ for $(\lambda,\arc)$ is then a relative class $\sigma \in H^1(\widehat{N_{\arc}},\partial \widehat{N_{\arc}};\RR)^-$ that is anti-invariant with respect to the covering involution.
We also require that $\sigma$ evaluates positively on {\bf standard transversals}
\label{ind:standardtrans}, which are choices of relative cycles $t_i$ crossing the arc $\widehat{\alpha}_i$ of $\widehat\arc$ (and no others) which are disjoint from $\widehat \lambda$, given the orientation which agrees with the co-orientation of $\widehat{\alpha}_i$. 
From positivity on standard transversals and anti-invariance of $\sigma$ under the covering involution, we recover a (positively) weighted filling arc system \[\arcwt(\sigma) = \sum_{\alpha_i \in \arc} \sigma(t_i) \alpha_i \in |\Arcfill(S\setminus\lambda)|_\RR.\]
The space of shear-shape cocycles with arc system $\arc$ 
is denoted by $\SH(\lambda;\arc)$;\label{ind:shshwitharc} it is routine to check that the definition does not depend on the choice of snug neighborhood $N_{\arc}$ \cite[Lemma 7.7]{shshI}.

In more concrete terms, a shear-shape cocycle $\sigma \in \SH(\lambda;\arc)$ defines a weighted filling arc system $\arcwt(\sigma)$ as above and a function on the set of unoriented arcs that are transverse to $\lambda$ and disjoint from $\arc$ that satisfies:\label{ind:shshaxioms}
\begin{enumerate}
    \item [{(SH0)}] (support): If $k$ does not intersect $\lambda$ then $\sigma(k)=0$.
    \item [{(SH1)}] (transverse invariance): If $k$ and $k'$ are isotopic through arcs transverse to $\lambda$ and disjoint from $\arc$, then $\sigma(k) = \sigma(k')$.
    \item [{(SH2)}] (finite additivity): If $k = k_1 \cup k_2$ where $k_i$ have disjoint interiors, then $\sigma(k) = \sigma(k_1) + \sigma(k_2)$.
    \item[(SH3)] ($\arcwt$-compatibility): 
    Suppose that $k$ is isotopic rel endpoints and transverse to $\lambda$ to some arc which may be written as $t_i \cup \ell$, where $t_i$ is a standard transversal and $\ell$ is disjoint from $\arc$. Then the loop $k \cup t_i \cup \ell$ encircles a unique point $p$ of $\lambda \cap \arc$, and 
    \[\sigma(k) = \sigma(\ell) + \varepsilon c_{\alpha_i}\]
    where $\varepsilon$ denotes the winding number of $k \cup t_i \cup \ell$ about $p$ (where the loop is oriented such that the edges are traversed $k$ then $t_i$ then $\ell$). 
\end{enumerate}
The axioms (SH0) through (SH3) determine uniquely a shear-shape cocyle \cite[Proposition 7.14]{shshI}.

In \S9 of \cite{shshI}, the authors constructed ``train track coordinates'' for $\SH(\lambda, \arc)$; the point is that a shear-shape cocyle is completely determined by the weights it deposits on a suitable train track.
Compare \S\ref{sec:geometric tts} below, in which we construct geometric train track neighborhoods of laminations on hyperbolic surfaces and recall the standard terminology for train tracks.

The space of {\em all} shear-shape cocycles $\SH(\lambda)$\label{ind:SHlambda} is glued together from spaces $\SH(\lambda;\arc)$, where $\arc$ ranges over all filling arc systems of $S\setminus \lambda$, and the gluings are encoded by the combinatorics of $\Arcfill(S\setminus \lambda)$.  
Then $\SH(\lambda)$ has the structure of an affine $\cH(\lambda)$ bundle over $\Base$, where $\cH(\lambda)$ is Bonahon's space \cite{Bon_GLTHB,Bon_THDGL} of transverse cocycles\label{ind:transversecoc} for $\lambda$ \cite[Theorem 8.1]{shshI}.\footnote{The proof of the structure Theorem 8.1 in \cite{shshI} is written for measurable laminations $\lambda$, but extends with minor changes to chain recurrent geodesic laminations.}

\subsection{The shear-shape cocycle of a hyperbolic metric}
Let $X\in\T_g$ and let $\lambda$ be a chain recurrent geodesic lamination.
Let $\arcwt(X,\lambda)\in \Base$ be the weighted filling arc system with underlying geometric arc system $\arc$ recording the geometry of $X\setminus\lambda$ as in \S\ref{subsec:arc systems}.
Following Bonahon \cite{Bon_SPB}, we defined in \cite[\S13]{shshI} the {\bf geometric shear-shape cocycle} $\sigma_\lambda(X) \in \SH(\lambda;\arc)$ of the hyperbolic metric $X$.\label{ind:geomshsh}

Using the axioms for shear-shape cocycles, it is enough to compute the value of $\sigma_\lambda(X)$ on small transversal arcs $k$ to $\lambda$ that are disjoint from $\arc$.\footnote{This is essentially the assertion that augmented train tracks can be used to give coordinates for shear-shape space.}
Briefly, a small enough transverse arc $k$ has endpoints in hexagons $H_u$ and $H_v$ in the universal cover. These hexagons give basepoints to their boundary geodesics, and the orthogeodesic foliation defines an isometry between these two basepointed geodesics.
The quantity $\sigma_\lambda(X)(k)$ is a \emph{signed} distance between these two basepoints.

There is a bi-linear pairing\label{ind:Thform}
\[\omega_{\SH}: \SH(\lambda)\times \mathcal H(\lambda) \to \RR\]
that computes hyperbolic length: for every transverse measure $\mu$ with support contained in $\lambda$, we have 
\begin{equation}\label{eqn:positive pairing}
    \ell_X(\mu) = \omega_{\SH}(\sigl(X), \mu)>0.
\end{equation}
Thus the image of $\sigl$ is contained in the locus $\SH^+(\lambda)$ satisfying the positivity conditions \eqref{eqn:positive pairing}.\label{ind:SH+}
For $\lambda$ measurable, $\sigl$ is a homeomorphism onto $\SH^+(\lambda)$ \cite[Theorem 12.1]{shshI}.

\section{Thin parts}\label{sec:thin_parts}
In this section, we prove that $\cO$ maps the end of $\PoM_g$ into the end of $\QoM_g$.
More precisely, if $X\in \T_g$ has a short curve, then $\cO(X,\lambda)$ has a short curve, both in the singular flat metric (Proposition \ref{prop:thinparts}) and in the uniformizing hyperbolic metric (Corollary \ref{cor:thinparts}).
We also observe that $\cO$ does not necessarily map bounded sets to bounded sets (Example \ref{example:thicktothin}).

\begin{proposition}\label{prop:thinparts}
Suppose that $(X, \lambda) \in \PoM_g$ and that $\gamma$ is a simple closed curve on $X$ such that $\ell_X(\gamma) < \delta \le 1$. Then the length of the geodesic representative of $\gamma$ on $\cO(X, \lambda)$ is at most $\delta+(\log1/\delta)^{-1}$.
\end{proposition}
\begin{proof}
The flat geodesic length of $\gamma$ on $\cO(X, \lambda)$ is bounded above by the sum of the total variation of its real and imaginary parts; therefore we need only to bound $i(\gamma, \Ol(X))$ and $i(\gamma, \lambda).$

As the measure on $\Ol(X)$ is defined by isotoping transverse arcs onto $\lambda$ and measuring the length along $\lambda$, and nearest point projections to geodesics are distance-nonincreasing, 
we see immediately that 
\[i(\gamma, \Ol(X)) \le \ell_X(\gamma) < \delta.\]

To bound $i(\gamma, \lambda)$, we observe that if $\gamma$ is either a component of or disjoint from $\lambda$ then $i(\gamma, \lambda) = 0$ and we are done.
So assume that $\gamma$ and $\lambda$ intersect. By the density of weighted multicurves in $\ML$, there exists a multicurve $\beta$ such that $\frac{1}{\ell_X(\beta)} \beta$ is an arbitrarily good approximation of $\lambda$. 
Since $\gamma$ is $\delta$-short on $X$, the collar lemma \cite[Theorem 4.1.1]{Buser} implies that $X$ contains an embedded hyperbolic annulus of width 
\[w(\delta):=\sinh^{-1}(1/ \sinh(\delta/2))\]
about $\gamma$. In particular, each time that $\beta$ intersects $\gamma$ it must travel through this collar neighborhood and so $i(\gamma, \beta) w(\delta) \le \ell_X(\beta).$
By continuity of the intersection number, this in turn implies
\[i(\gamma, \lambda) \approx i \left( \gamma, \frac{1}{\ell_X(\beta)} \beta \right) \le 1/w(\delta) < \frac{1}{\log1/\delta},\]
where the last inequality holds for all $\delta\in (0,1]$.
Since the approximating curve $\frac{1}{\ell_X(\beta)} \beta$ can be taken arbitrarily close to $\lambda$, we conclude that
\[\text{length of } \gamma \text{ on } \cO(X, \lambda) \le i(\gamma, \Ol(X)) + i(\gamma, \lambda) < \delta + \frac{1}{\log1/\delta},\]
which is what we wanted to show.
\end{proof}

We say that a hyperbolic surface $X$ is {\bf $s$-thick} if its systole $\sys(X)$ is at least $s$; otherwise $X$ is {\bf $s$-thin}.\label{ind:thickthin}
We say that $q\in \QT_g$ is {\bf $s$-thick} if the hyperbolic metric in the conformal class of $q$ is $s$-thick.
Note that if $Y$ is the hyperbolic structure in the conformal class of $\cO(X,\lambda)$, then it may not hold that $\ell_X(\gamma)<\delta$ implies $\ell_{Y}(\gamma) <\delta + \frac{1}{\log1/\delta}$. 
However, we can still deduce that \emph{some} curve is short on $Y$.

\begin{corollary}[Inverse takes thick to thick]\label{cor:thinparts}
For any $s>0$ there is an $s'>0$ such that if $q\in \QoT_g$ is $s$-thick, then $\cO\inverse(q)\in \PoT_g$ is $s'$-thick.
\end{corollary}
\begin{proof}
    We are going to exhibit an $s'$ such that if $Y$ is the hyperbolic metric underlying $\cO(X,\lambda)$ and $\ell_Y(\gamma)>s$ for all curves $\gamma$, then $\ell_X(\gamma)>s'$ for all curves $\gamma$.
    Following \cite{Rafi:thickthin}, we let $\mu$ be a short marking for $Y$. 
    Applying the collar lemma for hyperbolic surfaces, we find that every curve $\beta$ in $\mu$ has length \[s\le \ell_Y(\beta) = O(\log(1/s)). \]
    
    Suppose the shortest curve $\gamma$ in $X$ has length $\delta$.
    By Proposition \ref{prop:thinparts}, $\ell_q(\gamma) = O\left(\frac{1}{\log(1/\delta))}\right)$.
    From the proof of \cite[Theorem 2 and Lemma 3]{Rafi:thickthin}, there are constants $C_1$, $C_2$, and $C_3$ depending only on the topology of $S$ such that 
    \[\ell_Y(\beta)\le C_1\exp\left(C_2 \left(\frac{1}{s}\right)^{C_3} \right)\frac{1}{\log(1/\delta)} \]
    for all curves $\beta$ in $\mu$, in particular, for $\gamma$.
    
    From the proof of \cite[Theorem 4]{Rafi:thickthin}, the area of $q$ is bounded above by the square of the sum of the $q$-lengths of the curves in $\mu$.
    On the other hand, since $q$ had area $1$, we see that 
    \[\log(1/\delta) \le \# \mu \cdot C_1\exp\left(C_2 \left(\frac{1}{s}\right)^{C_3} \right). \]
    This inequality is violated if $\delta$ is too small.
    This shows that $X$ is $s'$-thick with $s'$ depending only on $s$ and the topology of $S$, which is what we wanted.
\end{proof}

\begin{remark}
Proposition \ref{prop:thinparts} and Corollary \ref{cor:thinparts} also clearly hold on the locus of differentials with area bounded above and below.
They do not, however, hold on the entirety of $\PT_g$ and $\QT_g$, as rescaling a differential corresponds to rescaling a measured lamination as well as applying the dilation flow to the underlying hyperbolic surface.
\end{remark}

\subsection{Examples}

We now demonstrate that Corollary  \ref{cor:thinparts} is sharp with the following example, which shows that short curves of $(X, \lambda)$ may not be taken to short curves of $\cO(X, \lambda)$.

\begin{example}[Short curve to long curve]
For $t>0$, define an arc system on the 4-holed sphere as in Figure \ref{fig:gammathin}; as discussed in Section \ref{subsec:arc systems}, this specifies a hyperbolic metric with boundary components of length $(1,1, 2t, 2t)$.
Glue together the boundary components of the same length (with any twisting) to yield a hyperbolic structure $X_t$ on a genus two surface equipped with two simple closed curves $\gamma_1$ and $\gamma_2$ such that
\[\ell_{X_t}(\gamma_1) = 1 \text{ and } \ell_{X_t}(\gamma_2) = 2t.\]
Define $\underline{\gamma_t} = \gamma_1 + 2t \gamma_2$; then the quadratic differential $q_t := \cO( X_t, \underline{\gamma_t})$ is the union of two tori along a slit of length $t$. Compare Figure \ref{fig:gammathin}.

\begin{figure}[ht]
\centering
\includegraphics[scale=.6]{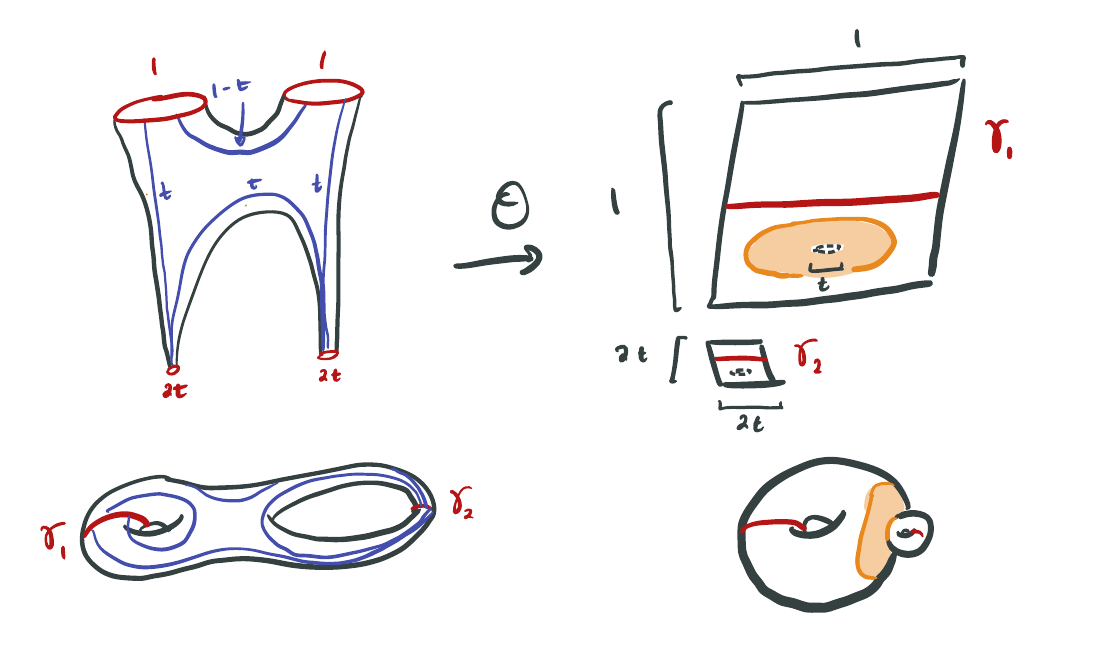}
\caption{Short curves on corresponding surfaces. While $\gamma_2$ is hyperbolically short on $X_t$, it is not conformally short on $q_t$. The shaded expanding annulus demonstrates that the separating curve $\gamma_3$ is conformally short on $q_t$.
}
\label{fig:gammathin}
\end{figure}

Now as $t \to 0$, we see that the moduli of the two pieces remain bounded as they always lie on the same horocycle in $\mathcal M_{1,1}$ (corresponding to different choices of twisting).
In particular, $\gamma_2$ is always the core curve of an embedded cylinder of height $t$ and width $2t$, i.e., modulus 1/2. 

However, we can find a large expanding annulus (shaded in Figure \ref{fig:gammathin}) homotopic to the curve $\gamma_3$ that separates the two tori, and hence the extremal length of $\gamma_3$ must go to $0$ as $t \to 0$.
\end{example}

We now observe that the converse to Proposition \ref{prop:thinparts} also does not hold: there are families of thick pairs $(X, \lambda)$ whose corresponding differentials exit the cusp of $\QoM_g$. Compare with \cite[Remark 5.10]{Wright_MirzEQ}.

\begin{example}[Thick to thin]\label{example:thicktothin}
Fix a surface of any genus and consider a pants decomposition $P$ one of whose curves separates off a genus one subsurface $\Sigma$.  Let $\gamma$ denote the curve of $P$ contained inside $\Sigma$.
Using Fenchel-Nielsen coordinates, set all curves of $P$ to have length 1. This in particular ensures the existence of an orthogeodesic arc $\alpha$ of weight $\nicefrac{1}{2}$ running between the two sides of $\gamma$.
Glue together the curves of $P \setminus \gamma$ with arbitrary twists, and glue together the two boundaries corresponding to $\gamma$ such that the endpoints of $\alpha$ match up; call the resulting hyperbolic structure $X$ and resulting curve $\delta$. By construction, we have that $\cO_P(X)$ contains a foliated annulus of weight $\nicefrac{1}{2}$ with core curve $\delta$. See Figure \ref{fig:thicktothin}.

\begin{figure}[ht]
\centering
\includegraphics[scale=.6]{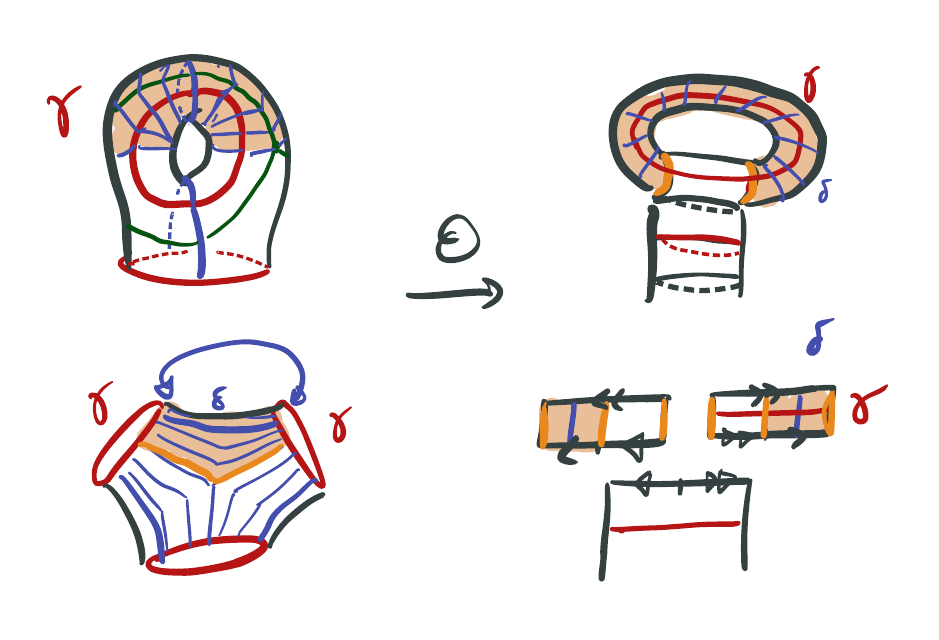}
\caption{Thick hyperbolic surfaces mapping to thin flat surfaces. As the coefficient of $\gamma$ gets smaller, it deposits less mass on $\delta$ and so there is a flat cylinder of larger modulus.}
\label{fig:thicktothin}
\end{figure}

Define the weighted multicurve
\[\underline{P}_t := \sum_{P \setminus \gamma} \gamma' + t \gamma.\]
Then the quadratic differentials $q_t := \cO(X, \underline{P}_t)$ each have area in $=(3g-4, 3g-3]$ and each contain a vertical cylinder with core curve $\delta$ whose circumference is $i(\delta, \underline{P}_t) = t$ and whose height is $\nicefrac{1}{2}$; compare Figure \ref{fig:thicktothin}.

The extremal length of $\delta$ on $q_t$ is bounded above by the reciprocal of the modulus of an embedded cylinder with core curve $\delta$, hence by $\frac{t}{1/2} = 2t$.  
Since this tends to $0$ with $t$ and extremal length provides upper bounds for  hyperbolic length in the same conformal class, we see that $q_t$ leaves the end of $\QoM_g$.
\end{example}

\section{Hausdorff-close geodesic tuples}\label{sec:triplecenters}

The goal of this section is to establish some elementary estimates in the hyperbolic plane related to configurations of geodesics. 
We will apply these estimates in 
Part \ref{part:continuity}
when we wish to show that if two laminations $\lambda$ and $\lambda'$ are Hausdorff close, then certain auxiliary points are also close; see Lemma \ref{lem:triple_centers} and Corollary \ref{cor:triplebasepoints}.
Suppose $G = (g_1, g_2, g_3)$ is a configuration of complete, pairwise disjoint geodesics in $\HH^2$ such that that none separates the others.
Recall from \S\ref{subsec:hexagons center and basepoints} that the center of $G$ is the center of the inscribed circle and this circle meets each geodesic $g_i$ at its basepoint.

\subsection{Close triples have close centers}
We consider two triples of geodesics satisfying that no geodesic in one triple separates the others in the same triple.  If the configurations are close on a ball of large radius around a center of one of the triples, then the center of the other triple is very close to that of the first.

\begin{lemma}\label{lem:triple_centers}
Fix $r\ge \log(\sqrt 3)$, let $e\ge 0$, and suppose $\zeta<1$.
Let $g_1$, $g_2$, and $g_3$ be pairwise disjoint complete geodesics in $\HH^2$, none of which separates the others, with distance at most $r$ from their center, $u$.
Suppose $h_i$ are complete geodesics that are $\zeta$-Hausdorff close to $g_i$ in the ball of radius $\log(1/\zeta)-e$ about $u$, for $i =1, 2, 3$. 
Then the center $v$ of $(h_1,h_2,h_3)$ satisfies \[d(u,v) \le O(\zeta^2),\]
where the implicit constant depends only on $r$ and $e$. 
\end{lemma}

\begin{proof}
Let $H_i(\zeta, e)$ be the set of geodesics that are $\zeta$-close to $g_i$ on the ball of radius $\log(1/\zeta)-e$ around $u$, and notice that $h_i\in H_i(\zeta, e)$.
The set $\partial H_i(\zeta,e)$ of endpoints of these geodesics has two components in the circle at infinity.  
An easy computation in the plane produces a bound $C_1\zeta^2$ on the diameter of each component of $\partial H_i(\zeta,e)$ in the visual metric on $\partial \HH^2$ centered at $u$, where $C_1$ depends only on $e$ and $r$.

The map that assigns a center to a triple of  geodesics, none of which separates the others, is smooth, and our hypotheses on the $g_i$, $r$, $e$, and $\zeta$ imply that the product $H_1(\zeta, e)\times H_2(\zeta,e)\times H_3(\zeta,e)$ is compact in this space.
We can therefore find a Lipschitz constant $L=L(g_1, g_2, g_3, r, e, \zeta)$ so the diameter of the corresponding set of centers is at most $L C_1\zeta^2$ in $\HH^2$.
Since $H_i(\zeta,e)$ shrinks as $\zeta$ goes to $0$, we may assume that $L$ is independent of $\zeta$.
Also, the set of triples of pairwise disjoint complete geodesics with distance at most $r$ from their center, none of which separates the others, is compact in the space of triples, so we can further remove the dependence of $L$ on $(g_1, g_2, g_3)$.
We have now produced the  desired bound.
\end{proof}

In addition to the above estimate on the distance between centers of Hausdorff-close triples, we also want to be able to compare the projections of these centers to each of the geodesics in the triple.
We begin with a general estimate on the distance between the projections of a point to Hausdorff-close geodesics.

\begin{lemma}\label{lem:projectionestimate}
Fix $r\ge 0$, 
let $g$ be a complete geodesic in $\HH^2$, and suppose $u \in \HH^2$ is distance at most $r$ from $g$.
Suppose that $h$ is a complete geodesic that is $\zeta$-Hausdorff-close on the ball of radius $b\log(1/\zeta)$ about $u$, where $b\in (0,1]$ and $\zeta<1$.
Then if $\pi_g$ and $\pi_h$ are the closest-point projection maps,
\[d(\pi_g(u), \pi_h(u)) = O(\zeta^{1+b})\]
where the implicit constant depends only on $r$.
\end{lemma}

\begin{proof}
We refer to Figure \ref{fig:basepointsclose} throughout the proof.

Let $\rho_g(u)$ denote the intersection of the line containing $u$ and $\pi_h(u)$ with $g$, and let $\rho_h(u)$ be the intersection of the line containing $u$ and $\pi_g(u)$ with $h$.  
Gromov hyperbolicity of $\HH^2$ implies that $g$ and $h$ fellow travel at scale $\zeta$ on segments of length at least $b\log(1/\zeta) - r- C$ to the right and to the left of $\pi_g(u)$, where $C$ is some universal constant.  
Negative curvature then bounds
\[d(\rho_h(u),\pi_g(u)) \le \zeta e^{-(b\log(1/\zeta)-r-C)} = O(\zeta^{1+b}).\]
A symmetric argument gives the same bound $d(\rho_g(u),\pi_h(u)) =O(\zeta^{1+b})$.  

Because $\pi_g$ is a closest point projection, we know that 
\[d(u, \pi_g(u)) \le d(u, \rho_g(u)) \le d(u, \pi_h(u)) + O(\zeta^{1+b})\]
where the last inequality is the triangle inequality.
Similarly, we have 
\[d(u, \pi_h(u)) \le d(u, \rho_h(u)) \le d(u, \pi_g(u)) + O(\zeta^{1+b}).\]

Thus, the right triangle with vertices $\pi_g(u), \rho_g(u)$ and $u$ is very nearly isosceles, with two sides whose lengths differ by $O(\zeta^{1+b})$.
The hyperbolic law of cosines then gives 
\[\cosh(d(u,\pi_g(u))+O(\zeta^{1+b}))=\cosh(d(u,\rho_g(u)) =\cosh(d(u,\pi_g(u))\cosh(d(\rho_g(u), \pi_g(u))).\]
Rearranging the terms and expanding, this implies 
\[\frac{\cosh(d(u,\pi_g(u))+O(\zeta^{1+b})}{\cosh(d(u,\pi_g(u)))} = 1+ O(\zeta^{1+b}) = 1+ O(d(\rho_g(u), \pi_g(u))).\]
Thus $d(\rho_g(u),\pi_g(u))=O(\zeta^{1+b})$, with  implicit constant depending only on $r$.

An application of the triangle inequality proves the lemma.
\end{proof}

\begin{figure}
\centering
\includegraphics[scale=.6]{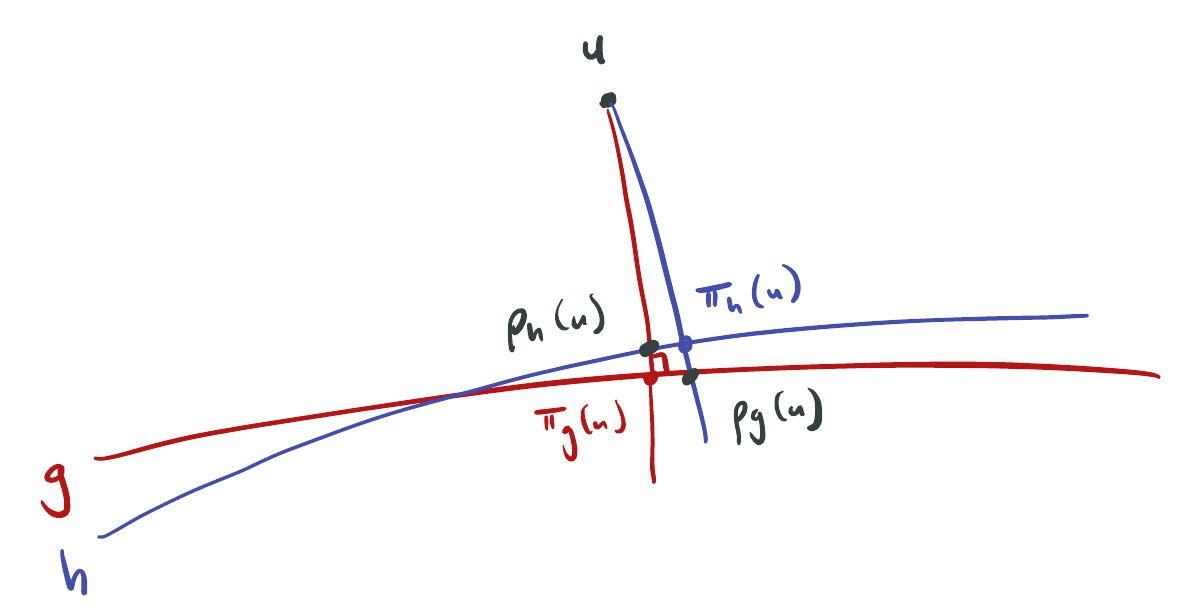}
\caption{Projections to Hausdorff-close geodesics}
\label{fig:basepointsclose}
\end{figure}

In particular, we highlight the following special case, which can be proven by taking $b = d(u,g) / \log(1/\zeta)$ and observing that if $h$ fellow-travels $g$ in a larger ball then the estimates in the Lemma get better.

\begin{corollary}\label{cor:projectionestimate2}
Let $r, g, u$ be as above and let $B$ be any hyperbolic ball centered at $u$ and meeting $g$. Then if $h$ is any geodesic that has Hausdorff distance at most $\zeta$ from $g$ in $B$, we have
\[d(\pi_g(u), \pi_h(u)) = O(\zeta)\]
where the implicit constant depends only on the cutoff $r$.
\end{corollary}

\begin{corollary}\label{cor:triplebasepoints}
Let all notation as in Lemma \ref{lem:triple_centers}.
If $p_i \in g_i$ and $q_i \in h_i$ are the basepoints for these geodesic triples, then we have
\[d(p_i, q_i) = O(\zeta^{2})\]
where the implicit constant depends only on $r$ and $e$.
\end{corollary}
\begin{proof}
Fix an $i \in \{1, 2, 3\}$ and let $\pi_g$ denote the closest-point projection of $\HH^2$ to $g_i$.
Since the closest-point projection map is a contraction, we know that 
\[d(p_i, \pi_g(v)) = d(\pi_g(u), \pi_g(v)) \le d(u, v) = O(\zeta^{2})\]
where the last estimate follows by Lemma \ref{lem:triple_centers}. 
On the other hand, Lemma \ref{lem:projectionestimate} implies that 
\[d(\pi_g(v), q_i) = d(\pi_g(v), \pi_h(v)) = O(\zeta^{2}).\]
These two estimates plus the triangle inequality prove the desired statement.
\end{proof}

\subsection{Equidistant configurations}

Valence $n$ vertices of $\Sp$ occur when an $n$-tuple of geodesics are equidistant from a point, but these configurations are highly unstable as one deforms in the Hausdorff topology.
Here we prove an estimate on nearly-equidistant tuples that will be important in Section \ref{sec:persistent} to compare arc systems on nearby laminations.

\begin{definition}\label{def:equidistant}
We say that a configuration of $n$ complete pairwise disjoint geodesics in $\HH^2$, none of which separates the others, is {\bf $\zeta$-equidistant} if the collection of centers of triples of geodesics in the collection form a set of diameter at most $\zeta$. 
\label{ind:zetaequid}
\end{definition}

The boundary of a regular ideal $n$-gon is $0$-equidistant, as is any collection of disjoint geodesics equidistant from a given point.

The following lemma states that if two pairs of $4$-tuples are close to one another on a ball of large radius but different pairs of opposite geodesics are closer in one $4$-tuple than the other, then both configurations are nearly equidistant.
Compare Corollary \ref{cor:persist_notallarcs} and Figure \ref{fig:persistent} below.
As discussed in Example \ref{example:arc complex 4-gon}, the weighted filling arc system of a configuration of $4$ pairwise disjoint geodesics in the plane, none of which separates the others, is homeomorphic to $\RR$.

\begin{lemma}\label{lem:4_centers_close}
Let $G=(g_1, g_2, g_3, g_4)$ be pairwise disjoint complete geodesics in $\HH^2$, none of which separates the others, and let $r$ be an upper bound on the distance between any center of $G$ and any geodesic of $G$. 
Suppose $H=(h_1, h_2, h_3, h_4)$ is $\zeta$-Hausdorff close to $G$ in a ball of radius at least $\log 1/\zeta -e $ around one of the centers $u$ of $G$. 
If the dual arcs to the two configurations $G$ and $H$ cross each other, 
then both $G$ and $H$ are $O(\zeta^2)$-equidistant with implicit constant depending only on $r$ and $e$.  
\end{lemma}
\begin{proof}
We let $H_i$ be the collection of complete geodesics that are $\zeta$-Hausdorff close to $g_i$ in a ball of radius $\log1/\zeta-e$ around $u$.
As in the proof of Lemma \ref{lem:triple_centers}, the visual diameter centered at $u$ of each component of the set of endpoints of geodesics in $H_i$ is at most $O(\zeta^2)$ with implicit constant depending only on $r$ and $e$.

We may homotope the endpoints of $h_i$ to endpoints of $g_i$ so that the geodesics joining those points at infinity stay inside $H_i$ for all time.  
This defines a path $H_t$ in the space of $4$-tuples from the configuration $H$ to the configuration $G$.  
By continuity of the map which associates to a $4$-tuple its weighted dual arc system, there is some time $t_0$ for which the arc system of the configuration $H_{t_0}$ is empty. That is, $H_{t_0}$ is an $0$-equidistant configuration.  

There is a unique center $o$ of $H_{t_0}$ and now we can apply Lemma \ref{lem:triple_centers} to see that all centers of $G$ and of $H$ are $O(\zeta^2)$ from $o$.
The implicit constant in the previous sentence depends (continuously) on the isometry type of $H_{t_0}$, but there is a compact family of such configurations at bounded distance $r + O(\zeta)$ from $o$.  We may therefore choose the implicit constant to depend only on $r$. 
The triangle inequality concludes the proof of the lemma.
\end{proof}

Using Lemma \ref{lem:projectionestimate}, this of course implies that the basepoints of all of the triples of geodesics in $G$ differ by $\zeta^2$, and the same for $H$.
In particular, combined with Corollary \ref{cor:triplebasepoints}, this gives us a bound of $O(\zeta^2)$ for the diameter of the collection of all of the basepoints on the coresponding geodesics of $G$ and $H$.

\part{Train tracks and cellulations}\label{part:tts and cellulations}

In this part of the paper, we discuss a number of geometric-combinatorial constructions that we will make frequent use of in the later parts.
For measured geodesic laminations on hyperbolic surfaces, these are (augmented) ``equilateral train tracks,'' while for quadratic differentials these are specific families of cellulations by saddle connections.
The main point of this part is to show that these constructions are dual to each other and thereby develop a technical framework in which to state quantified versions of our continuity results.

\section{Geometric train tracks}\label{sec:geometric tts}
\subsection{Uniform train tracks}\label{subsec:uniformttdef}

The following is a variant of Thurston's construction of a train track neighborhood of a geodesic lamination on a hyperbolic surface (\cite[\S8.9]{Thurston:notes} and \cite[\S4]{Th_stretch}).
\begin{construction}[Uniform train track neighborhood]\label{const:uniform_tt}
Let $\lambda\subset S$ be a geodesic lamination and $X\in \T_g$.  Consider the $\delta$-neighborhood $\cN_{\delta}(\lambda)\subset X$; it is foliated by leaves of the orthogeodesic foliation.\label{ind:unifnbhd}

If the restriction of every leaf to $\cN_{\delta}(\lambda)$ is a segment, then the map $\pi:\cN_{\delta}(\lambda) \to \cN_{\delta}(\lambda)/\sim$ that collapses these leaves extends to a homotopy equivalence $X\to X$.  The leaf space $\tau = \tau(X, \lambda, \delta)$ is a train track carrying $\lambda$ which can be $C^1$ embedded into $\cN_{\delta}(\lambda)$ as a deformation retract.
\end{construction}

\begin{remark}
Thurston's construction employed a horocyclically foliated neighborhood of a geodesic lamination and made some choices about how to interpolate between opposite horocycles away from the spikes.
The combinatorics of train tracks constructed using horocycle foliations and orthogeodesic foliations are identical for appropriate defining parameters, but 
there are no choices in the definition of the orthogeodesic foliation and it behaves uniformly over different topological types of laminations.
\end{remark}

A train track $\tau$ constructed as in Construction \ref{const:uniform_tt} from parameters $X$, $\lambda$, and $\delta$ is called a \textbf{uniform geometric train track}.\label{ind:uniftt}
Any uniform geometric train track is implicitly equipped with the data of its collapse map $\pi: \cN_\delta(\lambda)\to \tau$.
\medskip

We adopt the following standard definitions and notation:\label{ind:ttbasics}
\begin{itemize}
\item The \textbf{branches} of $\tau$ are its edges, and the \textbf{switches} of $\tau$ are its vertices (with their $C^1$ structure). 
\item  A train track is \textbf{generic} or \textbf{trivalent} if all switches are trivalent.
\item  The \textbf{ties} of $\tau$ or $\cN_\delta(\lambda)$ are connected components of restrictions of the leaves of $\Ol(X)$ to the neighborhood.
All ties are rectifiable compact arcs and meet $\lambda$ at right angles.
\item  A \textbf{train path} $\gamma$ is a $C^1$ edge-path in $\tau$.
\item The \textbf{width} of $\tau$ or $\cN_{\delta}(\lambda)$ is the supremum of the length of its ties.
\item {\bf Spikes} of $S \setminus \lambda$ (or $S \setminus \tau$) are regions bounded by two segments of $\lambda$ or of $\tau$ which share an ideal endpoint (or tangential direction at a switch), together with an arc in the complement.
\end{itemize}

The spikes of $\tau$ often correspond to pairs of asymptotic geodesics in $\lambda$, but may also arise when there is an arc of $\arc(X,\lambda)$ with length smaller than $2\delta$.
In \S\ref{subsec:geominvis}, we discuss the notion of ``proto-spikes,'' which allows us to treat these two geometrically similar cases in a uniform manner.

\subsection{Length along train tracks}\label{subsec:length along tau}
Geometric train tracks inherit a length function from $\lambda$.
More precisely, we can identify each branch of $\tau$ with an interval as follows: given two points $p_1$ and $p_2$ in the same branch, the distance from $p_1$ to $p_2$ is the length of any segment of a leaf of $\lambda$ running along the branch between the ties $\pi^{-1}(t_1)$ and $\pi^{-1}(t_2)$, and this is well-defined because transport along the orthogeodesic foliation preserves length along the lamination.

As an example, let us compute the lengths of the sides of triangles of $S \setminus \tau$ when the base lamination is maximal.
Consider first an ideal hyperbolic triangle, equipped with the orthogeodesic foliation relative to its boundary.
The foliation has a unique trivalent singular leaf, whose intersections with the boundary are the basepoints.
Elementary hyperbolic geometry allows us to compute that the distance $D_\delta$ along the boundary between any basepoint and either of the leaves of the orthogeodesic foliation of length $2\delta$ is given by \label{ind:Ddelta}
\begin{equation}\label{eqn:defDdelta}
D_\delta := \log\left(\frac{\coth(\delta)}2\right) = \log(1/\delta)-\log (2)+O(\delta^2).
\end{equation}
Note in particular that this is a constant depending only on $\delta$.

Now suppose $\lambda$ is maximal and $\tau = \tau(X, \lambda, \delta)$ is a uniform geometric train track. Then each side of each triangle of $S \setminus \tau$ defines a train path running between spikes of $\tau$, and by the above discussion this path must have length (as measured along $\lambda$) exactly $2D_\delta$. Moreover, the midpoint of each of these train paths is the image of a basepoint of a complementary triangle under the collapse map.

We emphasize that when the base lamination is not maximal, the side lengths of $S \setminus \tau$ are not necessarily all equal (even if $\tau$ is maximal).

\subsection{Equilateral geometric train tracks}\label{subsec:equittdef}
The following construction of a geometric train track neighborhood reflects bettter the singular flat geometry of $\cO(X,\lambda)$ when $\lambda$ is not maximal.
Our main desideratum is that the distance between the tips of the spikes of $X \setminus \tau$ and the corresponding basepoints should be constant over all spikes.
We therefore build neighborhoods with this property in mind.

Recall from \eqref{eqn:defDdelta} that $D_\delta \approx \log(1/2\delta)$ is the distance between any basepoint of an ideal triangle and the leaves of the orthogeodesic foliation of length $2\delta$, as measured along the triangle's boundary.

\begin{construction}[Equilateral train track neighborhood]\label{constr:equitt}
Let $\delta<\log(\sqrt3)$ be given and fix a crowned hyperbolic surface $Y$. The set of basepoints of the orthogeodesic foliation $\cO_{\partial Y}(Y)$ is a finite set in $\partial Y$, and we let $I \subset \partial Y$ denote those points that are at most $D_\delta$ away from a basepoint, as measured along $\partial Y$. Set
\[\text{Thick}_\delta(Y) := \{ y \in Y \mid \text{ the leaf of } \cO_{\partial Y}(Y) \text{ through } y \text{ meets } I \}\]
and set Thin$_\delta(Y)$ to be its complement. 
Then we define the $\delta$--equilateral neighborhood of the boundary as
\[{\cE}_{\delta}(\partial Y) := \cN_{\delta}(I) \cup \text{Thin}_\delta(Y)\]
where $\cN_{\delta}(I)$ is the uniform $\delta$-neighborhood of $I$. See Figure \ref{fig:equi-nbhd}.\label{ind:equithickthin}

Given a hyperbolic surface $X$ and a geodesic lamination $\lambda$, we define a $\delta$-{\bf equilateral neighborhood} ${\cE}_{\delta}(\lambda)$ of $\lambda$ in $X$ by taking the union of the equilateral neighborhoods for each component of $X \setminus \lambda.$\label{ind:equinbhd}
\end{construction}

We observe that the boundary of each component of $\Thin_\delta(Y)$ decomposes into two geodesic segments of $\lambda$, together with either one or two leaves of $\cO_{\partial Y}(Y)$.

\begin{figure}
    \centering
    \includegraphics[scale=.2]{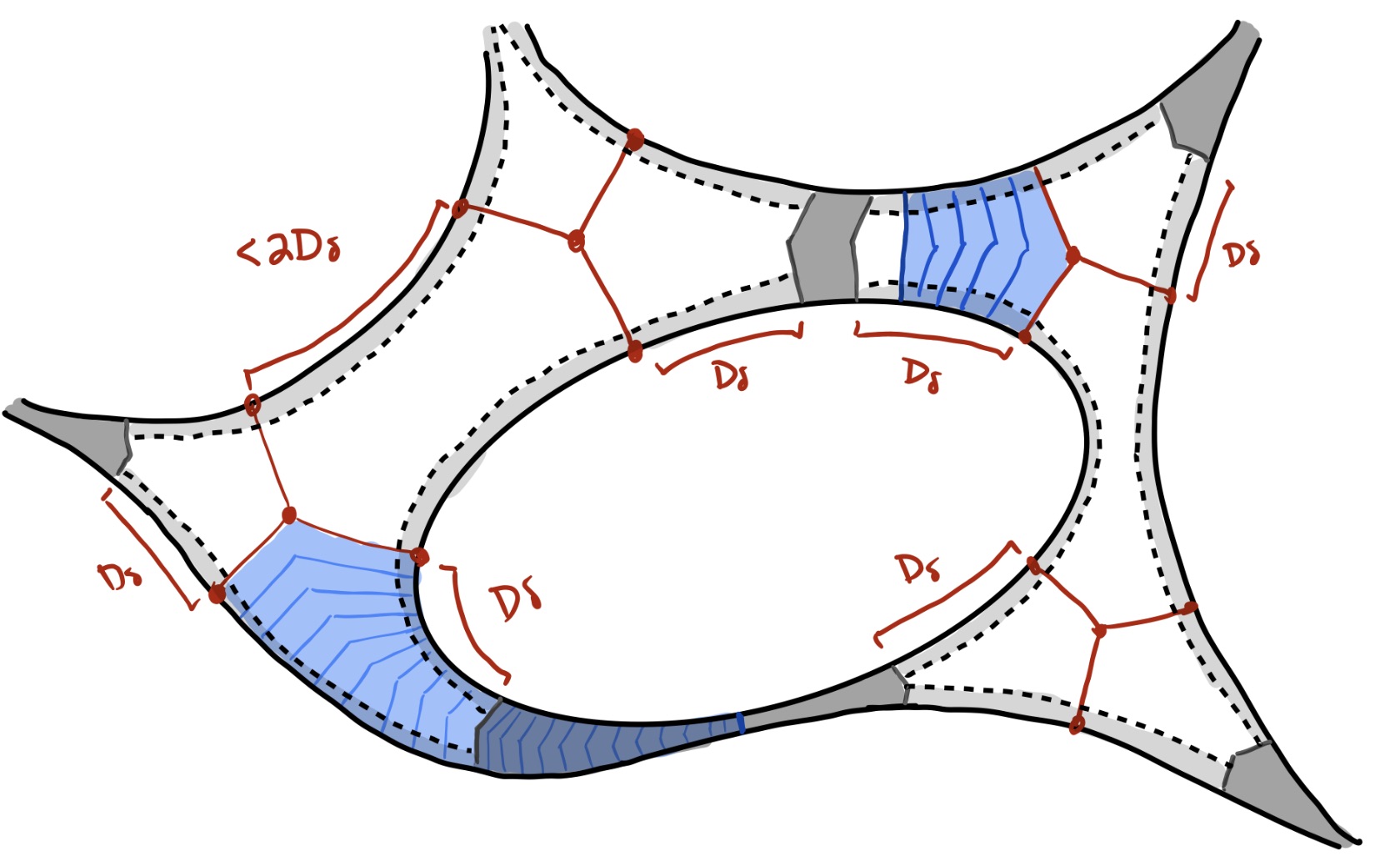}
    \caption{Construction of the equilateral neighborhood (shaded) of the boundary of a crowned hyperbolic surface. The dark shaded regions are the $\delta$-thin parts of $Y$; observe that they do not necessarily have to contain the orthogeodesic arc in the homotopy class of leaves.
    The shaded blue regions are both proto-spikes.}
    \label{fig:equi-nbhd}
\end{figure}

One can also define the decomposition of $Y$ into Thick$_\delta(Y)$ and Thin$_\delta(Y)$ as follows. Contracting along the leaves of $\cO_{\partial Y}(Y)$ defines a deformation retract $r$ of $Y$ onto its spine $\Sp(Y)$, and the lengths of projections of edges to $\partial Y$ allows us to identify each each of $\Sp(Y)$ with an interval, hence defining a metric.
Letting $J$ denote the set of points of $\Sp(Y)$ at most $D_\delta$ away from a vertex with respect to this metric, we have that
\[\text{Thick}_\delta(Y) = r^{-1}(J).\]
From this discussion, it becomes apparent that the complement of $J$ in $\Sp(Y)$ is a union of intervals, hence each component of Thin$_\delta(Y)$ is a band foliated by parallel leaves of $\cO_{\partial Y}(Y)$.

Any equilateral neighborhood $\cE_\delta(\lambda)$ is foliated by (segments of) leaves of $\Ol(X)$; if the collapse map 
\[\pi:\cE_\delta(\lambda) \to \cE_\delta(\lambda)/\sim\]
is a homotopy equivalence, 
then we say that
\[\tau( X, \lambda , \delta) := \cE_\delta(\lambda)/\sim\]
is an {\bf equilateral train track} and $\cE_\delta(\lambda)$ is an {\bf equilateral train track neighborhood}.\label{ind:equitt}
Note that the length along $\lambda$ endows the branches of $\tau$ with a length function, as in Section \ref{subsec:uniformttdef}.

For the rest of the paper, the expression $\tau(X, \lambda, \delta)$ will always refer to an equilateral train track (and not a uniform geometric one).
\medskip

\begin{remark}
The reader may find it helpful to compare this construction to the ``unzipping'' of a quadratic differential $q$ along its horizontal separatrices, as in Construction \ref{const:ttfromQD} below.
\end{remark}

\subsection{Geometrically (in)visible arcs}\label{subsec:geominvis}
Fix an equilateral neighborhood ${\cE_\delta(\lambda)} \subset X$ as in Construction \ref{constr:equitt}.  Suppose that $\cE_\delta(\lambda)$ has no closed leaves so that the leaf space $\tau =  \tau(X,\lambda,\delta)$ is a equilateral train track. 
As noted above, the components of Thin$_\delta(X \setminus \lambda)$ are spikes and bands foliated by parallel leaves of $\cO_{\lambda}(X)$.  Leaves of the compact components (bands) are necessarily isotopic to some arc of $\arc = \arc(X, \lambda)$.
However, the orthogeodesic representative of an arc may not be completely contained in $ \cE_\delta(\lambda)$, which may occur when one side of an arc is ``sharper'' than the other; compare Figure \ref{fig:equi-nbhd}.

\begin{definition}\label{def:geometric visibility}
The arcs of $\arc$ which can be properly isotoped into Thin$_\delta(X \setminus \lambda)$ are {\bf invisible} with respect to $ \cE_\delta(\lambda)$ or $\tau$.\label{ind:geominvis}
The remaining arcs of $\arc$ are called {\bf visible}, and we denote the subsystem of all visible arcs by $\gvarc \subset \arc$.
\end{definition}

The spikes of $\tau$ correspond either to pairs of asymptotic geodesics or invisible arcs.
In either case, we call the region of $X \setminus \lambda$ that is foliated by leaves of $\Ol(X)$ parallel to the switch leaf that lie on the same side of the orthogeodesic realization of that leaf (if one exists) a {\bf proto-spike} at scale $\delta$. \label{ind:proto-spike}
Note that so long as $\delta$ is taken small enough, the orthogeodesic representatives of leaves of invisible arcs are always leaves of $\Ol(X)$.
See Figure \ref{fig:equi-nbhd}.

Let $\arc_\bullet = \arc_\bullet (X,\lambda,\delta)$ denote the visible arc system and ${\cE}= \cE_\delta(\lambda)$.
For each arc of $\arc_\bullet$, pick the leaf of $\Ol(X) \setminus \lambda$ in the ``middle'' of the isotopy class, i.e., such that the distances (as measured along $\lambda$) between the leaf and the basepoints on either side are equal.
Each such leaf intersects ${\cE}$ in two orthogeodesic segments, and so its image under the collapse map is an arc with two well-defined endpoints on the equilateral train track $\tau$.
We may therefore consider the {\bf augmented equilateral train track} \label{ind:augtt}
\[\tau\cup \arc_\bullet = (\cE \cup \arc_\bullet)/ \sim\]
which can be thought of as a train track with arcs attached at right angles. 
In Section \ref{subsec:smoothings} below, we show how to smooth the incidences of $\alpha_\bullet$ onto $\tau$ to obtain a train track $\taua$ extending $\tau$.

\begin{lemma}\label{lem:visiblearcsfill}
Any augmented equilateral train track is filling, i.e., $S \setminus (\tau \cup \arc_\bullet)$ is a union of topological disks.
\end{lemma}
\begin{proof}
Any simple closed curve contained in $S \setminus \tau$ can be realized as a curve in $X \setminus {\cE}$. In particular, it is disjoint from $\lambda$. Now $\arc(X, \lambda)$ fills $S \setminus \lambda$, so the curve must run through some arcs of $\arc(X, \lambda)$.
These cannot be contained in ${\cE}$ (or isotoped into it) by assumption, so they must be visible.
\end{proof}

\subsection{Equilateral and uniform neighborhoods}
The two Constructions \ref{const:uniform_tt} and \ref{constr:equitt} of geometric train track neighborhoods are related as follows.
This estimate will allow us to make geometric statements about (the more useful) equilateral train tracks using (the simpler) uniform ones.

\begin{proposition}[Comparing equilateral and uniform]\label{prop:ttdefs_comp}
There are universal constants $\delta_{\ref{prop:ttdefs_comp}}<\log(\sqrt3)$ and $w_{\ref{prop:ttdefs_comp}}\ge 1$ such that for all $\delta\le \delta_{\ref{prop:ttdefs_comp}}$,
\[\cN_{w_{\ref{prop:ttdefs_comp}}\inverse\delta}(\lambda) \subset \cE_{\delta}(\lambda)\subset \cN_{w_{\ref{prop:ttdefs_comp}}\delta}(\lambda).\]
\end{proposition}
The proof of this statement follows from bounding the following geometric quantity.

\begin{definition}
Let $Y$ be a crowned hyperbolic surface and fix $\delta < \log(\sqrt{3})$.
Let $V$ be a proto-spike at scale $\delta$, and suppose first that $V$ is a spike or is bounded by an arc $\alpha$ contained in $\text{Thin}_\delta(Y)$. 
Define $\Delta_V(\delta)$ to be one half of the length of the longest leaf of $\cO_{\partial Y}(Y)$ contained in (the closure of) Thin$_\delta(Y) \cap V$.\label{ind:DVd}

Otherwise, $V$ is bounded by an arc $\alpha$ that is disjoint from $\text{Thin}_\delta(Y)$, but can be properly homotoped into a component $Z$ of $\text{Thin}_\delta(Y)$; see Figure \ref{fig:equi-nbhd}.
In this case, we define $\Delta_V(\delta)$ to be one half the length of the boundary leaf of $Z$ that is closest to the center of the hexagon containing $V$.
\end{definition}

We note that if Thin$_\delta(Y)$ is nonempty, then $\Delta_V(\delta)$ will always be achieved by the corresponding boundary leaf.
This quantity allows us to compare the width of an equilateral neighborhood with that of a uniform one; Proposition \ref{prop:ttdefs_comp} is therefore an immediate consequence of the following lemma.

\begin{lemma}\label{lem:bounded_neighborhood}
There are universal constants $\delta_{\ref{prop:ttdefs_comp}}<\log(\sqrt3)$ and $w_{\ref{prop:ttdefs_comp}}\ge 1$ such that for any crowned hyperbolic surface $Y$ and any proto-spike $V$ at scale $\delta\le \delta_{\ref{prop:ttdefs_comp}}$, we have
\[w_{\ref{prop:ttdefs_comp}}\inverse\delta \le{\Delta_V(\delta)} \le w_{\ref{prop:ttdefs_comp}} \delta.\]
\end{lemma}

It is not difficult to see that if $\delta$ is small enough, then any leaf of $\cO_{\partial Y}(Y)$ with distance along $\partial Y$ at most $\log(2+\sqrt3)$ from a leaf of length $2\delta$ has length at most $\log(3)$.  Let $\delta_{\ref{prop:ttdefs_comp}}$ be as such.

Assume that $V$ is a protospike at scale $\delta\le \delta_{\ref{prop:ttdefs_comp}} $.  In particular, if $V$ is bounded by an arc $\alpha$, then $\ell(\alpha)< \log(3)$.
Using \cite[Lemma 6.6]{shshI}, we know  that for any such protospike, there is a leaf of the orthogeodesic foliation of length $\log 3$ in $V$.
Note that $V$ inherits basepoints from the hexagon containing it.

\begin{lemma}\label{lem:boundedlengthspikes}
The distance $E_V$ along $\partial Y$ between the basepoints of $V$ and the leaf of length $\log 3$ in $V$ is at most $\log(2+\sqrt3)$.
\end{lemma}

\begin{proof}
Suppose first that $V$ is a spike.  If the distance $E_V$ between the leaf of length $\log 3$ and the center of the hexagon $H$ containing it was $\log 2 $ or larger, then the center of $H$ would be infinitely far away, which is absurd; see Figure \ref{fig:ortho_spikes}.

Now suppose that $V$ is bounded by an arc $\alpha$; by assumption, $\ell(\alpha)\le 2\delta<\log3$.
A computation shows that the larger angle $\varphi$ made by the leaf of the orthogeodesic foliation of length $\log3$ in $V$ and the edge of the spine that it meets ranges between $\pi/2$ and $2\pi/3$; these values of $\varphi$ correspond, respectively, to the cases  $\ell(\alpha)= \log3$ and $\ell(\alpha) =0$, i.e., that $V$ is a spike.  
Consider the complete geodesic in $\HH^2$ that extends the relevant edge of the spine, and the projection of its ideal endpoints to a geodesic comprising the spike $V$ (see Figure \ref{fig:ortho_spikes}).  
Then $E_V$ is bounded above by the distance between these points and the leaf of length $\log3$.
Another explicit computation shows that this upper bound varies monotonically between $\log(2)$ when $\varphi = 2\pi/3$ and $\log(2+\sqrt3)$ at the other extreme where $\varphi = \pi/2$.
This completes the proof.
\end{proof}

\begin{figure}[ht]
\centering
\includegraphics[scale =.2]{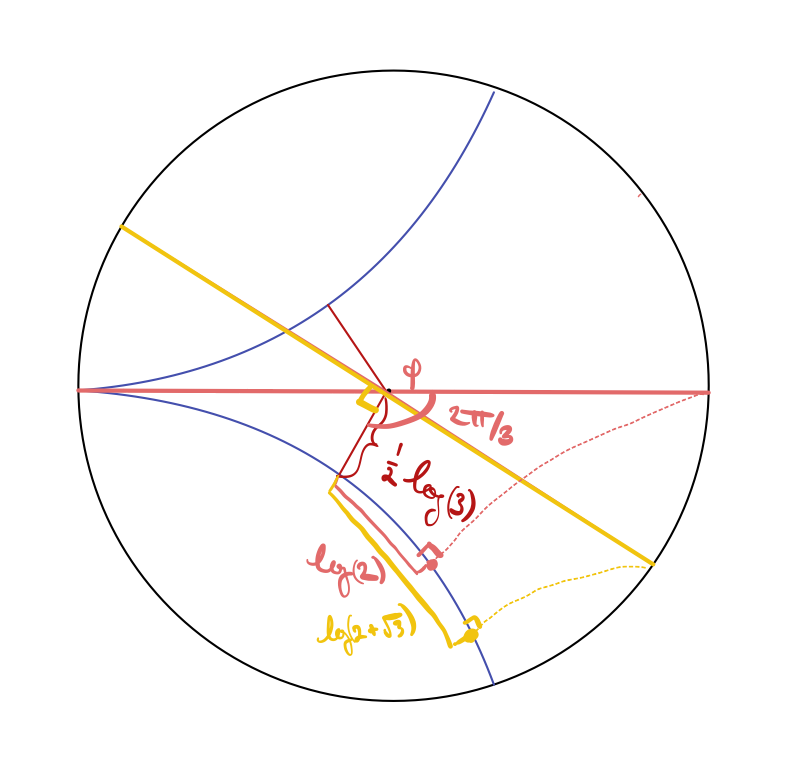}
\hspace{3ex}
\includegraphics[scale=.15]{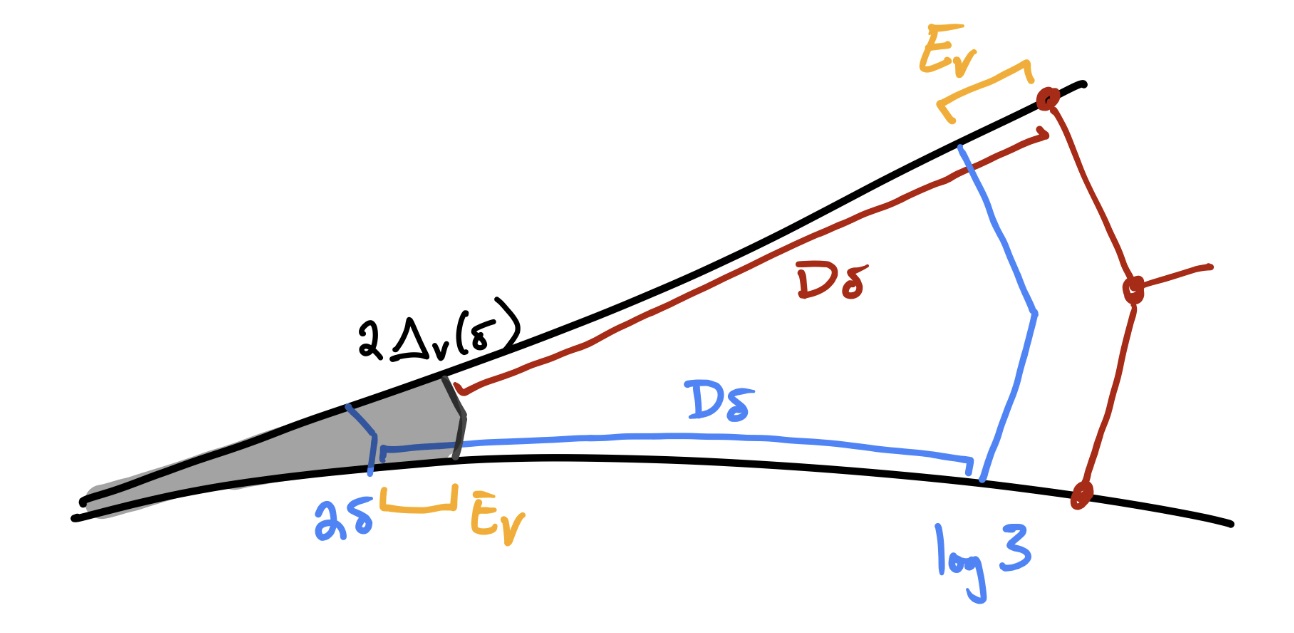}
\caption{Left: The angle $\varphi$ between the leaf of $\cO_{\partial Y}(Y)$ of length $\log (3)$ in $V$ and the spine of $Y$.
Right: The distances between leaves in a spike.
}
\label{fig:ortho_spikes}
\end{figure}

\begin{proof}[Proof of Lemma \ref{lem:bounded_neighborhood}]
First we prove the assertion for spikes in a component $Y$ of $X\setminus \lambda$.

The leaf through the basepoints of $V$ has length at least $\log(3)$.
By definition of $D_\delta$, the distance in the spike $V$ along its boundary between the leaf of length $2\delta$ and the leaf of length $\log(3)$ is exactly $D_\delta$. 
By definition of $\Delta_V(\delta)$, the distance in $V$ along its boundary between the basepoint on $V$ and the leaf of length $2\Delta_V(\delta)$ is also $D_\delta$.
Thus, $\delta\le \Delta_V(\delta)$ and the distance between the leaf of length $2\delta$ and the leaf of length $2\Delta_V(\delta)$ is equal to $E_V$.
Since $\delta\le \delta_{\ref{prop:ttdefs_comp}}$ and $E_V\le \log(2+\sqrt3)$ by Lemma \ref{lem:boundedlengthspikes}, we have $2\Delta_V(\delta)\le\log(3)$.
Compare Figure \ref{fig:ortho_spikes}.

Now the lengths of leaves of the orthogeodesic foliation of length at most $\log(3)$ in a spike are comparable to the lengths of the horocyclic segments with the same endpoints by a universal
multiplicative factor $C$. 
Because horocycles grow exponentially in spikes,
this implies that $\Delta_V(\delta)$ cannot be larger than $C\delta e^{E_V}$.
Applying Lemma \ref{lem:boundedlengthspikes} yields the conclusion for spikes.
\medskip

The proof for proto-spikes bounded by an arc $\alpha$ of length at most $\log3$ is similar.
This time, however, we have to consider the possibility that Thin$_\delta(Y)\cap V$ is empty but that there is a component $Z$ of Thin$_\delta(Y)$ contained in the proto-spike $V'$ that meets $V$ along their shared arc (as in Figure \ref{fig:equi-nbhd}).
Then $E_V$ still computes the distance along the boundary of $V\cup V'$ between the leaf of the orthogeodesic foliation contained in $V$ of length $2\delta$ and the boundary component of $Z$ facing $V$.
Again since $\delta\le \delta_{\ref{prop:ttdefs_comp}}$, we have $2\Delta_V(\delta)\le\log(3)$.

Now we use the fact that for every leaf of the orthogeodesic foliation in $V \cup V'$ there is a hypercycle with the same endpoints on $\partial Y$.
The lengths of a hypercycles parallel to $\alpha$ distance $d$ and $d+E_v$ away are $\ell(\alpha)\cosh(d)$ and $\ell(\alpha)\cosh(d+E_V)$, respectively.  
The ratio of these lengths does not exceed $2\cosh(E_V)$.

For leaves of $\cO_{\partial Y}(Y)$ with length at most $\log(3)$, there is a constant $C'$ such that the ratio of the length of a leaf of $\cO_{\partial Y}(Y)$ in $V$ and the length of the corresponding hypercyclic arc with the same endpoints is bounded above by $C'$ and below by $1/C'$. 
The existence of such a $C'$ follows from the fact that as $d \to \infty$, proto-spikes converge to spikes and hypercycles to horocycles, and hence we can apply the bounds of the previous paragraph.

In particular, comparing the leaves of $\cO_{\partial Y}(Y)$ of length $2\delta$ and $2\Delta_V(\delta)$ to their corresponding hypercycles, we obtain
\[\frac{\delta}{C'^22\cosh(E_V)}\le \Delta_V(\delta) \le C'^22\cosh(E_V)\delta.\]
Appealing to Lemma \ref{lem:boundedlengthspikes} completes the proof.
\end{proof}

\subsection{Geometry of equilateral train tracks}

The width of a geometric train track neighborhood shrinks linearly in the defining parameter, where the linear rate is uniform for all surfaces in a compact part of the moduli space.

\begin{lemma}[Width of equilateral neighborhoods]
\label{lem:equittwidth}
    Given $s>0$ there is a $W_{\ref{lem:equittwidth}}>0$ such that if $X\in \T_g$ is $s$-thick, and $\lambda$ is a geodesic lamination, then the width of $\cE_\delta(\lambda)$ is at most $W_{\ref{lem:equittwidth}}\delta$.
\end{lemma}

\begin{proof}
We will first produce a bound on the length of a tie of $\cN_\delta(\lambda)$.

Let $t$ be a tie of $\cN_\delta(\lambda)$. As the intersection $t\cap \lambda$ has $1$-dimensional Lebesgue measure zero (Lemma \ref{lem:lamfacts} item (\ref{item:measurezero})), it suffices to bound the sum of the lengths of each component of $t\setminus \lambda$.
For each proto-spike $V$, we consider the components of $t\setminus \lambda$ corresponding to $V$, i.e.,  $t\cap V$.  The longest such component has length at most $2\delta$.
There is a smooth train path that follows the proto-spike in forward time, and recurrences to $t$ correspond to shrinking components of $t\cap V$.
Subsequent recurrences to $t$ define homotopically non-trivial loops in $S$, and therefore have length at least $s$.

As in the proof of Lemma \ref{lem:bounded_neighborhood}, the lengths of leaves of the orthogeodesic foliation decay at rate $Ce^{-d}$, where $d$ is the distance along the spike and $C$ is a universal constant.
Thus the total length of $V \cap t$ is at most
\[\sum_{r = 0}^\infty 2C \delta e^{-rs}.\]
There are at most $6|\chi(S)|$ proto-spikes of $\cE_{\delta}(\lambda)$, so the total length of a tie $t$ of $\cN_\delta(\lambda)$ is at most $6|\chi(S)|$ times the bound above. Invoking Proposition \ref{prop:ttdefs_comp} now yields a bound on the length of a tie of $\cE_\delta(\lambda)$.
\end{proof}

Thus, if the width is smaller than length of the systole, then $\cE_\delta(\lambda)$ is a train track neighborhood of $\lambda$.

\begin{lemma}[Uniform defining parameter]\label{lem:equittswork}
For any $s>0$, there is a $\delta_{\ref{lem:equittswork}}>0$ such that if $X\in \T_g$ is $s$-thick, $\lambda$ is a geodesic lamination, and $\delta \le \delta_{\ref{lem:equittswork}}$, then $\pi: \cE_\delta (\lambda) \to \tau(X,\lambda,\delta)$ extends to a homotopy equivalence $X\to X$ and the width of $\cE_\delta(\lambda)$ is at most $1/2$.  In particular, $\tau(X,\lambda,\delta)$ is a train track.
\end{lemma}

\begin{proof}
Let $W_{\ref{lem:equittwidth}}$ be our bound on the width of $\cE_\delta(\lambda)$ from above.  As long as $W_{\ref{lem:equittwidth}}\delta<\min\{1/2,s\}$, 
no tie can close up into a (homotopically essential) closed loop, and the width of $\cE_\delta(\lambda)$ is at most $1/2$.
\end{proof}

\section{Train tracks and dual cellulations}\label{sec:dual cellulations}

In this section, we explain how to pass back and forth between (augmented) equilateral train tracks built from a small parameter $\delta$ and $(X,\lambda) \in \PT_g$ and certain ``horizontally convex'' cellulations $\mathsf T$ of $\cO(X,\lambda)$ by saddle connections.  
This dictionary will be useful in the next section, where we use complex weight spaces of (smoothed, augmented) equilateral train tracks as period coordinate charts for strata of quadratic differentials.

\subsection{Cellulations from train tracks}\label{subsec:dualcell}
Let $X\in \T_g$, let $\lambda \in \ML_g$,  let $\delta$ be small enough that $\cE_\delta(\lambda)\subset X$ is an equilateral train track neighborhood whose leaf space $\tau = \tau(X,\lambda,\delta)$ is trivalent, let $\arc_\bullet =\gvarc$ be the visible arc system,  and let $\tau \cup \arc_\bullet$ be the augmented track, considered as a 1-complex with some tangential data.

Since $\tau \cup \arc_\bullet$ is filling
(Lemma \ref{lem:visiblearcsfill}), its dual complex is a cellulation $\mathsf T$.\label{ind:dualcelltott}
Its vertices correspond to the complementary components of $S \setminus (\tau \cup \arc_\bullet)$, which are in turn in bijection with hexagons of $S \setminus (\lambda\cup \arc)$ and hence both the vertices of the spine and the zeros of $\cO(X, \lambda)$.
Some of its edges of $\mathsf T$ are dual to (segments of) branches of $\tau$, while others are dual to visible arcs.

\begin{definition}
    Let $q$ be a quadratic differential. A saddle connection $e$ is \textbf{veering}\label{ind:veering} if any lift of $e$ to the universal cover $\tilde{q}$ spans a singularity-free rectangle embedded in $\tilde{q}$.
    Horizontal and vertical saddle connections are veering by fiat.
\end{definition}

\noindent See \cite{Agol:veering, G:veering, LMT:veering} for the genesis of this terminology. 

For $z\in \CC$ we recall from \cite{shshI} the definition of the following function:\label{ind:z+}
\[[z]_+:= \left\{ \begin{array}{rl}
    z & \arg(z) \in [0, \pi) \\
    -z & \arg(z) \in [\pi, 2\pi)
\end{array}\right.\]
The point of this function is that while periods of a quadratic differential are only defined up to $\pm 1$, the function $[z]_+$ distinguishes a choice of square root for each saddle connection.

\begin{proposition}[Geometric cellulation dual to train track]\label{prop:dual cell veering}
    The edges of the cellulation $\mathsf T$ of $\cO(X,\lambda)$ dual to $\tau \cup \arc_\bullet$ are realized by veering saddle connections.  Moreover, for each non-horizontal edge $e$ of $\mathsf T$, we have 
    \[[\hol_{\cO(X,\lambda)}(e)]_+ = \sigma_\lambda(X)(t_e) +i\lambda(t_e)\]
    where $t_e$ is a tie crossing the segment of the branch of $\tau$ dual to $e$, and for each horizontal edge, we have that $[\hol_{\cO(X,\lambda)}(e)]_+$ is the weight of the arc of $\arc$ dual to $e$.
\end{proposition}

That the quantity $\sigl(X)(t_e)$ is well-defined can be checked using the axioms for shear-shape cocycles (see \S\ref{subsec:shsh} or \cite[\S7]{shshI}; see also the discussion on ``admissible routes''  in \cite[\S14.5]{shshI}).

\begin{remark}
The realization of the visible arc system $\arc_\bullet$ as leaves of $\Ol(X)$ determines the combinatorics of the dual cellulation. A small perturbation of the visible arc system by proper isotopy may determine a different dual cellulation, which is also realized by veering saddle connections; see Corollary \ref{cor:moresaddles}.  However, the veering property of dual saddles may not persist under arbitrary (large) proper isotopy.
\end{remark}

\begin{proof}
Let $b$ be an edge of $\tau \cup \arc_{\bullet}$ with dual edge $e$ joining components of $X\setminus \tau \cup \arc_\bullet$ that have centers $v$ and $w$.
If $b$ is an arc of $\arc_\bullet$, then then $e$ is the dual edge of the spine of $\cO_\lambda(X)$, which is realized as a horizontal saddle connection on $\cO(X,\lambda)$, whose length is exactly the weight of the arc corresponding to $b$.

Otherwise, $b$ comes from a branch of $\tau$. Choose a non-switch tie $t_e$ of $\cE_{\delta}(\lambda)|_b$.
By construction of the equilateral neighborhood $\cE_\delta(\lambda)$, since $t_e$ terminates in the component of $X\setminus \tau \cup \arc_{\bullet}$ containing $v$, the distance along $\lambda$ between $t_e$ and the corresponding basepoint $p_v$ of $v$ is strictly less than $D_\delta$ (by our choice of representatives for $\arc_{\bullet}$). Furthermore, the same is true for some basepoint $p_w$ of $w$.
Using the definition of the equilateral neighborhood $\cE_\delta(\lambda)$, the distance between $t_e$ along $\lambda$ (in the universal cover) and the center of any other component of $X\setminus (\lambda\cup \arc)$ meeting $t_e$ is strictly greater than $D_\delta$.
See Figure \ref{fig:dualbranchmax}.

\begin{figure}[ht]
\includegraphics[width=\linewidth]{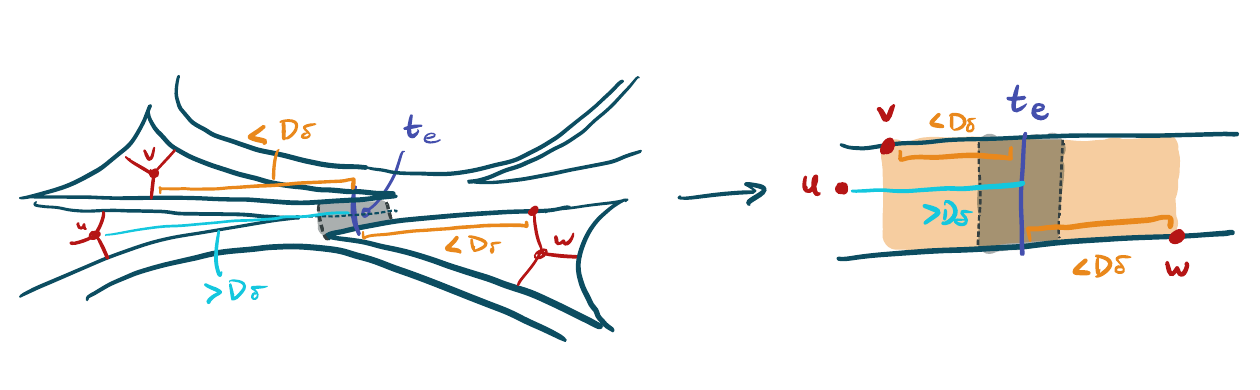}
\caption{Distances along a lamination from a tie to basepoints of incident components of $X\setminus \tau \cup \arc_\bullet$. The data of a geometric train track encodes a cellulation by veering saddle connections.
The highlighted orange rectangle is singularity-free.}
\label{fig:dualbranchmax}
\end{figure}


This implies that there is a band of isotopic leaves of $\Ol(X)$ running from $p_v$ to $p_w$ about $t_e$ such that any leaf of $\lambda$ meeting all of these leaves has length $2 D_\delta$.
There is a map from $\mathcal E_\delta(\lambda)$ to $\cO(X,\lambda)$, obtained by integrating the transverse measure on ties coming from (the measure on) $\lambda$, which is locally isometric along the leaves of $\lambda$. Compare with \cite[Proposition 5.10]{shshI}.
Under this map, the band of isotopic leaves about $t_e$ becomes an embedded rectangle in (the universal cover of) $\cO(X, \lambda)$ of length $2D_\delta$, height $\lambda(t_e)$, and containing the singularities corresponding to $v$ and $w$ on its boundaries.
In particular, there is a nonsingular geodesic segment connecting the corresponding singularities.

Since the horizontal measure in this rectangle is given by intersection with $\Ol(X)$ (equivalently the hyperbolic length measured along $\lambda$) and the vertical measure is given by intersection with $\lambda$, the period of this saddle is as claimed. The sign of the real part of the period of $e$ dual to $b$ is consistent with the sign of $\sigl(X)(t_e)$; see, e.g., \cite[Thoerem 13.13]{shshI}. 
\end{proof}

The proof of the Proposition actually shows the following:

\begin{corollary}\label{cor:moresaddles}
Suppose that the endpoints of a tie of $\cE_\delta(\lambda)$ are within $D_\delta$ of basepoints $p$ and $q$ (as measured along $\lambda$). Then the path joining the centers $u$ and $v$ corresponding to $p$ and $q$ respectively is realized by a veering saddle connection on $\cO(X, \lambda)$.
\end{corollary}

\subsection{Train tracks from cellulations}\label{subsec:ttfromcell}

Given a cellulation $\mathsf T$ of $q$ by saddle connections, we would like to produce a $C^1$ structure at the vertices of the dual graph making it a train track that carries the imaginary folliation of $q$.
For the following class of cellulations, such a choice can be made unambiguously. 

\begin{definition}[Horizontally convex cellulation]\label{def:horizconvex}
Say that a polygon $P \subset \mathbb{R}^2$ is {\bf horizontally convex} if its intersection with each leaf of the horizontal foliation $\ker(dy)$ is connected. \label{ind:horconv}
It is {\bf simply horizontally convex} if for every horizontal edge $e$ of $\partial P$, the intersection of $P$ with the horizontal line containing $e$ is exactly $e$.
Observe that simple horizontal convexity implies that $P$ is allowed to have at most two horizontal edges (one along its top and one along its bottom).

Given a cellulation $\mathsf{T}$ of a quadratic differential $q$ by saddle connections, we say that $\mathsf{T}$ is {\bf (simply) horizontally convex} if each component of $q \setminus \mathsf{T}$ is a (simply) horizontally convex polygon. 
That is, for any cell $P$ of $q \setminus \mathsf{T}$, let $d(P) \subset \mathbb{R}^2$ denote the image of $P$ under a developing map $d:\widetilde{q} \to \mathbb{R}^2$.
This polygon is well-defined up to rotation by $\pi$ and translation, and $q$ is (simply) horizontally convex if each $d(P)$ is.
\end{definition}

This class of cellulations arises naturally. Indeed, the dual cellulations considered above are of this form.

\begin{lemma}[Dual of equilateral tracks horizontally convex]\label{lem:dual track is horizontally convex}
Let $(X, \lambda) \in \PM_g$ and consider the augmentation $\tau \cup \arc_\bullet$ of an equilateral train track by its visible arc system $\arc_\bullet$.
Then the cellulation $\mathsf T$ of $\cO(X, \lambda)$ dual to $\tau \cup \arc_\bullet$ is simply horizontally convex.
\end{lemma}
\begin{proof}
Each polygon of the dual cellulation is dual to a vertex of $\tau \cup \arc_\bullet$. The edges of the cellulation either cross an arc of $\arc_{\bullet}$, in which case they are horizontal, or cross $\tau$, in which case they are not. 
The statement now follows by observing that every vertex of $\tau \cup \arc_\bullet$ has at most four tangential directions: two coming from the tangential structure of $\tau$ and up to two coming from incidences of $\arc_{\bullet}$ with $\tau$, each of which has a unique arc emanating from it. 
Compare Figure \ref{fig:dualltt_smooth}. 
\end{proof}

We can construct a train track out of a (simply) horizontally convex cellulation as follows. The following is a generalization of \cite[\S4.4]{MirzEQ} and of \cite[Construction 10.4]{shshI}.

\begin{construction}[Train track dual to cellulation]\label{constr:dualtt} 
Suppose that $\mathsf{T}$ is a horizontally convex cellulation of a quadratic differential $q$ by saddle connections and let $\mathsf{H} \subset \mathsf{T}$ denote the set of horizontal saddle connections.
Consider the 1-skeleton $\mathsf{T}^*$ of the dual cellulation to $\mathsf{T}$, and let $\mathsf{H}^*$ denote the edges of $\mathsf{T}^*$ dual to the edges of $\mathsf{H}$.
Then in each polygon $P$ of $q \setminus \mathsf{T}$ (identified with its image under a developing map), we assign tangential data at the dual vertex $v_P$ of $\mathsf{T}^*$ so that the arcs which exit $P$ out its left-hand side have the same tangential data at $v_P$ and similarly for the arcs which leave $P$ from its right-hand side.
See Figure \ref{fig:dualltt_smooth}.

This turns $\mathsf{T}^* \setminus \mathsf{H}^*$ into a train track which we denote by $\tau(q, {\mathsf{T}})$.\label{ind:ttdualtocell}
Moreover, the edges of $\mathsf{H}^*$ define an arc system $\arc(q, {\mathsf{T}})$ properly embedded on $S \setminus \tau$ which meets $\tau$ at the vertices of $\mathsf{T}^*$.
\end{construction}

\begin{figure}[ht]
    \centering
    \includegraphics[width=.8\linewidth]{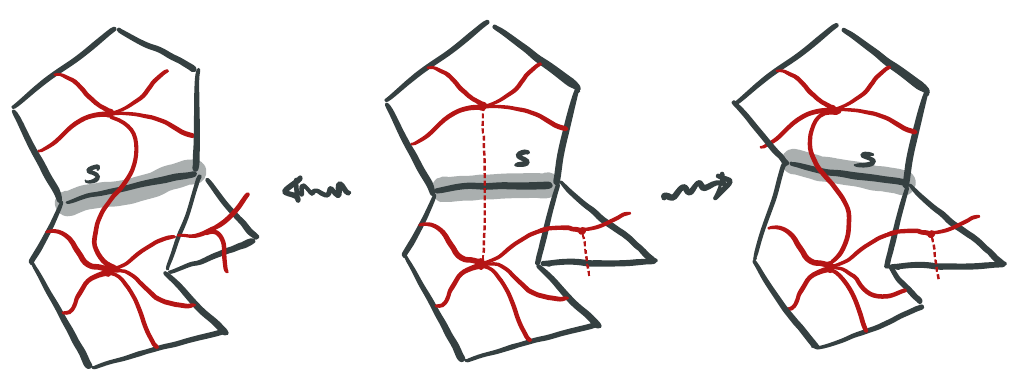}
    \caption{Simply horizontally convex cellulations and their dual train tracks. The tangential data of the edge dual to the saddle connection $s$ depends on the slope of $s$. Dashed red arcs denote arcs dual to horizontal saddle connections.}
    \label{fig:dualltt_smooth}
\end{figure}

Our choice of tangential data is such that this construction is exactly dual to the one from Proposition \ref{prop:dual cell veering}; the proof is just a definition chase.

\begin{lemma}\label{lem:dualtt_equitt}
Let $(X,\lambda)\in \PT_g$ and let $\tau$ be an equilateral train track with visible arc system $\arc_{\bullet}$.
Then if $\mathsf{T}$ denotes the dual cellulation to $\tau \cup \arc_\bullet$, we have that
\[ \tau(\cO(X, \lambda), \mathsf{T}) = \tau.\]
Furthermore, we have that $\arc(\cO(X, \lambda), \mathsf{T}) = \arc_{\bullet}$ as arc systems on $S \setminus \tau$.
\footnote{The combinatorics of incidences to $\tau$ may be different.}
\end{lemma}

\subsection{Equilateral train tracks via unzipping}

There is another method, more intrinsic to the flat geometry of $\cO(X, \lambda)$, by which we can recover the equilateral train track $\tau = \tau(X,\lambda,\delta)$ and the visible arc system  $\arc_\bullet = \arc_\bullet(X,\lambda,\delta)$.
The following construction is well-known to experts.

\begin{construction}[Train track via horizontal unzipping]\label{const:ttfromQD}
Let $q\in \QT$ and $D>0$.  Starting at each zero of $q$, cut along all horizontal separatrices distance $D$, recording all compact horizontal saddle connections $\mathsf H(q,D)$ that are completely cut through.

As long as $D$ is long enough to sever every vertical leaf at least once, the resulting object is a bi-foliated band complex, which is tiled by rectangles comprised of maximal unions of connected vertical segments.
The vertical leaf space $\tau(q,D)$ of this band complex is a graph which inherits a $C^1$ structure at its vertices corresponding to the vertical segments along which three or more rectangles are attached, and the compact horizontal saddle connections $\mathsf H(q,D)$ define a dual arc system $\arc(q,D)$  in the complement of $\tau$ in $q$.
\end{construction}

The following lemma states that the output of  Construction \ref{const:ttfromQD} using singular flat geometry is identical to that of Construction \ref{constr:equitt} using hyperbolic geometry. Its proof is apparent from the definitions.

\begin{lemma}[Unzipping recovers equilateral track]\label{lem:hyperbolic_flat_tts_same}
For any $(X, \lambda) \in \PoT_g$, the train track $\tau(\cO(X,\lambda), D_\delta)$ is isotopic to $\tau(X,\lambda,\delta)$. 
Furthermore, we have that 
$\arc(\cO(X,\lambda), D_\delta)=\arc_\bullet(X,\lambda,\delta)$
as arc systems on $S \setminus \tau$.
\end{lemma}

\section{Train track period coordinate charts}\label{sec:ttperiod coords}

In the previous section we exhibited a duality between augmented train tracks and cellulations.
In this one, we explain how the complex weight space of a ``smoothing'' $\taua$ of an augmented equilateral train track $\tau \cup \arc_{\bullet}$ can be used to give local (period) coordinate charts around $\cO(X, \lambda)$ in its ambient stratum.
With a bit more work, we can extend these to coordinates for a neighborhood of $\cO(X,\lambda)$ in $\QT_g$.

Our philosophy is that that shear-shape coordinates give train track charts for $\PT_g$, and these should be thought of as an analogue of period coordinates for $\QT_g$.

\subsection{Weights and smoothings}\label{subsec:smoothings}
Suppose $\tau$ is a train track; in this paper, $\tau$ is always bi-recurrent.
The \textbf{weight space} $W(\tau)$ of $\tau$ is the linear subspace of $\RR^{b(\tau)}$ cut out by the switch conditions.\label{ind:wtspaces}
That is, each switch $V$ locally separates $\tau$, so any half branch incident to $V$ is on one side or the other, which we give (arbitrary) names incoming and outgoing.
The switch conditions stipulate that the sum of the numbers assigned to the incoming half-branches is equal to the sum of the numbers assigned to outgoing half-branches.
By $W_\mathbb{C}(\tau)$, we mean the complexification of $W(\tau)$; equivalently, the subspace of $\CC^{b(\tau)}$ cut out by the switch conditions.

Let $W^{>0}(\tau)$ denote the open convex cone of strictly positive weights and $W^{\ge0}(\tau)$ its closure.
Any measured lamination carried by $\tau$ defines a non-negative weight system by integrating the measure over fibers of a carying map, and every non-zero $w\in W^{\ge0}(\tau)$ defines, by an unzipping process, a measured geodesic lamination whose support is carried by $\tau$ (see \cite[Construction 1.7.7]{PennerHarer}).
If $\tau$ is maximal, i.e., the complement of $\tau$ consists only of triangles, then $W^{>0}(\tau)$ defines an open set in $\ML_g$.

We say that a train track $\tau'$ is an \textbf{extension} of $\tau$ if some number of branches of $\tau'$ can be removed from $\tau'$ (together with the remaining bivalent switches) such that the result is isotopic to $\tau$. 
In this setting, there is a natural identification of $W(\tau)$ with the subspace of $W(\tau')$ where the branches of $\tau' \setminus \tau$ have $0$ weight.\label{ind:extensions}

In the sections above, we associated to geometric data (either $(X,\lambda)$ plus an auxiliary defining parameter $\delta$ or $q$ and a cellulation $\mathsf{T}$) a train track $\tau$ and an arc system.
It will be convenient in the sequel to combine these data into an extension of $\tau$.
Compare \cite[Construction 9.3]{shshI}.

\begin{definition}[Smoothing augmented train tracks]\label{def:smoothing}
Let $\tau$ be a train track and let $\alpha$ be an arc properly embedded on $S \setminus \tau$ that is not isotopic into a spike.
A {\bf smoothing} of $\tau \cup \alpha$ is an extension of $\tau$ obtained by assigning tangential data to the points of $\alpha \cap \tau$ such that, when viewed from a point inside $\alpha$, either both endpoints turn right onto $\tau$ or both turn left onto $\tau$.\label{ind:smoothing}
This definition can be extended to smooth any disjoint union of nonisotopic arcs.
A smoothing is called {\bf standard} if every arc turns left when it encounters $\tau$.
\end{definition}

The isotopy class of a smoothing depends on the position of the endpoints of $\arc$ on $\tau$, not just on its isotopy class rel $\partial (S\setminus \tau)$.
Given any proper isotopy class of arc system $\arc$ on $S \setminus \tau$, there is a natural equivalence between any two of its standard smoothings given by sliding the feet of $\arc$ along the boundary of $S \setminus \tau$.
The corresponding moves on the train track $\taua$ are all shifts and so their weight spaces are naturally identified \cite[Proposition 2.2.2]{PennerHarer}; furthermore, the subspaces corresponding to $W(\tau)$ are also identified.
We say that any two standard smoothings of $\tau$ and an arc system $\arc$ on $S \setminus \tau$ are {\bf slide equivalent.}\label{ind:slideequiv}
\medskip

Standard smoothings are natural places to record the transverse measure to $\lambda$, shears across it, and the weights of an arc system, all at the same time.
In particular, if $(X, \lambda) \in \PT_g$ and if $\tau$ is a train track snugly carrying $\lambda$, then clearly $\lambda$ defines a weight system on $\tau \prec \taua$.
We proved in \cite[Proposition 9.5]{shshI} 
that if $\arc(X, \lambda)$ is the geometric arc system, then the shear-shape cocycle $\sigl(X)$ may be represented by a weight system on a standard smoothing $\taua$.
The following statement is an instance of this phenomenon in the flat setting, and is a generalization of \cite[Lemma 10.10]{shshI} to a larger class of train tracks. See Figure 15 of that paper, and compare also with the equality between periods and weights in Proposition \ref{prop:dual cell veering}.

\begin{lemma}[Periods are complex wieghts]\label{lem:ttwts_from_cell}
Let $\mathsf{T}$ be a horizontally convex cellulation of a quadratic differential $q$. Let $\tau = \tau(q, \mathsf{T})$ be the dual train track from Construction \ref{constr:dualtt} and let $\taua$ be a standard smoothing of $\tau$ together with the arcs $\arc(q, \mathsf{T})$ dual to the horizontal saddle connections of $\mathsf{T}$.
Then $[\hol(q)]_+$ defines a $\mathbb{C}$-valued weight system on $\taua$ satisfying the switch conditions.
\end{lemma}

\subsection{Train track charts for strata}
Saddle connections are stable as one deforms in a stratum (i.e., the property of any individual path being a saddle connection defines an open subset), so any $q'$ near a fixed $q$ in the same stratum can be cellulated by the same set of saddle connections.
The property of a saddle connection being horizontal is not stable, however, we have carefully chosen our definitions so that simple horizontal convexity persists under small deformations.

\begin{lemma}\label{lem:simplehorizconvex_stable}
Given a simply horizontally convex cellulation $\mathsf{T}$ of a quadratic differential $q$, there is a neighborhood of $q$ in its ambient stratum such that for every $q'$ in that neighborhood, the cellulation $\mathsf{T}$ persists and remains simply horizontally convex on $q'$.
\end{lemma}
\begin{proof}
The property of being simply horizontally convex is a stable condition on polygons.
\end{proof}

If a saddle connection of $\mathsf{T}$ changes from horizontal to not as one deforms $q$, then the dual track from Construction \ref{constr:dualtt} changes by smoothing the corresponding arc (see Figure \ref{fig:dualltt_smooth}). 
Moreover, so long as all non-horizontal edges of $\mathsf{T}$ remain non-horizontal (an open condition), the corresponding edges of the dual train track persist.
We record this fact as follows:

\begin{lemma}[Dual tracks stable in strata]\label{lem:simplehorizconvex_smooth}
There is a neighborhood $B^{\cQ}_{\mathsf{T}}(q)$ of $q$ in its ambient stratum such that for any $q' \in B^{\cQ}_{\mathsf{T}}(q)$, the dual train track $\tau' := \tau(q', {\mathsf{T}})$ is a smoothing of $\tau \cup \arcb$, where $\arcb \subseteq \arc(q,\mathsf T)$.\label{ind:BQT(q)}
\end{lemma}

Together with Lemma \ref{lem:ttwts_from_cell}, this statement allows us to parameterize differentials in $B^{\cQ}_{\mathsf{T}}(q)$ in terms of weight systems on different smoothings of $\tau$.
However, in the sequel it will be convenient to consider these all as weight systems on the same train track.
We therefore introduce a combinatorial move on smoothings that will allow us to do this.
Let $\tau$ be a train track, $\arc$ an arc system on $S \setminus \tau$, and $\taua$ any smoothing of $\tau \cup \arc$. We can {\bf flip} $\taua$ by changing the smoothing direction of any arc $\alpha \in \arc$, so if $\alpha$ turned right when encountering $\tau$ then it turns left in the flip of $\taua$.
\label{ind:flipping}

\begin{lemma}[Flipping weights]\label{lem:flipiso}
Suppose that $\tau$ is a train track and $\alpha$ is an arc on $S \setminus \tau$.
Let $\tau_\alpha$ and $\tau_\alpha'$ denote the two possible smoothings of $\tau \cup \alpha$ (so they are obtained from each other by flipping $\alpha$).
Then there is a natural $\RR$-linear isomorphism between $W_{\CC}(\tau_\alpha)$ and $W_{\CC}(\tau_\alpha')$ that negates the weight on $\alpha$ and restricts to the identity on the common subspace $W_{\CC}(\tau)$.
\end{lemma}
\begin{proof}
This follows from Figure \ref{fig:flip}; observe that at each of its endpoints, flipping the branch corresponding to $\alpha$ reverses which side of the switch at which it appears, so negating the weights means that the switch conditions in the flipped train track are fulfilled.
\end{proof}

\begin{figure}
    \centering
    \includegraphics[width=.7\linewidth]{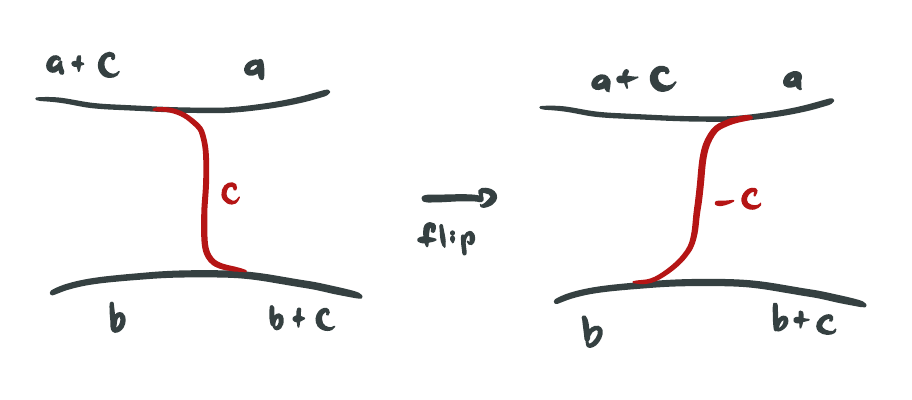}
    \caption{Flipping an edge and negating its weight.}
    \label{fig:flip}
\end{figure}

When an arc system has multiple arcs, flipping each of them is an independent move and preserves the weight system on the remainder of the train track.
In particular, we see that any smoothing of $\tau \cup \arc$ can be flipped to the standard one, inducing an isomorphism between their weight systems.

Recall from \S\ref{subsec:QD bckgrd} that a period coordinate chart near $q$ for its ambient stratum is modeled on $H^1(\widehat q; Z(\widehat q), \CC)^-$, the $-1$-eigenspace for the covering involution on the relative cohomology of the orientation cover $\widehat{q}$.

\begin{proposition}[Train track charts for strata]\label{prop:ttcharts_stratum}
Let $\mathsf{T}$ be a simply horizontally convex cellulation of a quadratic differential $q$ and set $\tau = \tau(q, {\mathsf{T}})$ and $\arc = \arc(q, \mathsf{T})$.
Let $\taua$ denote the standard smoothing of $\tau \cup \arc$, and let $B^{\cQ}_{\mathsf{T}}(q)$ be as in Lemma \ref{lem:simplehorizconvex_smooth}.
Then the $\CC$-linear map 
\[H^1(\widehat q; Z(\widehat q), \CC)^- \to \CC^{b(\taua)}\]
given by evaluating relative cycles dual to branches of $\taua$ composed with $[\cdot ]_+$ is an isomorphism onto the weight space $W_\CC(\taua) < \CC^{b(\taua)}$.
Furthermore,  $\Im [\hol(q')(e)]_+>0$ for every $q' \in B^{\cQ}_{\mathsf T}(q)$ and relative cycle $e$ dual to a branch $b\in \tau$.
\end{proposition}

In particular, if $(X, \lambda) \in \PM_g$ and $\taua$ is the standard smoothing of the augmentation of an equilateral train track by its visible arc system,
then any $q'$ near $\cO(X, \lambda)$ and in the same stratum may be represented by a weight system on $\taua$ with positive imaginary parts on the branches of $\tau$.

\begin{proof}
Lemmas \ref{lem:simplehorizconvex_smooth} and \ref{lem:ttwts_from_cell} combined imply that every $q' \in B^{\cQ}_{\mathsf{T}}(q)$ can be represented as a complex weight system on some smoothing of $\tau \cup \arc$; since the weights correspond to periods, this is also locally a PL homeomorphism.
It therefore remains to show that each weight system near $[\hol(q)]_+$ in $W_{\CC}(\taua)$ corresponds to a unique $q' \in B^{\cQ}_{\mathsf{T}}(q)$.

We deduce this from the fact that the imaginary parts of the weight systems $[\hol(q')]_+$ are always positive.
In particular, given any weight system $w$ on the standard smoothing $\taua$ near $[\hol(q)]_+$, the imaginary parts of the weights on branches of $\tau$ will remain positive, while the imaginary parts of the weights on branches corresponding to arcs of $\arc$ may be of any sign.
If any of these are negative, then we may flip those branches of $\taua$ to arrive at a different smoothing of $\tau \cup \arc$ on which all of the imaginary parts of $w$ are positive; this then corresponds to some $q'$ near $q$.
\end{proof}

We remark that since the real and imaginary parts of these weight systems correspond to the real and imaginary parts of periods, the foliations $\Fol^{ss} \cap \cQ$ and $\Fol^{uu} \cap \cQ$ (see \S\ref{subsec:MF backgrd}) are mapped to the foliations of $W_{\CC}(\taua)$ by weight systems with fixed real and imaginary part, respectively.

\subsection{Breaking up zeros}\label{subsec:nbhd_strata}
We now extend the results of the previous subsection to a neighborhood of the ambient stratum (component) $\cQ$; this will require some details about how strata fit together.

A small neighborhood of a non-principal stratum $\cQ$ in $\QT_g$ can be described by how one breaks up the zeros of $\cQ$, possibly subject to a residue constraint if $\cQ$ consists of squares of abelian differentials.
In particular, for any $q'$ near $\cQ$ we get a finite-to-one correspondence between the zeros of $q'$ and the zeros of some nearby $q\in \cQ$.
In period coordinates adapted to $\cQ'$, the stratum $\cQ$ corresponds to the linear subspace where the short saddles connecting close-by zeros of $q'$ are zero.

Another way to think about the correspondence between zeros $Z(q')$ of $q'$ and zeros $Z(q)$ of $q$ is that we have an identification of $q\setminus Z(q)$  with a subsurface $Y_q \subset q'\setminus Z(q')$, where the complement of $Y_q$ is a set of punctured disks.\label{ind:Yq}
Each component of $q'\setminus Y_q$ contains all the zeros of $q'$ that correspond to a given zero of $q$.
See \cite[Chapter IV, Section 1 and Lemma 4.8]{HubMas} as well as \cite{BCGGM_kdiffs} (to apply the latter reference, one should realize $\cQ$ in the boundary of each incident stratum).
Compare also \cite[Section 5.2]{KZstrata} as well as the discussion of neighborhoods of the principal boundary in \cite[Lemma 9.8]{EMZprincipal} and \cite[Theorem 4]{MZprincipal}.
\medskip 

Recall that the {\bf injectivity radius} $\injrad_q(x)$ at a point $x$ on a flat surface $q$ is the radius of the largest Euclidean ball embedded in $q$ and centered at $x$; in particular, this ball cannot contain any cone points.  The injectivity radius at $z\in Z(q)$ is $0$.\label{ind:flatthickthin}

The {\bf $t$-thick part} is \[\Thick_t(q) = \{x \in q: \injrad_q(x) >t\},\] and the {\bf $t$-thin part} $\Thin_t(q)$ is the complement.
If $t$ is less than half the length of the (flat) systole of $q$, then each component of $\Thin_t(q)$
is a topological disk containing some (positive) number of zeros of $q$.
Moreover, each component of $\Thin_t(q)$
is (locally) geodesically convex and has diameter at most $Ct$, where $C$ is a topological constant depending only on the genus $g$.

Suppose $q_n \to q\in \QT_g$.  Then the singular flat metrics on $q_n$ converge to the singular flat metric on $q$ in the following sense:
For every $\epsilon>0$ and large enough $n$, there are maps $g_n: q\setminus Z_n \to q_n$ that are $(1+\epsilon)$-bi-Lipschitz diffeomorphisms onto their images where $Z_n$ are neighborhoods of $Z(q)$ with $\cap_n Z_n = Z(q)$.  Moreover, $\injrad_{q_n}(x)$ goes to $0$ uniformly in $n$ for $x\notin \im g_n$.
The following lemma provides inverses to our $g_n$'s from above and follows from the proof of Proposition 2.4 in \cite{MW_boundary}.

\begin{lemma}\label{lem:collapsemaps}
For any $q \in \QT_g$, any $\epsilon>0$, and any $t>0$ smaller than half the length of the systole of $q$, there is a neighborhood $U\subset \QT_g$ of $q$ such that for any $q'\in U$, there is a   ``collapse map'' $f:q'\to q$ satisfying:\label{ind:geomcollapse}
\begin{enumerate}
    \item $\Thin_t(q')$ is a disjoint union of disks.
    \item For every $z\in Z(q)$ the set $f\inverse (z)$ is a component of $\Thin_{t/2}(q')$.
    \item The restriction of $f$ to $\Thick_t(q')$ is a $(1+\epsilon)$-bi-Lipschitz diffeomorphism onto its image.
\end{enumerate}
\end{lemma}

In the sequel, we will say that \textbf{$q'$ collapses to $q$} to mean that 
there are suitable constants $\epsilon$ and $t$ and a collapse map $f: q'\to q$ satisfying the conclusions of the Lemma. 

We also want the following topological formulation of collapsing:\label{ind:topcollapse}

\begin{definition}\label{def:topcollapse}
Suppose $\Sigma_1$ and $\Sigma_2$ are finite sets on a closed surface $S$.  
A \textbf{topological collapse map} from $(S,\Sigma_1)$ to $(S,\Sigma_2)$ is a continuous map $f: S \to S$ homotopic to the identity 
such that $f(\Sigma_1) = \Sigma_2$.
Furthermore, we require that
$f$ is supported on a neighborhood of $\Sigma_1$ where each component of the support of $f$ is a disk containing one or more elements of $\Sigma_1$.
\end{definition}

These definitions allow us to extend the notion of having a common cellulation by saddle connections when $q'$ (topologically) collapses to $q$ but lives in a different stratum.

\begin{definition}\label{def:realizerefine}
Let $q\in \QT$ and let $e$ be a saddle connection on $q$.
Suppose that $q' \in \QT$ (topologically) collapses to $q$ with (topological) collapse map $f$.
We say that $q'$ {\bf realizes} $e$ if there is a saddle connection $e'$ on $q'$ such that $f(e')$ is isotopic to $e$ rel the zeros of $q$.\label{ind:rere}
Now suppose that $\mathsf{T}$ is a cellulation of $q$ by saddle connections. We say that a cellulation $\mathsf T'$ of $q'$ by saddle connections {\bf refines} $\mathsf{T}$ if every saddle connection of $\mathsf T$ is realized by saddle connection of $\mathsf T'$.  In other words, $f(\mathsf T')$ is isotopic to a cellulation of $q$ by saddle connections that contains $\mathsf T$.
\end{definition}

The way that one should think of a refinement $\mathsf{T}'$ is that it consists of (possibly multiple) realizations of the saddle connections of $\mathsf{T}$, together with short saddle connections between the zeros of $q'$ that are collapsed to a zero of $q$, along with possibly some extra saddles. 
See \cite[Definition 3.13]{Frankel} for an related notion of how cellulations of $q'$ degenerate to cellulations of $q$ that uses the geometry of the universal curve over Teichm{\"u}ller space.
Another way to think of a refinement is to consider $q' \setminus \mathsf{T}'$ as a set of polygons in the plane: shrinking the short edges of these polygons leaves us with a new collection of polygons that cellulate $q \setminus \mathsf{T}$.  
Compare Figure 
\ref{fig:refinecell}.

\begin{figure}
    \centering
    \includegraphics[width=.8\linewidth]{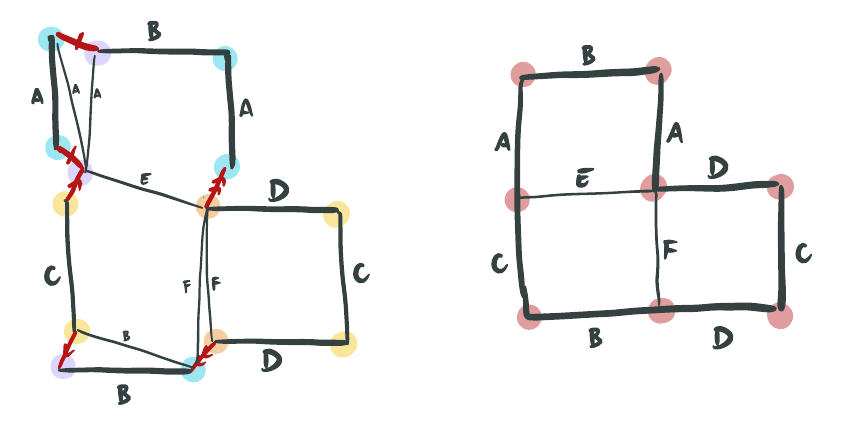}
    \caption{A refinement of a cellulation by squares. The surface $q'$ on the left is in the principal stratum of quadratic differentials, while the surface $q$ on the right is in the minimal stratum and is the square of an abelian differential. The short red edges of $q'$ are those that are sent to the zero of $q$ under a collapse map.}
    \label{fig:refinecell}
\end{figure}

\begin{lemma}\label{lem:cellspersist}
For every $q \in \QT_g$ and every horizontally convex cellulation $\mathsf{T}$ of $q$ by saddle connections, there is an open neighborhood $B_\mathsf{T}(q)$ of $q$ in $\QT_g$ \label{ind:BTq} such that every $q'$ in $B_\mathsf{T}(q)$ has a horizontally convex cellulation by saddle connections $\mathsf{T}'$ that refines $\mathsf{T}$.

Let $\tau'$ denote the dual train track $\tau(q', \mathsf{T}')$. Then $B_\mathsf{T}(q)$ can also be taken such that after removing any branches of $\tau'$ dual to the short edges of $\mathsf{T}'$ and collapsing a finite union of ``short'' branches, the resulting train track is isotopic to a smoothing of $\tau(q, \mathsf{T}) \cup \arcb$, where $\arcb \subset \arc(q, \mathsf{T})$.
\end{lemma}

In particular, the complex weight space of the dual of $\mathsf T$ on $q$ may naturally be identified with a subspace of the complex weight space of the dual of $\mathsf T'$ on $q'$.

\begin{proof}
Given $q$ and $\mathsf T$, take $t$ much smaller than the length of the systole of $q$ and the minimal distance between any saddle connection $e$ of $\mathsf T$ and any zero of $q$ not contained in $\partial e$.
     Provided that $\epsilon$ is small enough, 
      every saddle $e$ of $\mathsf T$ in a $(1+\epsilon)$-bi-Lipschitz deformation of $\Thick_{t/2}(q)$ is contained in a band of Euclidean segments joining the boundary components corresponding to $\partial e$.
    Thus if $U$ is a neighborhood of $q$ guaranteed by Lemma \ref{lem:collapsemaps} with parameters $t$ and $\epsilon$ chosen as above,
    then every saddle connection in $\mathsf T$ is realized on $q'\in U$. 

    We would now like to argue that we can find a horizontally convex cellulation $\mathsf T'$ of $q'\in U$ refining $\mathsf T$.  For this,
    we enumerate the zeros $Z(q) = \{z_1, ..., z_k\}$ and choose $z_i'\in Z(q')$ satisfying $f(z_i') = z_i$ for all $i = 1, ..., k$.
    Let $e$ be a saddle connection of $\mathsf T$ with $\partial s = \{z_i, z_j\}$.  In the universal cover of $q'$, there is a unique geodesic path joining lifts of $z_i'$ and $z_j'$ whose projection $[z_i', z_j']$ to $q'$ satisfies that $f([z_i', z_j'])$ is isotopic rel marked points to $e$.  
    Then $[z_i',z_j']$ is a union of short saddle connections joining zeros in $f\inverse(z_i)$, a saddle connection $e'$ realizing $e$, and short saddle connections joining zeros in $f\inverse(z_j)$.
    
    Using the fact that the metric on $q'$ is locally CAT(0) and a small deformation of $q$, we know that the union of saddle connections in $q'$ obtained in this way do not cross.
    Indeed, if two saddles $e_1'$ and $e_2'$ meet in their interior, then they must do so near a component of $\Thin_t(q')$, because $f\inverse(e_i)$ is Hausdorff close to $e_i'$ in $\Thick_t(q')$, and $e_1$ does not cross $e_2$.  
    Then $e_1$ and $e_2$ are subsegments of geodesic paths joining $z_{i_0}'$ to distinct  $z_{i_1}'$ and $z_{i_2}'$, and  $e_1\cap e_2$ is near $z_{i_0}'$.  
    This gives two homotopic geodesic paths from $z_{i_0}'$ to $e_1\cap e_2$, but since the metric on $q'$ is locally CAT(0), this means that $e_1 = e_2$.  
    
    Thus we may add in additional short saddle connections as needed to complete this union to a cellulation by saddle connections $\mathsf T'$ that refines $\mathsf T$.
    Adding in even more saddles as needed, we can make $\mathsf T'$ horizontally convex, (e.g., because every triangulation by saddle connections is horizontally convex).
    By taking the neighborhood smaller as necessary, we can further ensure that every realization $e'$ in $\mathsf T'$ of a non-horizontal saddle connection $e$ is not horizontal.
    \medskip
    
    Let us examine how the dual train track of $\mathsf T'$ in $q'$ is related to the dual train track of $\mathsf T$ on $q$.
    Let $\{P_t\}_{t\ge 0}$ be a continuous deformation of horizontally convex $n$-gons with distinct vertices in the plane.  Then the isotopy class of the dual train track carrying the horizontal foliation from Construction \ref{constr:dualtt} changes only if the slope of a saddle connection changes from $0$ to nonzero (or vice versa) at some time $t$.

Suppose first that we are in the special case where $\{P_t\}_{t>0}$ are a continuous family of horizontally convex $n$-gons that degenerate to an $(n-k)$-gon at $t=0$.
In this case, there are $n-k$ distinguished edges of $P_t$ that correspond to edges of $P_0$.
If we further assume that none of these edges are horizontal for any $t \ge 0$, then we see that the dual train track to $P_0$ is obtained from the dual track of $P_t$ by removing the $k$ branches dual to the small edges of $P_t$ that get collapsed.

More generally, even if the $P_t$ have no horizontal edges, the situation can be more complicated, as we may need to add in extra edges to ensure the $P_t$ are horizontally convex.
Thus, we are led to considering the dual train tracks to a continuous family of $n$-gons $\{P_t\}_{t>0}$ that degenerate to a horizontally convex $(n-k)$-gon at $t=0$ and for $t>0$, a continuous choice of cellulation of $P_t$ into  horizontally convex polygons.
In this case, a similar statement to the previous paragraph is true, except that we must also collapse the branches dual to the extra interior edges of the cellulation of $P_t$ in order to obtain the dual track to $P_0$ from the dual track of $P_t$. 
See Figure \ref{fig:dual_tt_deformation}.

\begin{figure}[hb]
    \centering
    \includegraphics[width=.8\linewidth]{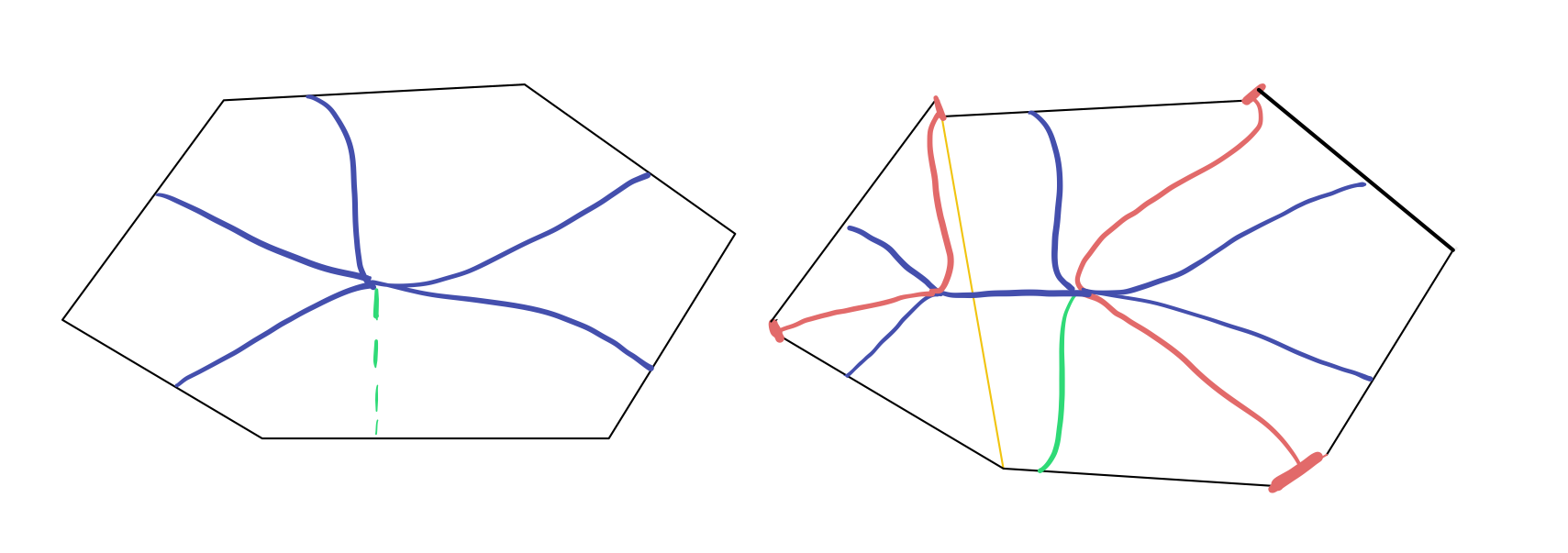}
    \caption{On the left, a horizontally convex $6$-gon $P_0$  with dual train track in blue and green arc dual to a horizontal edge. On the right, a continuous deformation $P_t$ that degenerates to $P_0$.  The red edges on the right collapse to vertices of $P_0$, and removing the dual branches results in a track that is slide equivalent to a smoothing of the track on the right. The yellow edge ensures that $P_t$ is a union of horizontally convex polygons.}
    \label{fig:dual_tt_deformation}
\end{figure}

We are left to consider the case where there is a horizontal saddle connection in $\mathsf T$ dual to an arc of $\arc(q,\mathsf T)$.
In this case, one of the polygons $P_0$ has a horizontal edge but the edge may not remain horizontal in $P_t$ for $t >0$. The dual train track $\tau(q', \mathsf T')$ will then contain a branch corresponding to (a smoothing of) the set of arcs $\arcb\subset \arc$ whose corresponding saddle connections do not remain horizontal.
\end{proof}

\subsection{Breaking up zeros horizontally}
In general, the horizontal foliations of those $q'$ obtained by breaking up the zeros of $q \in \cQ$ will have new trajectories that separate the newly-created zeros.
Equivalently, the dual train track to the refinement $\mathsf{T}'$ is an extension of the dual track to $\mathsf{T}$. 
However, if zeros of $q$ are broken up such that they only have horizontal saddle connections between them, the resulting zeros of $q'$ cannot be separated by horizontal trajectories, 
and so the dual track to $\mathsf{T}'$ does not have new branches, only new arcs (see Lemma \ref{lem:ttsaroundstrata} just below).
We collect all such $q'$ in a neighborhood of $\cQ$:

\begin{definition}\label{def:Q star}
Let $q' \in B_\mathsf{T}(q)$ and let $f$ be a (geometric) collapse map from $q'$ to $q$.
We say that $q' \in \cQ^*$ \label{ind:Q*} if for every pair of zeros $z_1'$ and $z_2'$ of $q'$ such that $f(z_1') = f(z_2')$, there is a path $\gamma$ of horizontal saddle connections connecting $z_1'$ to $z_2'$ such that $f(\gamma)=z$.
\end{definition}

\begin{remark}
It is useful to think of $\cQ^*$ as thickening $\cQ$ slightly in the unstable direction. Indeed, by our previous work in \cite{shshI}, for any leaf $\mathcal F$ of $\Fol^{uu}$, a neigborhood of $\cQ \cap \mathcal F$ in $\mathcal F$ can be described by breaking up zeros horizontally and changing the real parts of periods.
\end{remark}

A relatively open neighborhood of $q \in \cQ$ inside $\cQ^*$ can be covered by a union of reasonably nice sets.
Indeed, Lemma \ref{lem:cellspersist} above gives a family of period coordinate charts about $\cQ$. In these charts, $\cQ^*$ is a subspace cut out by 
$\Im(\hol_{q'}(e))=0$ for every short saddle connection $e$ obtained by breaking up a zero of $q$.
Alternatively, from the point of view of \cite{BCGGM_kdiffs}, $\cQ$ lies in the boundary of the principal stratum and a neighborhood of $\cQ$ can be described in terms of plumbing on meromorphic differentials in $\CC P^1$ (i.e., breaking up zeros).
The subspace $\cQ^*$ is then the subset of this neighborhood obtained by plumbing on real-normalized differentials in a way that the corresponding periods on the plumbed surface are still real. Thus we can locally describe points of $\cQ^*$ in terms of points in the stratum $\cQ$, tuples of real-normalized differentials on $\CC P^1$ (one for each zero of $q \in \cQ$), and compatible plumbing data.

In \cite[Section IV]{HubMas}, Hubbard and Masur prove that $\cQ^*$ is locally a submanifold of $\QT_g$ by analyzing the map from a neighborhood of $q \in \cQ$ to the space of possible ways to break up a zero (see Lemma 4.2 and Proposition 4.11 therein).
\footnote{Technically, \cite{HubMas} concerns a different set $V = f^{-1}(\prod E_k)$ which is a subset of $U \times \RR^n$, where $U \subset \QT_g$ is a small open set about $q$. This extra factor of $\RR^n$ is used to deal with the ``perverse'' tangential structure of $\cQ^*$, equivalently, with the fact that the space of transverse cocycles to the horizontal foliation is not locally constant. Since we will not concern ourselves with the tangential structure to $\cQ^*$ we have omitted this subtlety.}
Since $\cQ^*$ is a submanifold, we know {\em a posteriori} that the sets described above (either in terms of periods or plumbing) glue up nicely. This is reflected in Proposition \ref{prop:ttcharts_Q*} below, which demonstrates a natural polyhedral structure on $\cQ^*$ that agrees with the simplicial structure investigated in \cite[Chapter III, \S 6]{HubMas}.
\medskip 

From the definition (and Lemma \ref{lem:cellspersist}), it follows that a cellulation of $q \in \cQ$ can be refined to a cellulation of $q' \in \cQ^*$ by only adding in short horizontal saddles; this places strong restrictions on how the dual train track and arc system can look.

\begin{lemma}\label{lem:ttsaroundstrata}
Let $\mathsf{T}$ be a simply horizontally convex cellulation of a quadratic differential $q$ and set $\tau = \tau(q, {\mathsf{T}})$ and $\arc = \arc(q, \mathsf{T})$.
Then there is a relatively open neighborhood $B^*_\mathsf{T}(q)$\label{ind:BTq*} of $q$ in $\cQ^*$ such that 
\begin{enumerate}
    \item any $q' \in B^*_\mathsf{T}(q)$ has a horizontally convex cellulation $\mathsf{T'}$ that refines $\mathsf{T}$,
    \item the dual train track $\tau(q', \mathsf{T}')$ is a smoothing of $\tau \cup \arcb$ for $\arcb \subseteq \arc$, and 
    \item the arc system $\arc(q', \mathsf{T}')$ contains $\arc \setminus \arcb$.
\end{enumerate}
\end{lemma}
\begin{proof}
Take $B_{\mathsf T}^*(q) = B_{\mathsf T}(q) \cap \cQ^*$, where $B_{\mathsf T}(q)$ is obtained as in Lemma \ref{lem:cellspersist}.
We are only allowing ourselves to break up higher order zeros of $q$ horizontally and every complementary $(n-k)$-gon $P_0$ of $\mathsf T$ in $q$ is strictly horizontally convex; let $\{P_t\}_{t>0}$ be a continuous family of  $n$-gons $\{P_t\}_{t>0}$ degenerating to $P_0$ such that that all of the $k$ ``short'' edges are horizontal.
Then it is not difficult to see that $P_t$ is horizontally convex for small enough values of $t$,
and in particular, we do not have to add in extra diagonals to $P$.

Our description of the dual track and arc system from Construction \ref{constr:dualtt} follow from considering the dual notions to being horizontal.
Since we have added in no new diagonal edges, there is a correspondence between the cells of $\mathsf{T}'$ and the cells of $\mathsf{T}$, hence a correspondence between the switches of their dual train tracks.
Since all of the new (short) edges in $P_t$ are horizontal, the corresponding dual objects are in $\arc(q', \mathsf T') \setminus \arc(q, \mathsf T)$.
When a horizontal edge of $\mathsf T$ with dual arc $\alpha \in \arc(q,\mathsf T)$ becomes non-horizontal in $q'$, the corresponding dual branch of $\tau(q',\mathsf T')$ is a smoothing of $\alpha$, which is what we wanted.
\end{proof}

Combined with our discussion of flipping smoothings, we get the following extension of Proposition \ref{prop:ttcharts_stratum}.

\begin{proposition}\label{prop:ttcharts_Q*}
Any $q' \in B^*_{\mathsf{T}}(q)$ may be represented  as a complex weight system on a standard smoothing of $\tau \cup \arc'$, where $\arc'$ is some extension of $\arc$. 

Moreover, consider the closed $\RR_{\ge 0}$-cone in $W_{\CC}(\taua')$ cut out by the condition that the weights of $\arc'\setminus \arc$ must be real and nonnegative.
Then recording periods induces a PL map from a neighborhood of $[\hol (q)]_+$ in this cone to $B^*_\mathsf{T}(q)$, and $B^*_\mathsf{T}(q)$ is covered by the images of such.
\end{proposition}

\part{Continuity of orthogeodesic foliations}\label{part:continuity}

In the next three sections, we prove that the equilateral train tracks and dual cellulations constructed above are locally constant in $X$ and $\lambda$ (Propositions \ref{prop:ttstable} and \ref{prop:cellstable}, respectively).
These results give us the framework to state a quantified version of our continuity result in terms of weight systems on train tracks, equivalently, in terms of period coordinates for the corresponding quadratic differentials. See Theorem \ref{thm:contcell}.
The basic geometric estimates we develop to prove these Propositions are also the main ideas needed to prove the Theorem.

Throughout this part of the paper, we fix a base hyperbolic surface $X$ and a chain-recurrent lamination $\lambda$. We use $\arc$ to denote the (unweighted) geometric arc system of $X$ with respect to $\lambda$, that is, the multi-arc on $X \setminus \lambda$ consisting of (isotopy classes of) leaves of the orthogeodesic foliation.

For functions $f_1$ and $f_2$, we say $f_2= O_s(f_1)$ to mean that $f_2= O(f_1)$ where the implicit constant depends only on the systole $s$ of the base surface $X$ on which we are making our estimates.\label{ind:Os}

\section{Persistent arc systems}\label{sec:persistent}
In this section, we analyze how the arc system $\arc(X, \lambda)$ changes under small deformations of the pair.
This is our first step in building a common framework in which to compare the shear-shape cocycles for nearby pairs.
Along the way, we will also prove an estimate on how the weights of the arcs of $\arc(X, \lambda)$ change, which is the main geometric estimate needed to prove continuity of the ``shape'' part of shear-shape coordinates.

One of the main obstacles we must overcome is determining how to actually compare arc systems on subsurfaces with different topological types;
this will be accomplished with the aid of the combinatorial structure of train track carrying maps.

\subsection{(In)visible arcs}\label{subsec:invisarcs}
In order to discuss how the arcs of $\arc(X, \lambda)$ vary as $(X, \lambda)$ does, we first need to understand how the complementary subsurface $X \setminus \lambda$ varies.
Of course, its topological type depends on that of $\lambda$, and since the set of multicurves is dense in $\mathcal{GL}^\text{cr}$, the topological type of $X \setminus \lambda$ can vary wildly as the support of $\lambda$ varies in the Hausdorff topology.

However, when one also remembers the {\em geometry} of the complementary subsurface there is a stronger correspondence between $X \setminus \lambda$ and $X \setminus \lambda'$ for $\lambda'$ Hausdorff close to $\lambda$.
For example, if $\lambda$ is  maximal and chain-recurrent, then $X \setminus \lambda$ is a union of ideal triangles. 
If $\gamma$ is a non-separating simple closed curve that is very close to $\lambda$ in the Hausdorff metric on $X$, then $X \setminus \gamma$ is a surface of genus $g-1$ with two boundary components that looks like a union of nearly-ideal triangles, glued along short arcs corresponding to their vertices.
Compare with \cite[Lemma 3.5]{spine}.
Moreover, at the appropriate scale $\delta$, there is a correspondence between the triangles of $X \setminus \lambda$ and those of $X$ minus a (uniform/equilateral) $\delta$ neighborhood of $\gamma$.

We now establish a general version of this correspondence and develop language to discuss it in detail.

\para{Invisible arcs}
We recall from Definition \ref{def:geometric visibility} that if $\cE$ is an equilateral train track neighborhood of $\lambda$, then some segments of $\Ol(X)|_{\cE}$ may be isotopic to arcs of $\arc$.
These segments are collapsed when one takes the quotient to the leaf space, and the corresponding arcs are said to be invisible with respect to $\cE$.
It is immediate from the definitions that the components of $X \setminus \cE$ are in bijection with the pieces obtained by cutting $X \setminus \lambda$ along the invisible arcs.

In order to use this observation to set up a correspondence between pieces of subsurfaces of $X \setminus \lambda$ and those of $X' \setminus \lambda'$ for nearby pairs, we must give a topological reformulation of invisibility.
Indeed, as we vary $\lambda$ a fixed neighborhood $\cE \subset X$ may no longer be foliated by segments of $\cO_{\lambda'}(X)$, and as we vary $X$ there is not even a consistent neighborhood to which to refer.

\begin{remark}
In fact, we show in Proposition \ref{prop:ttstable} that for close enough pairs $(X', \lambda')$, the equilateral train tracks $\tau(X, \lambda, \delta)$ and $\tau(X', \lambda', \delta)$ are isotopic, so the new topological definition of invisible arcs given below will eventually be superseded by our original geometric one.
However, the proof of Proposition \ref{prop:ttstable} requires a correspondence between arc systems (Proposition \ref{prop:persistent}), which in turn is phrased in terms of (in)visible arcs.
The topological definition of invisible arcs is therefore necessary to avoid circularity in our argument.
\end{remark}

We first record a useful definition which captures one of the salient features of geometric train tracks.

\begin{definition}\label{def:fullcarry}
A train track $\tau$ {\bf fully} carries\label{ind:fully} $\lambda$ and we write $\lambda \preceq \tau$ if any carrying map takes $\lambda$ onto $\tau$.
\end{definition}

This condition can also be phrased in terms of weight systems on $\tau$.
If $\tau$ is bi-recurrent and $\lambda$ supports a transverse measure, then $\lambda \preceq \tau$ if and only if the transverse measure lies in the interior of the positive cone $W^{>0}(\tau)$.
For general chain-recurrent $\lambda$, the density of multicurves in $\GLcr$ implies that $\lambda$ is fully carried on $\tau$ if and only if it is the Hausdorff limit of measured laminations fully carried on $\tau$.

If $\lambda \preceq \tau$, then each branch of $\tau$ corresponds to two ``outermost'' boundary geodesics of $\lambda$ (which are possibly the same, in case $\lambda$ has an isolated leaf).
Indeed, given any carrying map $\lambda \rightarrow \tau$, the preimage of a small transversal to a branch is a transverse arc to $\lambda$, and the corresponding boundary geodesics are the first and last leaves of $\lambda$ encountered by this arc.
Compare with the discussion of proto-spikes and spikes of $\tau$ appearing in \S\ref{subsec:geominvis}.

As demonstrated by the following lemma, the property of being fully carried by an (equilateral) train track is stable as one varies $\lambda$ in the Hausdorff metric.
Its (easy) proof is left to the reader.

\begin{lemma}\label{lem:Hausdfullcarry}
If $\lambda'$ is Hausdorff close enough to $\lambda$, then $\lambda'$ is fully carried on $\tau(X,\lambda, \delta)$.
\end{lemma}

We therefore see that $\lambda, \lambda' \preceq \tau$ is a combinatorial analogue of $\lambda$ and $\lambda'$ being Hausdorff-close.
Using this definition, we can now give a topological notion of invisible arcs using the carrying map $\lambda \rightarrow \tau$.

\begin{definition}[Topologically invisible arcs]\label{def:invisarcs}
Suppose that $\lambda \preceq \tau$. 
Then the preimages of points in $\tau$ by any carrying map are arcs which run between consecutive boundary leaves of $\lambda$, which break up into finitely many proper isotopy classes of arcs on $S \setminus \lambda$.
Throwing out those which are isotopic into a spike, we are left with the 
{\bf invisible arc system} of $\lambda$ with respect to $\tau$, denoted by $\arc_\circ^\tau (\lambda)$ or $\arc_\circ$ if the context is clear.\label{ind:topinvis}
\end{definition}

This definition is designed so that if $\lambda \preceq \tau$, then there is a natural identification of $S \setminus (\lambda \cup \arc_\circ)$ with $S \setminus \tau$.
This allows us to identify the pieces of $S \setminus (\lambda \cup \arc_\circ^\tau(\lambda))$ and $S \setminus (\lambda' \cup \arc_\circ^\tau(\lambda'))$ whenever $\lambda'\preceq \tau$.

Moreover, if $\cE$ is an equilateral neighborhood of $\lambda$ and $\tau$ the corresponding train track, then it follows directly from the definitions and the structure of the collapse map that this notion of invisible arcs agrees with the one from Definition \ref{def:geometric visibility}.
Thus, if a leaf in the isotopy class of $\arc(X, \lambda)$ is short enough compared to the defining parameter $\delta$ of an equilateral train track, then it is contained in the invisible arc system $\arc_\circ^\tau(\lambda)$.

We now prove a converse: if an arc is invisible with respect to a geometric train track then it must be short.
In particular, the (topologically defined) invisible arc system is actually a part of the natural (geometric) arc system recording the complementary hyperbolic structure.

\begin{lemma}[Invisible arcs geometric]\label{lem:invisiblesubarc}
Let $\tau = \tau(X, \lambda, \delta)$ be an equilateral train track.
Then for any $X'\in \T_g$ close to $X$ and any $\lambda'\in \GLcr$ Hausdorff close enough to $\lambda$, we have that $\arc_\circ^\tau(\lambda') \subset \arc(X', \lambda')$.
\end{lemma}

In particular, the pieces of  $S \setminus (\lambda' \cup \arc_\circ^{\tau}(\lambda'))$ are actually realized as unions of hexagons of $X' \setminus (\lambda'\cup \arc(X',\lambda'))$.

\begin{proof}
We prove this first for a fixed hyperbolic metric.
Consider an invisible arc $\alpha$ connecting two boundary geodesics $g$ and $h$ of $\lambda'$ on $X$. 
Even though the orthogeodesic representative of $\alpha$ may not be contained in the $\delta$--equilateral neighborhood of $\lambda$, Lemma \ref{lem:equittwidth} implies that there is a path of length at most $W_{\ref{lem:equittwidth}}\delta$ connecting $g$ and $h$ that is otherwise disjoint from $\lambda'$, where $W_{\ref{lem:equittwidth}}$ depends only on the systole of $X$. 
Thus, the orthogeodesic representative of $\alpha$ has length at most $W_{\ref{lem:equittwidth}}\delta$, so by Lemma \ref{lem:closethenarc} must appear in the geometric arc system.

We conclude by noting that this argument does not use the geometry of $\lambda$ on $X$ in any particular way: we only invoke the Hausdorff closeness of $\lambda$ and $\lambda'$ and the bound on width of equilateral neighborhoods.
Taking $X'$ close enough to $X$ allows us to use the same bounds and therefore the same proof.
\end{proof}

\para{Visible arcs}
In Section \ref{subsec:geominvis}, we defined the visible arc system with respect to an equilateral neighborhood $\cE$ to be the complement of the invisible arcs.
More generally, if $\lambda$ is fully carried on some train track $\tau$ then we can also define the {\bf visible arc system} with respect to $\tau$ in the same way:\label{ind:topvis}
\[\arc_\bullet^\tau(X,\lambda) := \arc(X, \lambda) \setminus \arc_\circ^\tau(\lambda).\]
If context is clear we will sometimes denote the visible arc system just by $\arc_{\bullet}$.
When $\tau = \tau(X, \lambda, \delta)$ is an equilateral train track, since our two notions of invisible arcs agree (see the discussion right after Definition \ref{def:invisarcs}), so do our two notions of visible arcs. Thus we always have that have $\arc_\bullet^\tau(X,\lambda) = \gvarc$.

When $\tau$ is an equilateral train track and $\arc_\bullet$ is the visible arc system for $(X, \lambda)$ with respect to $\tau$,
Lemma \ref{lem:visiblearcsfill} implies that $\arc_\bullet$ specifies a filling arc system on $S \setminus \tau$.
Then for nearby pairs $(X', \lambda')$, Lemma \ref{lem:invisiblesubarc} allows us to identify $\arc_\bullet$ with an arc system on $S \setminus \lambda'$.
Our goal in the next subsection is to show that these arcs are geometric for the pair $(X',\lambda')$.

\para{Configurations of geodesics}
The setup from above allows us to compare isometry classes of finite configurations of geodesic leaves of $\tlambda\subset \tX$ with corresponding configurations in $\tlambda'\subset \tX'$, when $(X',\lambda')$ is close enough to $(X,\lambda)$.
Our discussion of visible arcs above gives us a dictionary between the components of $S\setminus (\tau\cup \arc_\bullet)$ and components of $X'\setminus (\lambda'\cup \arc_\circ^\tau(\lambda') \cup \arc_\bullet)$.

We record below some notation for the tuples of geodesics of $\lambda'$ specified by $\tau \cup \arc_\bullet$; we emphasize that the combinatorial data of $\tau \cup \arc_\bullet$ is specified by the geometry of the pair $(X, \lambda)$ and the parameter $\delta$.

\begin{definition}[Boundary configurations]\label{def:configurations of geodesics}
Let $P$ be a component of $S\setminus (\tau\cup \arc_\bullet)$.
Note that $P$ is simply connected (Lemma \ref{lem:visiblearcsfill}); choose a lift $\widetilde P$ to $\tS$.\label{ind:config}
For $(X',\lambda')$ as above, denote by $G_P(X',\lambda')$ the corresponding configuration of boundary geodesics of $\tlambda'$ in $\tX'$.
\end{definition}

Mainly, we will be interested in comparing  $G_P(X',\lambda)$ directly to $G_P(X',\lambda')$ in $\tX'$.

\subsection{A correspondence between hexagons}
We now analyze how the visible arc system varies as we vary $X$ and $\lambda$. For the rest of this subsection, we choose once and for all a $\delta \le \delta_{\ref{lem:equittswork}}(\sys(X))$
such that the equilateral train track $\tau:=\tau(X, \lambda, \delta)$ is trivalent
and set $\arc_{\bullet}:= \arc_\bullet^\tau(X,\lambda) = \gvarc$.
The reader may find it helpful to restrict to the case where $\delta$ is small enough such that $\tau$ is snug; this implies that every arc of $\arc$ is visible, which slightly simplifies the scenario at hand, but is still useful.

If there is a vertex of the spine of $(X,\lambda)$ of valence $4$ or higher, then most deformations of the pair will break up this vertex, introducing new arcs. Two different deformations may introduce arcs that cross one another.
See Figure \ref{fig:persistent}.
We must therefore work directly with the nearly-equidistant configurations of geodesics arising from perturbing an equidistant one.

Recall from 
\S\ref{subsec:arc systems}
that for any arc $\alpha \in \arc(X, \lambda)$, its weight $c_\alpha(X, \lambda)$ is the length of the projection to $\lambda$ of the packet of leaves of $\Ol(X) \setminus \lambda$ isotopic to $\alpha$.

\begin{proposition}[Persistent arcs]\label{prop:persistent}
For every $(X, \lambda)$, there are $M_{\ref{prop:persistent}}>0$ and  $\zeta_{\ref{prop:persistent}} >0$ such that the following hold. 
For every $\zeta \le \zeta_{\ref{prop:persistent}}$, there is a neighborhood $B_{\ref{prop:persistent}}(\zeta)$ of $X$ such that for any
$(X', \lambda')$ with $X' \in B_{\ref{prop:persistent}}(\zeta)$ and $d_X^H(\lambda, \lambda') < \zeta$, we have that:
\begin{enumerate}
    \item $\arc_\bullet':= \arc_\bullet^\tau(X', \lambda')$ contains $\arc_\bullet$.
    \item For every arc $\alpha \in \arc_\bullet$,
\[\left|c_{\alpha}(X, \lambda) - c_{\alpha}(X', \lambda') \right| \le M_{\ref{prop:persistent}}\zeta^{3/2}.\]
    \item For every $P\subset S\setminus (\tau\cup \arc_\bullet)$, the configuration $G_P(X',\lambda')$ is $M_{\ref{prop:persistent}}\zeta^{3/2}$-equidistant.
\end{enumerate}
\end{proposition}
We note that $\zeta_{\ref{prop:persistent}}$ can be arbitrarily small even for thick pairs $(X, \lambda)$; this follows because arcs can have arbitrarily small weight when $\cO(X, \lambda)$ is close to a non-principal stratum.
\begin{remark}
The exponent $3/2$ is an arbitrary number (strictly) between $1$ and $2$.
\end{remark}

The proof of the Proposition is essentially a combination of the results from Section \ref{sec:triplecenters} about centers and basepoints of triples, used to prove the result for a fixed $X$, together with the structure theory of shear-shape coordinates, used to bootstrap to variable $X$.
A technical complication is that we need uniformity in the first step in order for our bootstrapping argument to apply to a neighborhood of $(X, \lambda)$.

\begin{lemma}\label{lem:fixX persistent}
There is a $\zeta_{\ref{lem:fixX persistent}}>0$ depending only on the minimum weight of an arc of $\arc(X, \lambda)$ and the systole $s$ of $X$ such that for any $\lambda'$ with
$d_X^H(\lambda, \lambda') < \zeta_{\ref{lem:fixX persistent}}$, the following hold:
\begin{enumerate}
    \item $\arc_\bullet^\tau(X, \lambda')$ contains $\arc_\bullet$.
    \item For every arc $\alpha \in \arc_\bullet$, we have that 
\[\left|c_{\alpha}(X, \lambda) - c_{\alpha}(X, \lambda') \right| = O_s(d_X^H(\lambda, \lambda')^2).\]
    \item For every $P\subset S\setminus (\tau\cup \arc_\bullet)$, the configuration $G_P(X,\lambda')$ is $O_s(d_X^H(\lambda,\lambda')^2)$-equidistant.
\end{enumerate}
\end{lemma}
\begin{proof}
Throughout our proof, we will set $\zeta := d_X^H(\lambda, \lambda')$ for brevity.

We begin by proving that $\arc_\bullet$ and $\arc_\bullet^\tau(X, \lambda')$ are disjoint so long as $\zeta$ is small enough.
Suppose towards contradiction that $\alpha$ crosses some arc $\beta\in \arc_\bullet^\tau(X,\lambda')$.
We lift the situation to $\tX$, where lifts $\widetilde\alpha$ and $\widetilde\beta$ intersect in a point, each joining  a pair of geodesics corresponding to boundary leaves of both $\tlambda$ and $\tlambda'$.  
Let $G$ and $G'$ denote the $4$-tuples of boundary geodesics in $\tlambda$ and $\tlambda'$, respectively.
By Lemma \ref{lem:equittwidth} and the construction of equilateral train tracks, each geodesic of $G$ fellow-travels a geodesic of $G'$ at distance $W_{\ref{lem:equittwidth}}\zeta$ along a segment of length at least 
\[2 D_{\zeta} \approx 2\log(1/\zeta) - 2\log 2.\]
We conclude that $G$ is $O_s(\zeta)$-close to $G'$ on a ball of radius at least $D_\zeta$ around any center of $G$.

By Lemma \ref{lem:radius_bounded}, the distance between $G$ and any center of a triple of $G$ is bounded by some uniform $r(s)$.
Applying Lemma \ref{lem:4_centers_close}, we see that both $G$ and $G'$ are $O_s(\zeta^2)$-equidistant.
As the edge of the spine dual to $\alpha$ connects two of the four centers of $G$, its length is at most $O_s(\zeta^2)$, and since its weight $c_\alpha(X, \lambda)$ is the length of the projection of this arc to a boundary leaf of $\lambda$, we can conclude that
\begin{equation}\label{eqn:arcwtbd}
c_\alpha(X, \lambda) \le O_s(\zeta^2).
\end{equation}
Taking our threshold $\zeta_{\ref{lem:fixX persistent}}$ small enough compared to all positive arc weights $\{c_\alpha(X,\lambda)\}_{\alpha\in \arc_\bullet}$ would lead to a contradiction, hence for small $\zeta$ no arc of $\arc_\bullet$ can be crossed by any arc of $\arc_\bullet^\tau(X,\lambda')$. \medskip

With this established, we now prove (3). 
We first note that $G_P(X,\lambda)$ is $0$-equidistant for all $P$.
\footnote{This is trivial when $P$ meets only $3$ branches of $\tau$ and follows from the definition/construction of the geometric arc system $\arc(X,\lambda)$, otherwise.}
By the same reasoning as above, $G_P(X,\lambda)$ is $O_s(\zeta)$-close to $G_P(X,\lambda')$ on a ball of radius at least $D_\zeta$ 
about the center of $G_P(X,\lambda)$.
By Lemma \ref{lem:triple_centers}, all centers of $G_P(X,\lambda')$ are $O_s(\zeta^2)$-close to the center of $G_P(X,\lambda')$, hence to each other, and thus $G_P(X,\lambda')$ is $O_s(\zeta^2)$-equidistant.

Now that (3) is proved, we move on to item (2).
By Lemma \ref{lem:radius_bounded} again, the distance from any center of $P$ to the corresponding basepoint is bounded by $r(s)$.
Thus Corollary \ref{cor:triplebasepoints} implies that any basepoint of $G_P(X,\lambda')$ is $O_s(\zeta^2)$ close to the corresponding basepoint of $G_P(X, \lambda)$.
Since the weight $c_\alpha$ is the distance along $\lambda$ between two basepoints on either side of $\alpha$, this gives a uniform estimate on how much the weight of $\alpha$ can change, proving (2).

To conclude (1), that $\arc_\bullet^\tau(X, \lambda')$ actually contains $\arc_\bullet$ instead of just being disjoint, we simply take our threshold $\zeta_{\ref{lem:fixX persistent}}$ small enough to ensure that the weights on the arcs of $\arc$ change only by half, say.
\end{proof}

In the sequel, it will also be convenient to find a filling arc subsystem of $\arc_\bullet$ consisting of arcs of definite weight that are guaranteed to persist at some uniform scale over the thick part of moduli space.

\begin{figure}
    \centering
    \includegraphics[width=\linewidth]{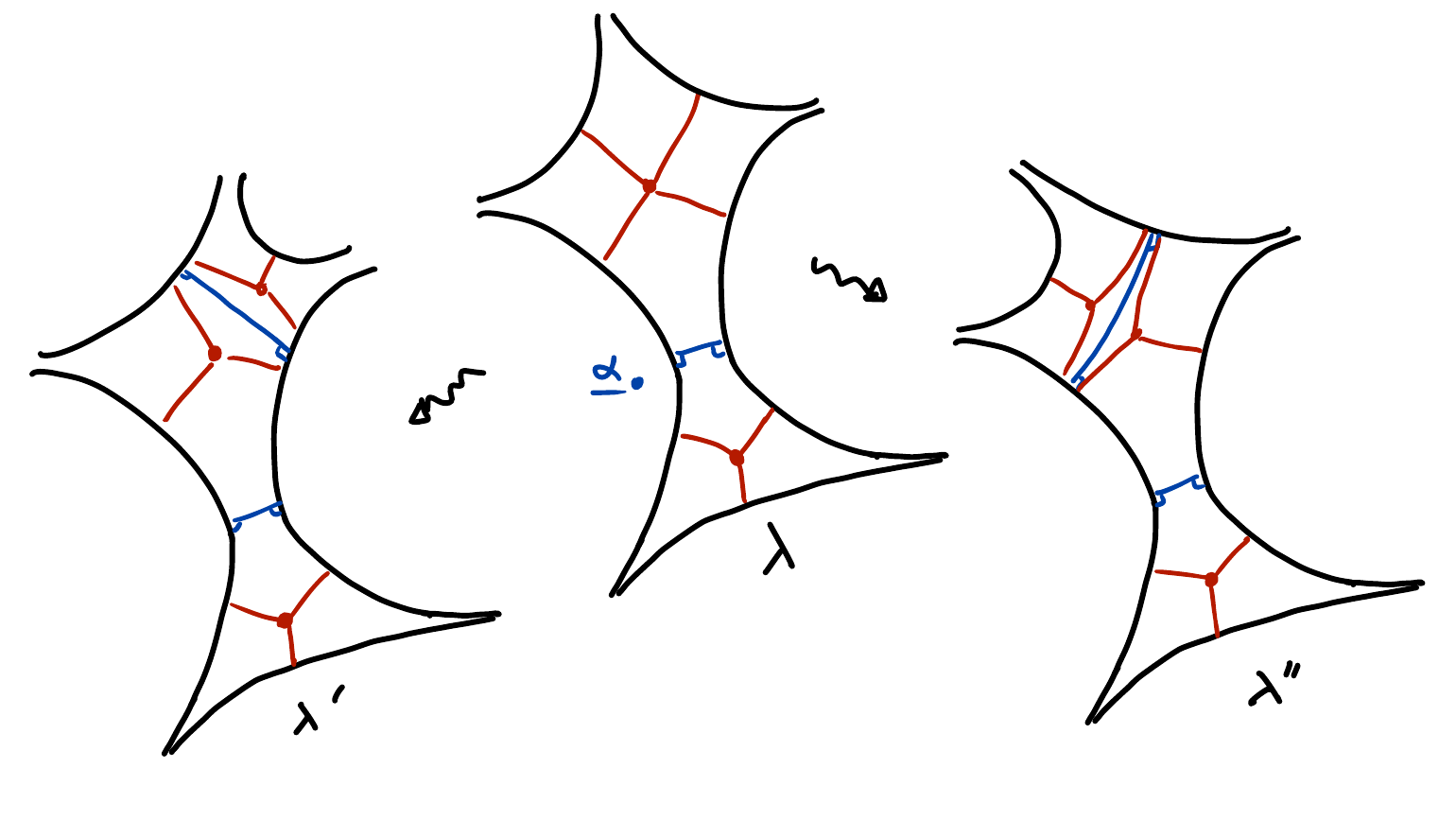}
    \caption{A persistent arc. Here, four of the geodesics in $\lambda$ are in an equilateral configuration, so nearby $\lambda'$ and $\lambda''$ may have crossing geometric arc systems.}
    \label{fig:persistent}
\end{figure}

\begin{corollary}\label{cor:persist_notallarcs}
There is a $\zeta_{\ref{cor:persist_notallarcs}}>0$ depending only on $s$ such that for any $s$-thick $X$ and any chain-recurrent $\lambda$ the following holds.
For any $\zeta \le \zeta_{\ref{cor:persist_notallarcs}}$, there is a filling subsystem $\arc(\zeta) \subset \arc_\bullet$ \label{ind:persistnotallarcs}
such that for any $\lambda'$ with $d_X^H(\lambda, \lambda') \le  \zeta$:
\begin{enumerate}
    \item $\arc_\bullet^{\tau}(X, \lambda')$ contains $\arc(\zeta)$.
    \item For every arc $\alpha \in \arc(\zeta)$, we have that
\[\left|c_{\alpha}(X, \lambda) - c_{\alpha}(X, \lambda') \right| = O_s(\zeta^2).\]
    \item For every $P\subset S\setminus (\tau\cup \arc(\zeta))$, the configuration $G_P(X,\lambda')$ is $O_s(\zeta^2)$-equidistant.
\end{enumerate}
\end{corollary}
\begin{proof}
The proof of Lemma \ref{lem:fixX persistent} actually shows that all arcs of weight at least $O_s(d_X^H(\lambda, \lambda')^2)$ persist after deformations of $\lambda$.
Similarly, all of the geometric estimates are uniform as they require only fellow-traveling estimates and the uniform bound on $r(s)$, the maximum distance from any point on the spine to the base lamination, and always yield estimates at scale $O_s(d_X^H(\lambda, \lambda')^2)$.

Thus, it remains to show that we can find a $c>0$ depending only on $s$ such that the arcs of $\arc$ with weight $\ge c$ necessarily fill $S \setminus \tau$.
To prove this, we note that any curve disjoint from $\lambda$ must have length at least $s$ (since $s$ is a bound on the systole), and so its projection to $\lambda$ must also have length at least $s/\cosh(r)$. But the total length of the projection is just the sum of the weights of the arcs that the curve crosses, so if $c < s/\cosh(r)(6g-6)$ then any essential curve in $X \setminus \lambda$ must cross an arc of weight $c$.
Thus the sub-arc system consisting of those of weight $c$ must be filling, and we can take our threshold $\zeta_{\ref{cor:persist_notallarcs}}$ small enough to ensure that all of these must persist.
\end{proof}

We now deduce Proposition \ref{prop:persistent} from Lemma \ref{lem:fixX persistent}.
Our threshold $\zeta_{\ref{lem:fixX persistent}}$ from above is uniform in the geometry of $X$ and the visible arc system, so as we vary $X$ we need only add in a small correction factor.

\begin{proof}[Proof of Proposition \ref{prop:persistent}]
Recall that we have fixed a uniform lower bound $s$ on the systole of $X$ and a scale $\delta \le \delta_{\ref{lem:equittswork}}(s)$ for our equilateral train track.

Choose a neighborhood $U$ about $X$ such that for every $X' \in U$, the Hausdorff metrics on $X$ and $X'$ are H{\"o}lder comparable with exponent $3/4$ and some uniform constant $K$.
Fix some $c$ smaller than the minimum weight of an arc of $\arc_\bullet$ and fix $\zeta_{\ref{lem:fixX persistent}}$ to be the threshold from Lemma \ref{lem:fixX persistent}, taken for thickness $s/2$ and arc weight $c$.
Choose $\zeta_{\ref{prop:persistent}}$ small enough such that
\[K \zeta_{\ref{prop:persistent}}^{3/4} < \min\{\zeta_{\ref{lem:fixX persistent}},
\sqrt{c}\}.
\]

For any $\zeta\le \zeta_{\ref{prop:persistent}}$, we now build our desired neighborhood $B_{\ref{prop:persistent}}(\zeta)$ of $X$.
Invoking the structure theorem of shear-shape coordinates
\cite[Theorem 8.1]{shshI} (see also \S\ref{subsec:shsh}),
there is an open neighborhood $V(\zeta) \subset \T_g$ of $X$ such that for every $X' \in V(\zeta),$
\begin{enumerate}[label=(\roman*)]
    \item $\arc(X', \lambda)$ contains $\arc(X, \lambda)$,
    \item For every arc $\alpha \in \arc_\bullet$, we have $c_{\alpha}(X', \lambda)> c$, and
    \item For every arc $\beta \in \arc_\bullet^{\tau}(X', \lambda)$
    we have 
    \[\left| c_{\beta}(X', \lambda) - c_{\beta}(X, \lambda) \right|< \zeta^{3/2}\]
    where if $\beta \notin \arc_\bullet$ then it is assigned weight $0$.
\end{enumerate}
Finally, set $B_{\ref{prop:persistent}}(\zeta)$ to be the intersection of $U \cap V(\zeta)$ with the $s/2$-thick part of Teichm{\"u}ller space. 

Using item (iii), we can conclude
\begin{claim}
There is a constant $L=L(X,\lambda)$ such that for every component $P\subset S\setminus (\tau\cup \arc_\bullet)$ and $X' \in B_{\ref{prop:persistent}}(\zeta)$, the configuration $G_P(X',\lambda)$ is $L\zeta^{3/2}$-equidistant.
\end{claim}
\begin{proof}[Proof of the claim]
    Complete $\arc(X,\lambda)$ to a maximal arc system $\arc'$ by formally adding in arcs of zero weight to $\arc(X,\lambda)$ arbitrarily. 
    Then Theorem 6.4 and Lemma 6.9 of \cite{shshI} assert that the weights $\{c_\alpha(\bullet,\lambda)\}_{\alpha \in \arc'}$ are analytic coordinates on $\T(S\setminus \lambda)$.
    Thus the configuration of geodesics $G_P(\bullet,\lambda)$ varies analytically as a function of these coordinates.

    Now, the function that records the center in $\HH^2$ from a triple of geodesics that do not separate each other is analytic, and for a fixed point $x\in \HH^2$, we know that $y\mapsto d(x,y)$ is $1$-Lipschitz.
    Therefore, the maximum distance between any pair of centers of $G_P(X',\lambda)$ is a Lipschitz function of the arc weight coordinates on any compact set of metrics on $S\setminus \lambda$.  Let  $L$ be a Lipschitz constant on the compact set of metrics  $\{X'\setminus \lambda : X' \in \overline{B_{\ref{prop:persistent}}(\zeta_{\ref{prop:persistent}})}
    \}$.
    This function is $0$ at $X$ and the metric on $X'\setminus \lambda$ is obtained by changing the weight parameters on arcs contained in $P$ by at most $\zeta^{3/2}$, so $G_P(X',\lambda)$ is at most $L\zeta^{3/2}$-equidistant.
    \end{proof}

The Proposition is now proved for for pairs of the form $(X',\lambda)$, where $X' \in B_{\ref{prop:persistent}}$, while Lemma \ref{lem:fixX persistent} proved the Proposition for pairs of the form $(X,\lambda')$, with $\lambda'$ close enough to $\lambda$. We finish by indicating how to combine the two.

So suppose $X' \in B_{\ref{prop:persistent}}(\zeta)$ and $\lambda'$ is $\zeta \le \zeta_{\ref{prop:persistent}}$ close to $\lambda$ on $X$.
Since $X' \in V(\zeta)$, we know that every arc $\alpha \in \arc_\bullet$ persists on $(X', \lambda)$ and has weight at least $c$.
Since $X' \in U$, we have that
\[d_{X'}^H(\lambda, \lambda') < K \zeta^{3/4} < \zeta_{\ref{lem:fixX persistent}}\]
and so we can apply part (1) of Lemma \ref{lem:fixX persistent} to deduce that $\alpha$ appears as an arc of 
$\arc_\bullet' = \arc_\bullet^\tau(X', \lambda')$, proving part (1) of the Proposition.
Moreover, part (2) of the same lemma implies that the weight of $\alpha$ changes by $O_s(\zeta^{3/2})$, proving proving (2) for every arc of $\arc_\bullet'$ that comes from an arc of $\arc_\bullet$.

Finally, the Claim tells us that the centers of $G_P(X', \lambda)$ are all $O(\zeta^{3/2})$ close, while the proof of Lemma \ref{lem:fixX persistent} proves that each center of $G_P(X', \lambda')$ is $O_s(\zeta^{3/2})$ close to the corresponding center of $G_P(X', \lambda)$.
Thus $G_P(X',\lambda')$ is 
\begin{equation}\label{eqn:non-uniform estimate persistent}
    O_{s/2}(\zeta^{3/2}) + O(\zeta^{3/2}) \le O(\zeta^{3/2})
\end{equation}
-equidistant, proving (3).
We remark that since we lost uniformity of our estimates in the Claim, we have lost uniformity in the equation above (hence why there is no $s$ subscript in our big $O$), so the implicit constant $M_{\ref{prop:persistent}}$ in \eqref{eqn:non-uniform estimate persistent} depends on both $X$ and $\lambda$ (but not $\zeta$).
\end{proof}

\begin{remark}
    Note that the threshold $\zeta_{\ref{prop:persistent}}$ depends on the pair $(X,\lambda)$, as does the constant $L$ from the Claim.  However, all the other estimates in the proof of Proposition \ref{prop:persistent} only depend on the systole $s$.
    Additionally, the constant $L$ can also be chosen only to depend on $s$ by invoking a compactness argument viewing certain Teichm\"uller spaces of surfaces with spikes as lying at infinity of other Teichm\"uller spaces (by pinching very short orthogeodesic arcs to zero length). Compare the arguments in \cite[\S6]{shshI}.
    Such an argument is omitted since we will not use this result in the sequel.
\end{remark}

\section{Comparing ties}\label{sec:ties}
We now analyze the structure of the orthogeodesic foliation near $\lambda$. The main result of this section is that, at least in a neighborhood of $\lambda$, the leaves of $\Ol(X)$ vary (H{\"o}lder) continuously as $\lambda$ varies in the Hausdorff metric.
This is essentially equivalent to proving continuity of the ``shear'' part of shear-shape coordinates as we let $\lambda$ vary (the result for variable $X$ is already contained in \cite{shshI}).

In this section, we fix an $s$-thick hyperbolic surface $X$ and a geodesic lamination $\lambda$.
For every $\delta \le \delta_{\ref{lem:equittswork}}$, the $\delta$--equilateral neighborhood $\cE_\delta(\lambda)$ defines a train track $\tau = \tau(X, \lambda, \delta)$.
In particular, we recall from the previous section that if $\lambda'$ is fully carried on $\tau$ then we have a natural identification between pairs of boundary geodesics corresponding to the same branch. Note that the correspondence between geodesics depends on $\tau$, hence on the defining parameter $\delta$.

Given a branch of $\tau$, let $g$ and $h$ denote the corresponding boundary geodesics of $\lambda$ and let $t$ be any tie of the equilateral neighborhood $\cE_{\delta}(\lambda)$ connecting $g$ and $h$.
Let $g'$ and $h'$ denote the boundary geodesics of $\lambda'$ corresponding to the same branch of $\tau$; then the main result of this section (Proposition \ref{prop:tiestable} below) yields an estimate on the Hausdorff distance between $t$ and segments of $\cO_{\lambda'}(X)$ connecting $g'$ and $h'$.

Let us first prove that there actually exist such segments, so long as we take the train track inducing the correspondence between $(g,h)$ and $(g', h')$ to be defined at a small enough scale.

\begin{lemma}\label{lem:change base ties still meet}
For any $s>0$, there is a $\delta_{\ref{lem:change base ties still meet}}(s)$ such that for any $\delta \le \delta_{\ref{lem:change base ties still meet}}(s)$, any $s$-thick $(X,\lambda)$, and any $\lambda'$ close enough to $\lambda$ (also depending only on $s$), any leaf of $\cO_{\lambda'}(X)$ meeting a tie $t$ of $\cE_\delta(\lambda)$ connecting $g$ and $h$ must meet the corresponding geodesics $g'$ and $h'$.
\end{lemma}

\begin{proof}
We begin by observing that for any $s$-thick $(X, \lambda)$, any point $x \in X$ and any geodesic $g$ of $\lambda$, we have that if $x$ is within $\delta_{\ref{lem:equittswork}} /w_{\ref{prop:ttdefs_comp}}$ of $g$,
then Proposition \ref{prop:ttdefs_comp} implies that $x \in \mathcal E_{\delta_{\ref{lem:equittswork}}}(\lambda)$ and hence the tie through $x$ (in particular, the leaf of $\Ol(X)$ through $x$) must meet $g$.

Now for any $\delta \le \delta_{\ref{lem:equittswork}}$ and any tie $t$ of $\cE_\delta(\lambda)$, we know by Lemma \ref{lem:equittwidth}
that any point $x \in t$ is at most $W_{\ref{lem:equittwidth}}\delta$ far from $g$ and $h$. Similarly, if $\lambda' \subset \cE_\delta(\lambda)$ then $x \in t$ is at most $W_{\ref{lem:equittwidth}}\delta$ far from the corresponding boundary geodesics $g'$ and $h'$ as well.
Therefore, so long as we take
$\delta \le \delta_{\ref{lem:equittswork}}/w_{\ref{prop:ttdefs_comp}}W_{\ref{lem:equittwidth}}=:\delta_{\ref{lem:change base ties still meet}}$
and take $\lambda'$ close enough to $\lambda$ such that it lies entirely in $\cE_\delta(\lambda)$, then we can apply the reasoning from the paragraph above to ensure that the leaf of $\cO_{\lambda'}(X)$ through $x$ must meet both $g'$ and $h'$.
\end{proof}

We now state the main estimate of this section. Fix $\delta$ smaller than the threshold from Lemma \ref{lem:change base ties still meet} and let $t'$ be any segment of $\cO_{\lambda'}(X)$ that meets $t$ and runs between the boundary geodesics $g'$ and $h'$ corresponding to $g$ and $h$.

\begin{proposition}[Ortho segments Hausdorff close]\label{prop:tiestable}
Given any $s>0$, any $\delta \le \delta_{\ref{lem:change base ties still meet}}(s)$, and any $a \in (0,1)$,
there is a $\zeta_{\ref{prop:tiestable}}>0$ such that 
for any $s$-thick $X$ and any $\lambda, \lambda'$ with $d_X^H(\lambda, \lambda') < \zeta_{\ref{prop:tiestable}}$,
\[d_X^H(t, t') = O_s(d_X^H(\lambda, \lambda')^a).\]
\end{proposition}

\subsection{Breaking into pieces}
Throughout our proof we work in the universal cover, but will mostly suppress this notation for clarity.

For any defining parameter $\zeta < \delta$, a tie of the $\delta$--equilateral train track $\tau$ will cut through a sequence of branches of $\tau_\zeta := \tau(X, \lambda, \zeta)$.
Equivalently, if one considers the equilateral neighborhoods $\cE_\delta:=\cE_\delta(\lambda)$ and $\cE_\zeta:=\cE_\zeta(\lambda)$, then there is a natural carrying map witnessing $\tau_\zeta \prec \tau$ induced by collapsing the ties of $\cE_\delta$. 

We begin with an estimate on how many branches of $\tau_\zeta$ can run through a branch of $\tau$.

\begin{lemma}\label{lem:branches_bounded}
For any tie $t$ of $\cE_\delta$, the number of components of $t \cap \cE_\zeta$ is $O_s(\log(\delta/\zeta))$.
\end{lemma}
\begin{proof}
Consider the train path traced out by a single proto-spike as $\delta$ decreases to $\zeta$.
Such a proto-spike returns some number of times $M$ to any given tie $t$ of $\cE_\delta$, forming an essential loop in $X$ each time.
The length of this train path therefore is at least $Ms$.

On the other hand, the distance traversed by the (boundary of the) proto-spike is equal to $D_\zeta - D_\delta$, which up to a universal additive error $C$ is equal to $\log(1/\zeta) - \log(1/\delta)$; see Equation \eqref{eqn:defDdelta} and Lemma \ref{lem:boundedlengthspikes}.
Thus 
\[M\le \frac{1}{s} \left( \log(1/\zeta) - \log(1/\delta)+C\right).\]
This spike contributes at most $(2M+1)$ branches of $\tau_\zeta$ running through $t$
(one for each side of the spike each time it meets $t$),
and there are at most $6|\chi(S)|$ spikes of $\tau_\zeta$.
So there are at most 
\[12|\chi(S)| (M+1) = O_s(\log(\delta/\zeta))\]
branches of $\tau_\zeta$ running through any branch of $\tau$.
\end{proof}

Number the branches of $\tau_\zeta$ that $t$ meets by $b_1, \ldots, b_N$ (equivalently, number the components of $t \cap \cE_\zeta$).
Each $b_i$ corresponds to a pair of components of $\tX \setminus \tlambda$ that are adjacent over $b_i$; denote by $g_i$ and $h_i$ the boundary geodesics of these components that run along $b_i$.
In the case that $b_i$ corresponds to an isolated leaf of $\lambda$, then $g_i = h_i$.
Thus, our choice of $\zeta$ gives rise to a sequence of boundary geodesics of $\lambda$
\begin{equation}\label{eqn:geoseq}
g = g_1, h_1, g_2, h_2, \ldots, g_N, h_N = h.
\end{equation}
Note that each geodesic in this sequence separates those that came before from those that come after.
Using \eqref{eqn:geoseq}, we may write $t$ as a concatenation of subsegments $t_i$ and $s_i$, where $t_i$ connects $g_i$ to $h_i$ (i.e., it is a tie of $\tau_\zeta$) and $s_i$ connects $h_i$ to $g_{i+1}$ (i.e., it traverses a spike of $S \setminus \tau_\zeta$).
See Figure \ref{fig:geoseq}, and compare with the discussion of ``admissible routes'' from \cite[\S 14.5]{shshI}.

\begin{figure}
    \centering
    \includegraphics[scale=.6]{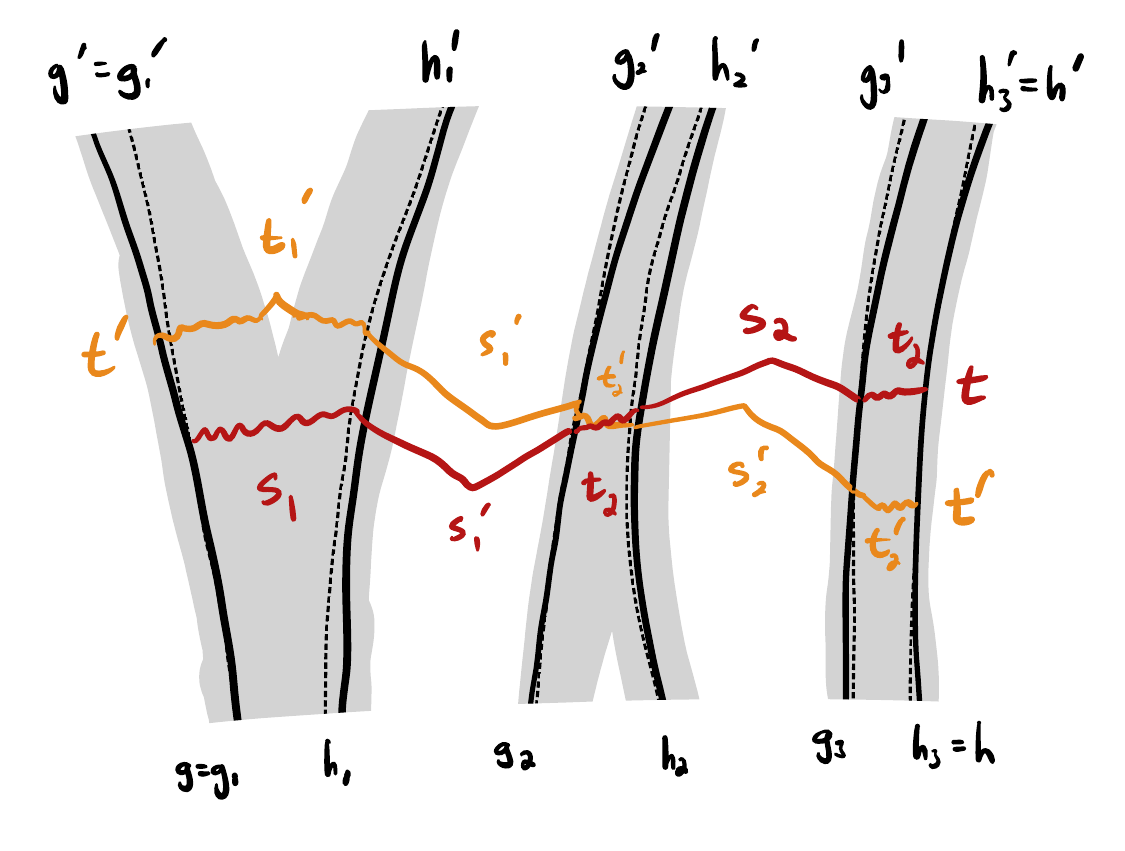}
    \caption{The sequence of boundary geodesics specified by a $\zeta$ track.}
    \label{fig:geoseq}
\end{figure}

Now for any $\lambda'$ sufficiently close, we have that $\lambda' \preceq \tau_\zeta$.
We may therefore form the sequence
\[g' = g_1', h_1', g_2', h_2', \ldots, g_N', h_N' = h'\]
of boundary geodesics of $\lambda'$ corresponding to the branches of $\tau_\zeta$ met by $t$,
and can similarly express the tie $t'$ connecting $g'$ and $h'$ as a concatenation of segments $t_i'$ and $s_i'$ corresponding to ties and spikes of $\tau_\zeta$.

\begin{remark}
We note that $t'$ does not have to meet the same branches of $\tau_\zeta$ as $t$ does; for example, this can happen if $t$ meets a branch near a switch.
However, since each geodesic of \eqref{eqn:geoseq} separates $g$ from $h$ we know $t'$ must meet the corresponding boundary geodesics $g_i'$ and $h_i'$ of $\lambda'$.
\end{remark}

\subsection{Estimates on pieces}
We now estimate the Hausdorff distance between small segments of $\Ol(X)$ and $\cO_{\lambda'}(X)$ connecting corresponding consecutive pairs of boundary geodesics.
We first prove an estimate that allows us to control the Hausdorff distance of different leaves of $\Ol(X)$ by their intersections with $\lambda$.

\begin{lemma}\label{lem:orthodiv}
For any $(X, \lambda)$ and any two transversely isotopic segments $t_1$ and $t_2$ of $\Ol(X)$ connecting the same boundary geodesics of $\lambda$,
\[d_X^H(t_1, t_2)  = O_s(d(t_1 \cap g, t_2 \cap g))\]
where $g$ is any geodesic of $\lambda$ meeting $t_1$.
\end{lemma}

Note that $d(t_1 \cap g, t_2 \cap g)$ does not actually depend on our choice of $g$ thanks to 
the fact that transport along the leaves of $\Ol(X)$ preserves length along $\lambda$. 
See \S\ref{subsec:orthofoliation}.

\begin{remark}
As in Lemma \ref{lem:radius_bounded} and the discussion following it, the implicit constant in $O_s(\cdot)$ necessarily depends on the thickness of $X$. Indeed, if some segment $t$ of $\Ol(X) \setminus \lambda$ is very long, then the exponential divergence of geodesics in $\mathbb{H}^2$ implies that small perturbations to $t \cap \lambda$ can have dramatic impact on $t$.
\end{remark}

\begin{proof}
It suffices to prove the estimate piece by piece; so without loss of generality suppose that $t_1$ and $t_2$ meet no leaves of $\lambda$ except at their endpoints.
In this case, they bound a hexagon which can be decomposed into two isometric quadrilaterals, each with two right angles where they meet $\lambda$.
Fix one of these quadrilaterals and label its vertices by $A$, $B$, $C$, and $D$; see Figure \ref{fig:quadslide}.
We have therefore reduced to bounding the Hausdorff distance between $\overline{AD}$ and $\overline{BC}$ in terms of the length of $\overline{AB}$.

We may further assume that the hexagon bounded by $t_1$ and $t_2$ does not contain the orthogeodesic representative $\alpha$ of the isotopy class: otherwise, we can decompose $ABCD$ into two further quadrilaterals, each with three right angles, and use these to estimate the distance from each $t_i$ to $\alpha$.
In particular, up to relabeling, we may assume that $\angle ADC$ is at least $\pi/2$.

\begin{figure}[ht]
    \centering
    \includegraphics[scale=.8]{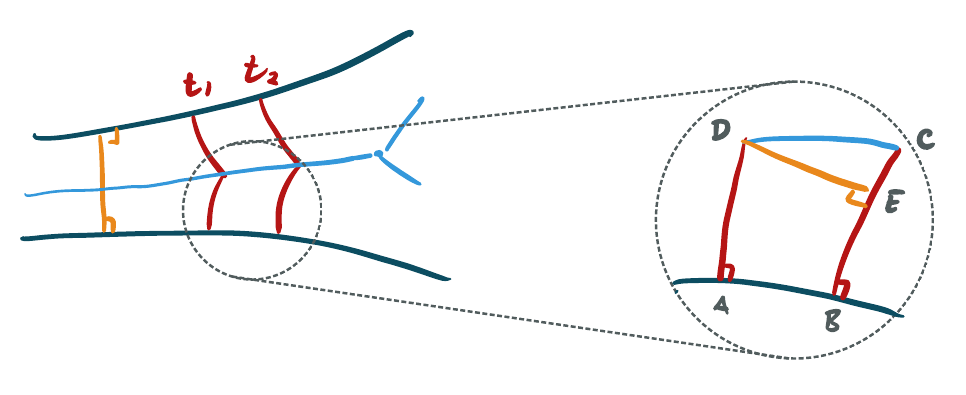}
    \caption{Right-angled quadrilaterals in the orthogeodesic foliation.}
    \label{fig:quadslide}
\end{figure}

The desired estimate now follows from elementary hyperbolic geometry considerations. 
Drop a perpendicular from $D$ to $\overline{BC}$ (this must be contained in the quadrilateral by our assumption on $\angle ADC$) and label its intersection $E$.
Now the lengths of $\overline{AD}$ and $\overline{BE}$ are both bounded above by $\delta_{\ref{lem:change base ties still meet}}(s)$, and so the hyperbolic trigonometry of trirectangles \cite[Theorem 2.3.1]{Buser} implies that the length of $\overline{DE}$ is at most some uniform constant $C_1$ times the length of $\overline{AB}$.
Since the distance between a point on a geodesic and its projection to another geodesic is a convex function, this implies that every point of $\overline{AD}$ is within $C_1 \ell(\overline{AB})$ of some point of $\overline{BE}$ and vice versa.
It remains to show that every point of $\overline{CE}$ is close to $D$; this is in turn a consequence of our above estimate on the length of $\overline{DE}$ plus a little more hyperbolic trigonometry (in particular, \cite[Theorem 2.2.2, (iii)]{Buser}) and the fact that $\angle BCD$ is bounded away from $0$ (Fact \ref{fact:kink angle}).
This completes the proof of the Lemma.
\end{proof}

We now consider how segments of $\Ol(X)$ at scale $\zeta$ change as one varies the lamination at that scale.

\begin{lemma}\label{lem:Hcloseest_branch}
For any $s >0$ and any small enough $\zeta$, the following holds.
Fix an $s$-thick $(X, \lambda)$ and suppose that $g_i$ and $h_i$ are two boundary geodesics of $\lambda$ corresponding to a branch of $\tau_\zeta$ and let $t_i$ be any tie of $\cE_\zeta$ connecting $g_i$ to $h_i$.
Then for any $\lambda'$ with
$d_X^H(\lambda, \lambda') < \zeta/w_{\ref{prop:ttdefs_comp}}$, we have
\[d_X^H(t_i, t_i') = {O_s}(\zeta),\]
where $t_i'$ is any segment of $\cO_{\lambda'}(X)$ that meets $t_i$ and connects the corresponding boundary geodesics $g_i'$ and $h_i'$ of $\lambda'$.
\end{lemma}

Observe that the hypothesis on Hausdorff distance implies that $\lambda' \subset \cE_\zeta$ and so $\lambda'$ is fully carried on $\tau_\zeta$.
Thus, it makes sense to use $\tau_\zeta$ to set up a correspondence between pairs of boundary geodesics of $\lambda$ and $\lambda'$.

\begin{proof}
We will in fact show that the diameter of $t_i \cup t_i'$ is $O_s(\zeta)$.
Since $t_i$ and $t_i'$ meet, it suffices to bound the length of each.
As $t_i$ is a tie of a $\zeta$--equilateral train track, Lemma \ref{lem:equittwidth} tells us its length is $O_s(\zeta)$.

On the other hand, $t_i'$ may not necessarily be contained in $\cE_\zeta(\lambda')$.
It is, however, a subsegment of a tie for a slightly larger defining parameter.
Indeed, by Proposition \ref{prop:ttdefs_comp} we know that $t_i$ is contained within a uniform $w \zeta$ neighborhood of $\lambda$ for $w=w_{\ref{prop:ttdefs_comp}}$.
Now since $\lambda'$ and $\lambda$ are $\zeta/w$ Hausdorff close, this implies that $t_i$ is completely contained in the $(w+ 1/w)\zeta$ uniform neighborhood of $\lambda'$, which is in turn contained in the $(w^2+1)\zeta$--equilateral neighborhood of $\lambda'$ (Proposition \ref{prop:ttdefs_comp} again).
Thus, any $t_i'$ as in the statement of the lemma must be subsegment of a tie of $\tau(X, \lambda', (w^2+1)\zeta)$
which is a train track for small enough $\zeta$ (depending only on $s$).
Lemma \ref{lem:equittwidth} then implies its ties have length 
\[W_{\ref{lem:equittwidth}}(w^2+1)\zeta = O_s(\zeta),\]
completing the proof of the lemma.
\end{proof}

We now turn to the spikes of $\tau_\zeta$ and the leaves of $\Ol(X)$ connecting them.

\begin{lemma}\label{lem:Hcloseest_spike}
Suppose that $h$ and $g$ are two boundary geodesics of $\lambda$ corresponding to a spike of $S \setminus \tau_\zeta$ and let $\ell$ be any segment of $\Ol(X)$ running from $h$ to $g$ of length at most $\log(3)$
that cuts off the corresponding spike of $S \setminus \tau_\zeta$ (see Figure \ref{fig:Hcloseest_spike}).
Then
\[d_X^H(\ell, \ell') = {O}_s(\zeta),\]
where $\ell'$ is any segment of $\cO_{\lambda'}(X)$ that meets $\ell$ and connects the corresponding geodesics $h'$ and $g'$ of $\lambda'$.
\end{lemma}

\begin{figure}[ht]
    \centering
    \includegraphics[width=.6\linewidth]{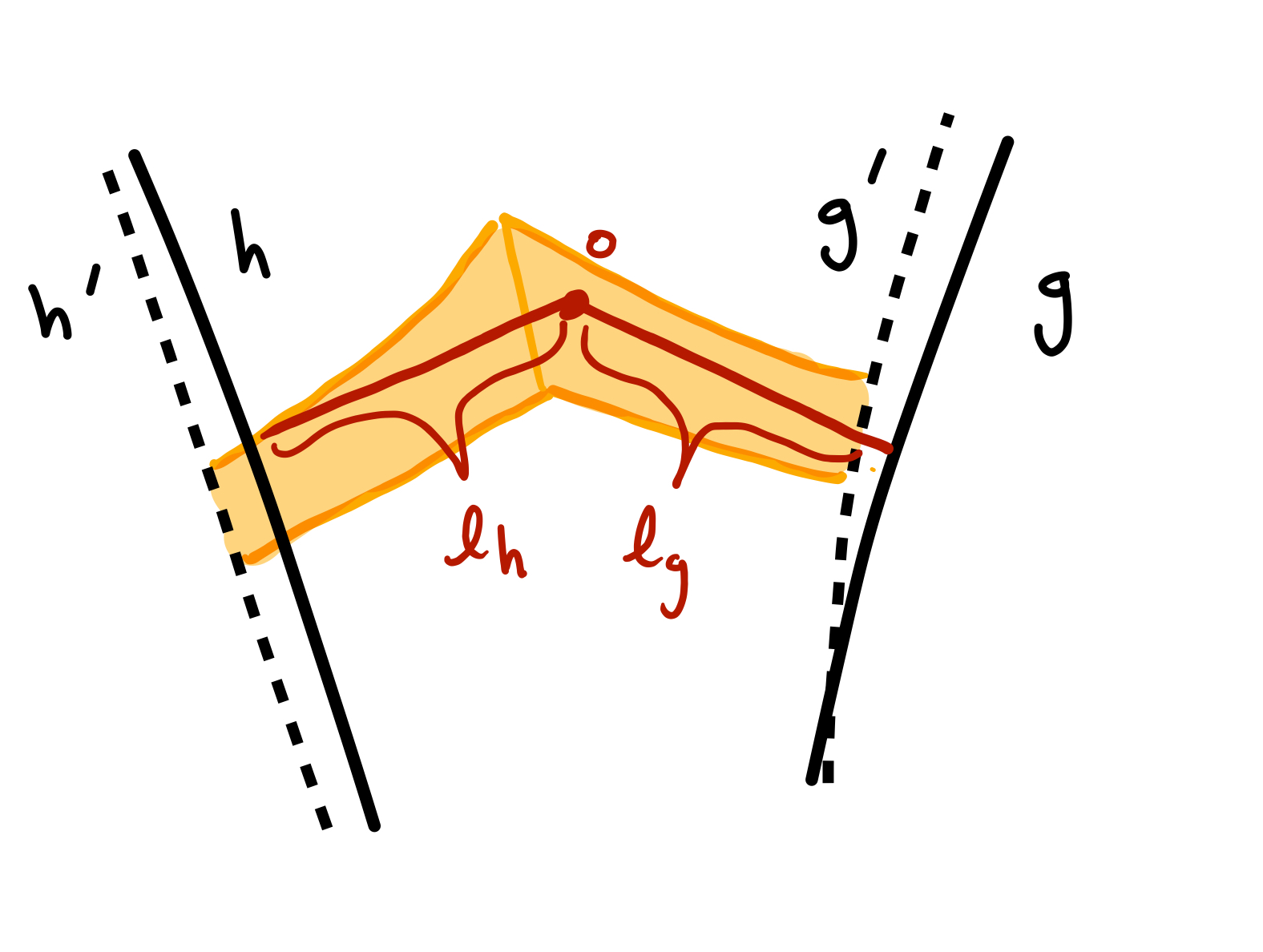}
    \caption{The projection of a segment of the orthogeodesic foliation $\Ol(X)$ to a Hausdorff-close lamination $\lambda'$. The shaded region is foliated by segments of $\cO_{\lambda'}(X)$.}
    \label{fig:Hcloseest_spike}
\end{figure}

\begin{proof}
The leaf $\ell$ lies within a band of parallel leaves of $\cO_{\lambda'}(X)$ running from $g'$ and $h'$, so by Lemma \ref{lem:orthodiv} above it suffices to bound the size of the closest-point projection of $\ell$ to $\lambda'$. Compare Figure \ref{fig:Hcloseest_spike}.
We first observe that if the endpoint of $\ell$ is within $W_{\ref{lem:equittwidth}} \zeta$ of the corresponding switch delimiting the spike of $S \setminus \tau_\zeta$, then $\ell$ itself has length at the order of $\zeta$ (by some basic hyperbolic geometry) and so we can apply Lemma \ref{lem:Hcloseest_branch} to get the desired statement.

So assume that the endpoints of $\ell$ are not within $W_{\ref{lem:equittwidth}} \zeta$ of the corresponding switch. The leaf $\ell$ is comprised of two geodesic segments: one, which we denote by $\ell_g$, whose closest-point projection to $\lambda$ is a point on $g$, and another, $\ell_h$, defined similarly. These segments have the same length $L$ and meet in an angle of at least  $2\pi/3$ at a point that we will call $o$ (Fact \ref{fact:kink angle}).

Let us first bound the size of the projection of $\ell_g$ to $g'$.
For any point $p \in \ell_g$, consider the ball $B$ of radius $d(p, g) + W_{\ref{lem:equittwidth}}\zeta$ centered at $p$.
Since $\lambda' \prec \tau_\zeta$, which in turn has width bounded by $W_{\ref{lem:equittwidth}} \zeta$, we see that $B$ meets both $g'$ and $h'$.
Since we assumed that the endpoints of $\ell$ were at least $W_{\ref{lem:equittwidth}} \zeta$ far away from the corresponding switch of $\tau_\zeta$, we have that every tie of $\cN_\zeta(\lambda)$ through any point of $g \cap B$ also meets $g'$. 
Thus, we have proven that $g$ and $g'$ are $W_{\ref{lem:equittwidth}}\zeta$-fellow travelers inside of $B$.

We now apply Corollary \ref{cor:projectionestimate2}.
Recall that the implicit constant in the statement of the Corollary depend only on the radius of $B$, which is bounded by $\log 3$ by assumption. Therefore, we get that the projections of $p$ to $g$ and $g'$ are $O(W_{\ref{lem:equittwidth}}\zeta)$ close with universal constants, and since $W_{\ref{lem:equittwidth}}$ depends only on the thickness $s$, this is $O_s(\zeta)$.
Since the projection of $\ell_g$ to $g$ is a single point, we thus have that
\begin{equation}\label{eqn:diambd tieproj}
\text{diam}(\pi_g(\ell_g) \cup \pi_{g'}(\ell_g)) = O_s(\zeta).
\end{equation}

It remains to estimate how much of $\ell_h$ is closer to $g'$ than to $h'$.
Let $p \in \ell_h$; since $\ell_h$ is a geodesic, we have that $d(p, h) = L - d(o,p)$
and since the angle between $\ell_g$ and $\ell_h$ is obtuse, we have that $d(p, g) > L.$
Again, since $\lambda' \prec \tau_\zeta$ and we can bound the width of $\tau_\zeta$, we know that the leaves $h$ and $h'$ are $W_{\ref{lem:equittwidth}}\zeta$-close around the endpoint of $\ell$, and the same for $g$ and $g'$.
We therefore get that
\[
d(p,h') \le L - d(o,p) + W_{\ref{lem:equittwidth}}\zeta
\text{ and }
d(p, g') > L - W_{\ref{lem:equittwidth}}\zeta.
\]
Taking these inequalities together, we see that if $p$ is closer to $g'$ than $h'$, then $d(o,p) < 2W_{\ref{lem:equittwidth}}\zeta$. Thus the subsegment of $\ell_h$ that projects to $g'$ has length at most $2W_{\ref{lem:equittwidth}}\zeta = O_s(\zeta)$; combined with \eqref{eqn:diambd tieproj} this completes the proof of the Lemma.
\end{proof}

\subsection{Putting the pieces together}
We now use the estimates of the previous section to build a path out of segments of $\cO_{\lambda'}(X)$ and $\lambda'$ that remains close to $t$. 
The previous Lemma \ref{lem:orthodiv} then lets us estimate the distance between $t'$ and this path, completing the proof of Proposition \ref{prop:tiestable}.

\begin{proof}[Proof of Proposition \ref{prop:tiestable}]
Reference to Figure \ref{fig:t''} will be helpful throughout this proof.

For each $i=1, \ldots, N$ choose a segment $t_i''$ of $\cO_\lambda'(X)$ that meets $t_i$ and runs between $g_i'$ and $h_i'$.
Likewise, choose segments $s_i''$ that meet $s_i$ and 
traverse the corresponding spike of $S \setminus \tau_\zeta$.
By Lemmas \ref{lem:Hcloseest_branch} and \ref{lem:Hcloseest_spike}, we know that each $t_i''$ and $s_i''$ remains $O_s(\zeta)$ close to the corresponding subsegment of $t$.
In particular, this implies that the endpoints of consecutive segments are also $O_s(\zeta)$ close, i.e.,
\[d(t_i'' \cap h_i', s_i'' \cap h_i') = O_s(\zeta) 
\text{ and }
d(s_i'' \cap g_{i+1}', t_{i+1}'' \cap g_{i+1}') = O_s(\zeta).\]
Connecting up the endpoints of these segments with $O_s(\zeta)$-short geodesic segments running along the corresponding leaves of $\lambda'$, we arrive at a path $t''$ that runs from $g'$ to $h'$ and by construction remains within $O_s(\zeta)$ of our original segment $t$ of $\Ol(X)$.

In light of Lemma \ref{lem:orthodiv} above, it suffices to bound the distance between the intersections of $t''$ and $t'$ with each geodesic separating $g'$ from $h'$ to complete the proof of the Proposition.

\begin{figure}[ht]
    \centering
    \includegraphics[scale=.6]{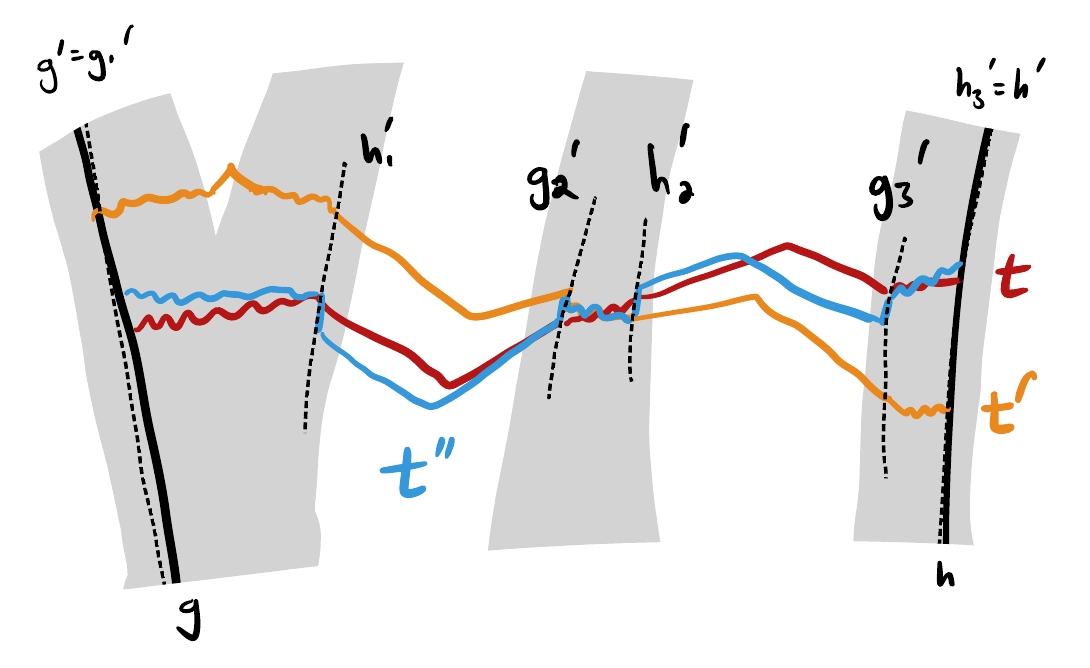}
    \caption{Building a path $t''$ that tracks $t$ out of segments of $\lambda'$ and $\cO_{\lambda'}(X)$.}
    \label{fig:t''}
\end{figure}

Without loss of generality, let us assume that $t$ meets the first subsegment $t_1'$ of $t'$ at a point $p$ (if not, then the same proof holds up to reindexing).
In particular, since $t$ and $t''$ are close, this implies that $p$ is within $O_s(\zeta)$ of $t_1''$.
Now consider the closest point projection of $p$ to the closest leaf $\ell$ of $\lambda'$; since $p \in t_1'$ this is the same as $t_1' \cap \ell$, so we have that
\[d(t_1' \cap \ell, t_1'' \cap \ell) = O_s(\zeta)\]
since closest point projection is distance-decreasing.
Since transport along the leaves of $\cO_{\lambda'}(X)$ preserves length along $\lambda$,
\[d(t_1' \cap g_1', t_1'' \cap g_1')
=d(t_1' \cap h_1', t_1'' \cap h_1')
= O_s(\zeta).\]

The path $t''$ then continues with a geodesic segment along $\lambda'$ of length $O_s(\zeta)$ connecting $t_1'' \cap h_1'$ to $s_1'' \cap h_1'$; we therefore see that
\[d(s_1' \cap h_1', s_1'' \cap h_1')
=d(s_1' \cap g_2', s_1'' \cap g_2')
= 2 O_s(\zeta).\]
Continuing on, $t''$ follows $g_2'$ for a distance $O_s(\zeta)$ in order to meet up with the endpoint of $t_2''$. Hence
\[d(t_2' \cap g_2', t_2'' \cap g_2')
=d(t_2' \cap h_2', t_2'' \cap h_2')
= 3 O_s(\zeta).\]
Iterating this argument, we see that both
\[\text{diam}(t'\cap g_i', t'' \cap g_i') = 2i O_s(\zeta)
\text{ and }
\text{diam}(t'\cap h_i', t'' \cap h_i') = 2i O_s(\zeta)
\]

Invoking Lemma \ref{lem:branches_bounded}, there are only $O_s(\log(\delta/\zeta))$ many pairs of geodesics $(g_i', h_i')$ (corresponding to branches of the $\zeta$--equilateral train track for $\lambda$), and so finally we see that for any leaf $\ell'$ of $\lambda'$ separating $g'$ from $h'$, we have that
\[\text{diam}(t'\cap \ell', t'' \cap \ell') = O_s(\zeta \log (\delta/\zeta))\]
which, for any $a
\in (0,1)$, is $O_s(\zeta^a)$.
Lemma \ref{lem:orthodiv} now tells us that the entire paths $t'$ and $t''$ must be $O_s(\zeta^a)$ Hausdorff-close, and therefore so are $t'$ and $t$.
\end{proof}

\section{Proof of continuity}\label{sec:continuity proof}

In this section, we use our work from above to prove 
that for nearby pairs $(X, \lambda)$ and $(X', \lambda')$, the corresponding differentials $\cO(X, \lambda)$ and $\cO(X', \lambda')$ have comparable cellulations and the periods of corresponding saddle connections are close.
We essentially already have the estimates on periods that we need; the main obstacle we still need to overcome is to show that we can piece these estimates together in a coherent way, making sense of what it means for saddle connections to ``correspond.''

Recall that we have fixed an $s$-thick hyperbolic surface $X$, a chain-recurrent lamination $\lambda$, and a $\delta$ such that the equilateral train track $\tau:=\tau(X, \lambda, \delta)$ is trivalent.
The requirement of trivalence will simplify some of our arguments, but will also prevent the thresholds in this section from being uniform for all $s$-thick surfaces.
Throughout the section $\arc$ will denote the entire geometric arc system $\arc(X, \lambda)$, while $\arc_\bullet$ will be used for the visible arcs of $\arc$ (with respect to either $\tau$ or the corresponding equilateral neighborhood $\cE$).

\subsection{Stability of train tracks}\label{subsec:ttstab}
As a first step towards our main theorem, we first prove that the assignment of a pair $(X, \lambda)$ to an equilateral train track $\tau$ is locally constant.

\begin{proposition}[Stability of train tracks]\label{prop:ttstable}
If $X'$ is sufficiently close to $X$ and $\lambda'$ is sufficiently Hausdorff-close to $\lambda$,
then the equilateral train track $\tau(X', \lambda', \delta)$ is isotopic to $\tau$.
\end{proposition}

This statement allows us to tie up the loose end from our discussion of (in)visible arcs.
Indeed, for $(X', \lambda')$ satisfying the hypotheses of the Proposition, we have that $\arc_\bullet(X', \lambda', \delta) = \arc_\bullet^{\tau}(X', \lambda')$. That is, our geometric and topological notions of the (in)visible arcs of $(X', \lambda')$ at scale $\delta$ agree.
This allows us to conclude that the extensions of $\taua$ form a nice family of coordinate charts around $(X, \lambda)$; see Corollary \ref{cor:augstable} below.

As with the persistence of arcs (Proposition \ref{prop:persistent}), we first prove a uniform version stability of $\tau$ in the case where only $\lambda$ is allowed to vary, then use this uniformity to allow $X$ to vary as well.

Recall that each branch $b$ of the equilateral train track $\tau = \tau(X, \lambda, \delta)$ has a well-defined length $\ell_\lambda(b)$ as measured along $\lambda$.
Denote by $d_\lambda(b)$\label{ind:dlb} the minimum distance from the midpoint of the branch $b$ to any basepoint on any leaf of $\lambda$ that runs through $b$.
Because $\tau$ is an equilateral train track this is strictly less than $D_\delta$.
Finally, we define the following auxiliary quantity:\label{ind:lxdb}
\[\ell(X, \lambda, \delta) := \min_b \left\{ \frac{\ell_\lambda(b)}{2}, D_\delta - d_\lambda(b) \right\}\]
where the minimum is taken over all the branches of $\tau$.

\begin{lemma}\label{lem:ttstablefixX}
There is a constant $\zeta_{\ref{lem:ttstablefixX}}>0$ such that for any $\lambda'$ with $d_X^H(\lambda, \lambda') < \zeta_{\ref{lem:ttstablefixX}},$
the equilateral train track $\tau(X, \lambda', \delta)$ is isotopic to $\tau$.
This constant depends only on the thickness of $X$ and $\ell=\ell(X, \lambda, \delta)$.
\end{lemma}
\begin{proof}
Throughout the proof, we will work in the universal cover but will mostly suppress this for economy.
Thus, we will use $\cE$ to denote the full preimage under the covering projection of the $\delta$--equilateral neighborhood of $\lambda$ to $\tX$ and $\cO$ to denote the preimage of $\Ol(X)$, and similarly will use $\cE'$ and $\cO'$ for the preimages of $\cE_{\delta}(\lambda')$ and $\cO_{\lambda'}(X)$.

We recall that there is a correspondence between the complements of $\lambda \cup \arc_\circ^{\tau}(\lambda)$, of $\tau$, and of $\lambda' \cup \arc_\circ^{\tau}(\lambda')$.
Our goal is to show that if two 
$P$ and $Q$ complementary to $\lambda \cup \arc_\circ^{\tau}(\lambda)$ are joined by a tie $t$ of $\cO|_{\cE}$,
then the corresponding components $P'$ and $Q'$ 
complementary to 
$\lambda' \cup \arc_\circ^{\tau}(\lambda')$
are joined by a tie $t'$ of $\cO'|_{\cE'}$.
Note that by Lemma \ref{lem:invisiblesubarc}, the regions $P$ and $Q$ are unions of hexagons of $X \setminus (\lambda \cup \arc)$, and an analogous statmenet holds for $P'$ and $Q'$.

The main point of our proof is that $t$ is close to basepoints on $\partial P$ and $\partial Q$ and far away from all other basepoints (c.f. Proposition \ref{prop:dual cell veering}).
Since basepoints and ties change continuously as $\lambda$ varies slightly in the Hausdorff metric 
(Corollary \ref{cor:triplebasepoints} and Proposition \ref{prop:tiestable}, respectively)
this will imply that $t'$ is close to basepoints on $\partial P'$ and $\partial Q'$ and far from all others. Thus $t'$ will actually be a tie of $\cO'|_{\cE'}$ joining $P'$ to $Q'$.

To implement this strategy, let us first specify our candidate $t'$.
The plaques $P$ and $Q$ are adjacent over some branch of $\tau$. 
Consider the leaf of $\cO|_\cE$ projecting down to the midpoint of this branch, and define $t$ to be the subsegment of this leaf that runs between $\partial P$ and $\partial Q$.
That is, we have trimmed off the ends of the leaf that run into the interiors of $P$ and $Q$. 
In the case that $P$ and $Q$ are adjacent over an isolated leaf, we simply take $t$ to be the point on $\lambda$ corresponding to the midpoint of the branch.

Lemma \ref{lem:change base ties still meet} ensures that each leaf of $\cO'$ meeting $t$ also meets both $\partial P'$ and $\partial Q'$; let $t'$ denote the subsegment of any such leaf running between the two geodesics.
In order to use this segment to demonstrate the adjacency of $P'$ and $Q'$ over $\tau(X, \lambda', \delta)$, we must show the following:
\begin{enumerate}
    \item The segment $t'$ is contained in $\cE'$.
    \item It terminates in the $\delta$-thick parts of $P'$ and $Q'$.
\end{enumerate}

To show that $t'$ is contained in $\cE'$, it suffices to prove that for any component $R'$ of $\tX \setminus \tlambda'$ meeting the interior of $t'$, the distance from $t'$ to any basepoint of $R'$ is at least $D_\delta$.
Now by our choices of $t$ and $\ell$, we have that the distance from $t$ to any basepoint of any component $R$ of $\tX \setminus \lambda$ meeting the interior of $t$ is at least $D_\delta + \ell$.
Fix $a \in (0,1)$.
By Proposition \ref{prop:tiestable}, the tie $t'$ is within $O_s(\zeta^a)$ of $t$, and by Corollary \ref{cor:persist_notallarcs}, every basepoint of $\lambda$ is within $O_s(\zeta^2)$ of a corresponding basepoint of $\lambda'$ so long as $\zeta$ is taken less than the cutoff $\zeta_{\ref{cor:persist_notallarcs}}(s)$.
\footnote{To be more concrete, the correspondence is between basepoints of the hexagons of $X \setminus (\lambda \cup \arc)$ and $X \setminus (\lambda' \cup \arc')$ making up the same piece of $S \setminus (\tau \cup \arc(\zeta))$, where $\arc(\zeta)$ is the filling persistent subsystem of $\arc_\bullet$ guaranteed by Corollary \ref{cor:persist_notallarcs}.}
Therefore, we have that the distance from $t'$ to any basepoint of any  $R'$ that meets its interior is at least
\[ D_\delta + \ell - O_s(\zeta^a + \zeta^2),\]
which is clearly greater than $D_\delta$ for $\zeta$ small enough.

A similar argument shows that $t'$ terminates in the $\delta$-thick parts of $P'$ and $Q'$.
Indeed, since $t$ is a tie of $\cE$, it is within $D_\delta$ of some basepoint $p \in \partial P$, and by our choice of $t$ and $\ell$ is actually within $D_\delta - \ell$ of $p$.
The tie $t'$ is within $O_s(\zeta^a)$ of $t$, and there is a corresponding basepoint $p'$ of $P'$ within $O_s(\zeta^2)$ of $p$, so altogether we see that $t'$ is within 
\[D_\delta - \ell + O_s(\zeta^a + \zeta^2)\]
of $p'$, which again is less than $D_\delta$ for small enough $\zeta.$

Thus, we have shown that there is a tie of $\cE'$ connecting any two complementary components of $X \setminus (\lambda' \cup \arc_\circ^\tau(\lambda'))$, that is, every branch of $\tau$ persists in $\tau':=\tau(X, \lambda', \delta)$.
To show that $\tau'$ has no new branches, consider the dual complex $\mathsf{C}$ to $\tau$ on $X$. If $\tau$ is not filling, then $\mathsf{C}$ is not a cellulation of $X$, but rather of the quotient space obtained by collapsing each component of $X \setminus \cE$ to a point (if $\tau$ is filling then this is just a homotopy equivalence).
In any case, since $\tau$ is trivalent, $\mathsf{C}$ is a triangulation. 
If we consider the dual complex $\mathsf{C}'$ to $\tau'$ on $X$, since there is a correspondence between complementary components we see that $\mathsf{C}'$ is a cellulation of the same space as $\mathsf{C}$ with the same vertex set. 
Since each branch of $\tau$ persists in $\tau'$, we also see that $\mathsf{C}'$ must contain every edge of $\mathsf{C}$.
But now since $\mathsf{C}$ is a triangulation 
(and no switches of $\tau'$ have valence $<3$) it must be that $\mathsf{C}'=\mathsf{C}$.
Therefore $\tau'$ has no more branches than $\tau$ and the two train tracks are isotopic as claimed.
\end{proof}

In light of the uniform threshold obtained in Lemma \ref{lem:ttstablefixX}, it suffices to produce a neighborhood of $X$ wherein the geometry of the equilateral neighborhood of $\lambda$ varies in a controlled way.

\begin{lemma}\label{lem:ttstablevarX}
Fix $\ell < \ell(X, \lambda, \delta)$. 
Then there is an open neighborhood $B_{\ref{lem:ttstablevarX}}(\ell)$ of $X$ such that for every $X' \in B_{\ref{lem:ttstablevarX}}(\ell)$,
the equilateral train track $\tau(X', \lambda, \delta)$ is isotopic to $\tau$
and has $\ell(X', \lambda, \delta) > \ell$.
\end{lemma}

\begin{proof}
This is a consequence of the fact that the geometry of an equilateral train track depends continuously on the complementary subsurface $X \setminus \lambda$ and the shearing data, which both depend continuously on $X$.

Using shear-shape coordinates, we can be more explicit. 
It turns out that the length of a branch $b$ of $\tau$ can be recorded as a simple linear combination of the values of the shear-shape cocycle $\sigl(X)$ and the number $D_\delta$; the precise formula depends on the type of the branch (small, mixed, or large) and the configuration of plaques around $b$.
We have illustrated the possible cases in Figure \ref{fig:branches_length}.
When $\lambda$ is maximal, this is equivalent to the linear isomorphism from shear coordinates to tangential coordinates (quotiented by the ``switch vectors''); compare with \cite[pp. 44-45]{Th_stretch}.

\begin{figure}[ht]
\includegraphics[width=\linewidth]{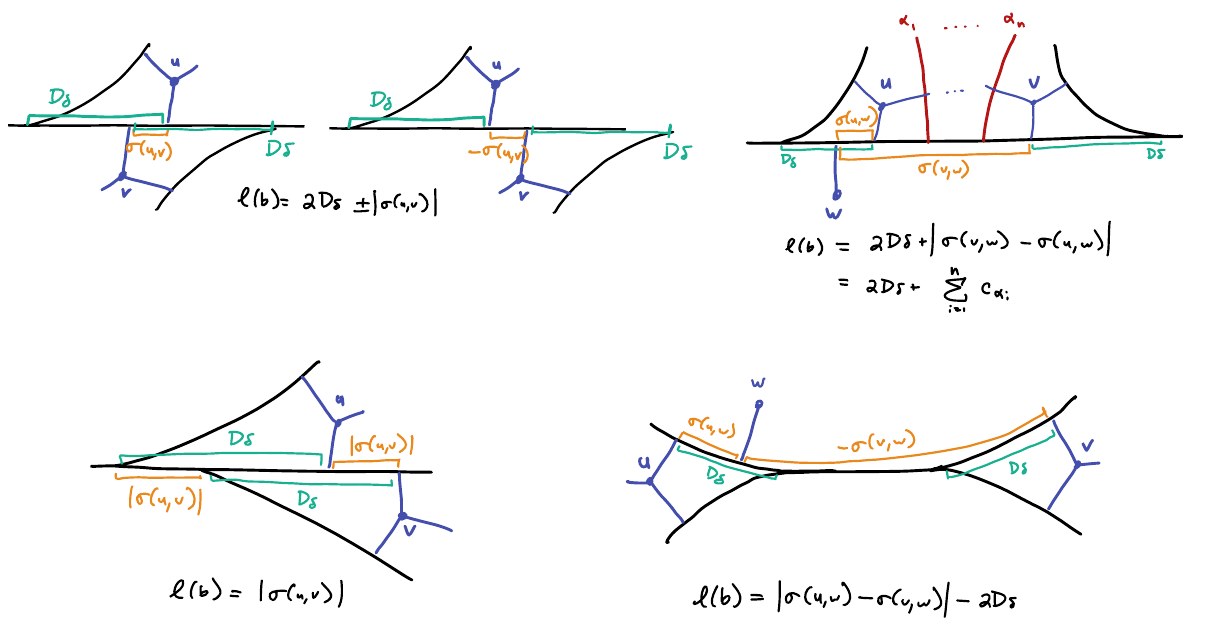}
\caption{The length of a branch $b$ in terms of $D_\delta$ and shear parameters. The top row expresses the possibilities when $b$ is small, while the bottom row covers the cases where $b$ is mixed or large. The precise formulas do not matter; what is important is just that they are continuous in the shear and arc data.}
\label{fig:branches_length}
 \end{figure}

In particular, the length of $b$ depends continuously on the values of $\sigl(X)$, thought of as a weight system on $\taua$.
Similarly, the distance between the midpoint of a branch $b$ of $\tau$ and the basepoints on $\lambda$ can also be computed in terms of these data.
We may therefore take $B_{\ref{lem:ttstablevarX}}(\ell)$ to be the preimage by $\sigma_\lambda$ of a suitably small neighborhood of $\sigma_\lambda(X)\in W(\taua)$, chosen such that it does not meet the boundary of the image of Teichm\"uller space in shear-shape coordinates and such that no branch of $b$ becomes shorter than $2\ell$, hence persists in $\tau(X', \lambda, \delta)$.
That $\tau(X', \lambda, \delta)$ is isotopic to $\tau$ follows the same argument as in Lemma \ref{lem:ttstablefixX}.
\end{proof}

The Proposition quickly follows by combining Lemmas \ref{lem:ttstablefixX} and \ref{lem:ttstablevarX}.

\begin{proof}[Proof of Proposition \ref{prop:ttstable}]
Fix $\ell < \ell(X, \lambda, \delta)$ and consider the neighborhood $B_{\ref{lem:ttstablevarX}}(\ell)$ of $X$.
By shrinking $B_{\ref{lem:ttstablevarX}}(\ell)$ to a smaller neighborhood, we may also ensure that for any $X'$ in this neighborhood,
\begin{enumerate}
    \item $X'$ is $s/2$ thick, and
    \item the Hausdorff metrics on $\mathcal{GL}_X$ and $\mathcal{GL}_{X'}$ are H{\"older} comparable with uniform constants.
\end{enumerate}
Fix any $X'$ from the neighborhood above.
By Lemma \ref{lem:ttstablevarX}, we know that
$\tau(X', \lambda, \delta)$ is isotopic to $\tau$ and $\ell (X', \lambda, \delta) > \ell$.
Now for any $\lambda'$ Hausdorff-close enough to $\lambda$, the uniformity in Lemma \ref{lem:ttstablefixX} together with the conditions (1) and (2) above guarantees that $\tau(X', \lambda', \delta)$ is isotopic to $\tau(X, \lambda', \delta)$.
\end{proof}

Since nearby pairs $(X, \lambda)$ and $(X', \lambda')$ have the same equilateral train track (Proposition \ref{prop:ttstable}) and the same visible arc systems (up to adding in arcs of weight 0, Proposition \ref{prop:persistent}), we would like to conclude that the standard smoothings of the augmented train tracks $\taua(X, \lambda, \delta)$ and $\taua(X', \lambda', \delta)$ are the same.
While this is morally true, multiple arcs of $\arc_\bullet$ can share the same endpoint on $\tau$,
which is an unstable configuration.
Instead, we appeal to the notion of slide-equivalence of smoothings developed in Section \ref{subsec:smoothings}.

\begin{corollary}\label{cor:augstable}
If $X'$ is sufficiently close to $X$ and $\lambda'$ is sufficiently Hausdorff-close to $\lambda$,
then the standard smoothings of 
$\tau(X,\lambda,\delta) \cup \arc_\bullet (X',\lambda',\delta)$
and 
$\tau(X',\lambda', \delta) \cup \arc_\bullet(X',\lambda', \delta)$
are slide-equivalent.
\end{corollary}

In particular, if we let $\taua$ denote the standard smoothing  of $\tau(X, \lambda, \delta) \cup \arc_\bullet(X, \lambda, \delta)$ and $\taua'$ the standard smoothing of $\tau(X', \lambda', \delta) \cup \arc_\bullet(X',\lambda', \delta)$, 
then there is a natural identification of $W(\taua)$ with the subspace of $W(\taua')$ corresponding to setting the weights on all arcs of $\arc_\bullet' \setminus \arc_\bullet$ equal to 0.
Since $\sigma_{\lambda'}(X')$ and $\lambda'$ are both carried on $\taua'$, this means that every nearby pair $(X', \lambda')$ can be represented by a pair of weight systems carried on one of the finitely many (slide-equivalence classes of) extensions of $\taua$.

\subsection{Common cellulations}\label{subsec:cellstab}
Corollary \ref{cor:augstable} already realizes the shear-shape cocycles $\sigl(X)$ and $\sigma_{\lambda'}(X')$ in a common space, so we could (and do) phrase our main continuity theorem in terms of the difference between the corresponding weight systems on the augmented geometric train track $\taua(X', \lambda', \delta)$.

However, in order to show a more direct connection with flat geometry, we now upgrade our argument from above to show that the quadratic differentials $q:=\cO(X, \lambda)$ and $q':=\cO(X', \lambda')$ have comparable cellulations by saddle connections, also allowing us to phrase our continuity theorem in terms of their periods.
As in \S \ref{subsec:dualcell}, let $\mathsf{T}$ denote the cellulation of $\cO(X, \lambda)$ dual to $\taua$.

\begin{proposition}[Stability of cellulations]\label{prop:cellstable}
If $X'$ is sufficiently close to $X$ and $\lambda'$ is sufficiently Hausdorff-close to $\lambda$,
then $\cO(X', \lambda')$ has a cellulation $\mathsf{T}'$ by saddle connections that refines $\mathsf{T}$.
\end{proposition}

As an immediate consequence, there is a finite set of period coordinate charts, each of which meets the chart corresponding to $\mathsf{T}$ (possibly containing it as a boundary subspace), such that $\cO(X', \lambda')$ is contained in one of these charts for all nearby $X'$ and all Hausdorff-close $\lambda'$.

\begin{proof}
In order to talk about cellulations on the different surfaces, we must first establish that we have a consistent identification between them.
Recall from Definition \ref{def:topcollapse} the notion of topological collapse map.
\footnote{We cannot invoke our geometric notion of collapse maps (Lemma \ref{lem:collapsemaps}), as that presupposes that the two differentials have comparable geometry.}

\begin{claim}\label{claim:topcollapsestable}
The carrying maps $\lambda \prec \tau$ and $\lambda' \prec \tau$ induce a natural isotopy class of topological collapse from $q' = \cO(X', \lambda')$ to $q = \cO(X, \lambda)$.
\end{claim}
\begin{proof}
Let $Z$ denote the set of zeros of $q$ and $Z'$ the zeros of $q'$. 
By Lemma \ref{lem:dualtt_equitt}, a standard smoothing of train track $\tau = \tau(X, \lambda, \delta)$ together with its visible arc system $\arc_{\bullet}$ is a deformation retract of $q \setminus Z$.
By Lemma \ref{lem:ttstablevarX}, we know that $\tau(X', \lambda', \delta) = \tau$ and via Proposition \ref{prop:persistent}, its visible arc system contains $\arc_\bullet$.
Thus, we have an identification of $q \setminus Z$ with a subsurface of $q' \setminus Z'$ (compare Corollary \ref{cor:augstable}). 

The multiple-to-one correspondence between the hexagons of
$X' \setminus (\lambda' \cup \arc')$ and $X \setminus (\lambda \cup \arc)$ induces a multiple-to-one correspondence between $Z'$ and $Z$, where 
the zeros in each fiber are connected by a tree of short horizontal saddle connections, each corresponding to an arc of $\arc_\bullet' \setminus \arc_\bullet$. The collapse map is then obtained by collapsing these saddles (dually, forgetting the extra arcs).
\end{proof}

Consider a non-horizontal saddle connection of $\mathsf{T}$; there is a dual tie $t$ running between two 
boundary geodesics $g$ and $h$ of $\lambda$ and connecting regions of $S \setminus \taua$ corresponding to hexagons $u$ and $v$ of $X \setminus (\lambda \cup \arc)$.
Since $\lambda'$ is close to $\lambda$, Lemma \ref{lem:Hausdfullcarry} allows us to pick out corresponding boundary geodesics $g'$ and $h'$ of $\lambda'$.
Proposition \ref{prop:persistent} allows us to identify $u$ with a union of hexagons of $X' \setminus (\lambda' \cup \arc')$; let $u'$ denote any of these hexagons with boundary on $g'$, and similarly define $v'$. 

In the proofs of Lemmas \ref{lem:ttstablefixX} and \ref{lem:ttstablevarX} we proved that there is a tie $t'$ of the equilateral neighborhood of $\lambda'$ on $X'$ that terminates within $D_\delta$ of the basepoints of $u'$ and $v'$. 
Corollary \ref{cor:moresaddles} now implies that the corresponding arc between zeros of $\cO(X', \lambda')$ is the diagonal of a nonsingular quadrilateral, and in particular a saddle connection.
Thus every saddle connection of $\mathsf{T}$ is realized on $\cO(X', \lambda')$. Adding in the horizontal saddles dual to the visible arcs $\arc_\bullet(X', \lambda', \delta)$ completes this collection to a cellulation that refines $\mathsf{T}$.
\end{proof}

Now that we have established the existence of a period coordinate chart containing both $q'$ and $q$, it is not hard to show that the (real parts of the) periods of $q$ and $q'$ are close.

\begin{theorem}\label{thm:contcell}
For every $X \in \T_g$, every $\lambda \in \GLcr$, and every $a \in (0,1)$, there is a threshold $\zeta_{\ref{thm:contcell}} >0$ such that the following holds.
For any $\zeta < \zeta_{\ref{thm:contcell}}$, there is an open neighborhood $B_{\ref{thm:contcell}}(\zeta)$ of $X$
such that for any $X' \in B_{\ref{thm:contcell}}(\zeta)$ and any $\lambda'$ with $d_X^H(\lambda, \lambda') < \zeta$, we have that
\[\| \sigl(X) - \sigma_{\lambda'}(X') \|_{\taua'} 
= O_s(\zeta^a),\]
where $\taua'$ denotes the standard smoothing of $\tau(X',\lambda',\delta) \cup \arc_\bullet(X',\lambda', \delta)$.
\end{theorem}

The cellulation $\mathsf T'$ of $\cO(X',\lambda')$ by saddle connections dual to $\taua(X',\lambda', \delta)$ is a refinement of the cellulation $\mathsf T$ of $\cO(X,\lambda)$ by saddle connections dual to $\taua(X,\lambda,\delta)$.
Using duality between train track coordinate and period coordinate charts developed in \S\ref{sec:ttperiod coords}, Theorem \ref{thm:contcell} can also be phrased by saying that
\[\left| \Re[\hol_{\cO(X, \lambda)}(e)]_+ - \Re[\hol_{\cO(X', \lambda')}(e')]_+ \right| = O_s(\zeta^a)\]
for every saddle connection $e$ of $\mathsf{T}$ and any saddle connection $e'$ of $\mathsf{T}'$ that realizes $e$.

\begin{proof}
For $\lambda = \lambda'$, this is just the content of the main theorem of \cite{shshI}.

We will further specify our thresholds later in the proof, but throughout we take $\zeta_{\ref{thm:contcell}}$ and a neighborhood of $X$ small enough to be able to apply Corollary \ref{cor:augstable}, Proposition \ref{prop:tiestable}, and Proposition \ref{prop:persistent}.

Let $t$ be a tie of an equilateral train track connecting two boundary geodesics $g$ and $h$ of $\lambda$. 
Let $u$ and $v$ be hexagons of $X \setminus (\lambda \cup \arc)$ with basepoints $p$ on $g$ and $q$ on $h$, respectively.
Then the shear $\sigl(X)(u,v)$ may be computed by adding the (signed) distances from $p$ to $t$ and from $t$ to $q$ along $\lambda$ (see \cite[Remark 13.3]{shshI} for a discussion of sign conventions).

There are boundary leaves $g''$ and $h''$ of $\lambda'$ on $X'$ corresponding to $g$ and $h$, respectively. 
By Proposition \ref{prop:persistent} item (1), we can pick a component $u''$ of $X'\setminus(\lambda'\cup \arc(X',\lambda'))$ corresponding to $u$ and bordering $g''$ (there may be multiple options if the geodesics of $u$ are in an unstable configuration). Similarly pick $v''$ bordering $h''$.
Let $p''$ and $q''$ denote the basepoints of these hexagons on $g''$ and $h''$, respectively, and choose a segment $t''$ of a leaf of $\cO_{\lambda'}(X')$ that runs between $g''$ and $h''$ (in the universal cover).

We would like to simply invoke the estimates of Propositions \ref{prop:persistent} and \ref{prop:tiestable} to compare $d(u, t)$ with $d(u'', t'')$ and similarly $d(t,v)$ with $d(t'',v'')$, as in the proof of Lemma \ref{lem:ttstablefixX}.
The main theorem of \cite{shshI} tells us that $
\sigl(X)$ and $\sigl(X')$ are close, while the argument just above gives us that $\sigl(X')$ is close to $\sigma_{\lambda'}(X')$.

Unfortunately, we cannot immediately conclude the Theorem, because the arc systems for $\lambda$ on $X'$ and $\lambda'$ on $X'$ may cross, preventing us from realizing $\sigl(X')$ and $\sigma_{\lambda'}(X')$ in the same train track chart. 
We will briefly explain how to use Proposition \ref{prop:persistent} to obtain the desired bounds, regardless.
\medskip

Item (1) of Proposition \ref{prop:persistent} allows us to conclude that for close enough pairs, the visible arc systems $\arc_\bullet(X', \lambda, \delta)$ and 
$\arc_\bullet(X', \lambda', \delta)$ both contain $\arc_\bullet$ as a filling sub-arc system. 
Let $P$ be the complementary component of $S\setminus(\tau\cup \arc_\bullet)$ corresponding to $u$.
By item (3) of Proposition \ref{prop:persistent}, 
the configurations of geodesics
$G_P(X',\lambda)$ and $G_P(X',\lambda')$ are both $O(\zeta^{3/2})$-equidistant.
Thus the diameter of the collection of centers of circles inscribed in the triples of those geodesics project to any boundary geodesic with diameter at most $O(\zeta^{3/2})$.
This also gives us a dictionary between centers of triples in $G_P(X,\lambda)$,  $G_P(X',\lambda)$, and $G_P(X',\lambda')$, hence their projections to the corresponding boundary leaves of $\lambda$ and $\lambda'$, respectively.
\medskip

Now we are ready to apply the strategy we outlined earlier in the proof.
Let $g'$ and $h'$ be the boundary leaves of $\lambda$ corresponding to $g$ and $h$, respectively, but realized in the metric $X'$.
Let $u'$ be one of the hexagons complementary to $\lambda\cup \arc(X',\lambda)$ on $X'$ contained in the nearly equidistant configuration $G_P(X',\lambda)$ whose boundary meets $g'$, and let $p'$ be the projection of $u'$ to $g'$.
There are two boundary geodesics of $\lambda\subset X'$ other than $g'$ that are closest to $u'$.  
Let $(u')'$ denote the center in $X'$ of the corresponding configuration in $\lambda'$, and let $(p')'\in g''$ from the projection of $(u')'$.
Define $v'$, $q'\in h'$, $(v')'$, and $(q')' \in h''$ analogously for $v$. 
See Figure \ref{fig:refinements}.

\begin{figure}
    \centering   
    \includegraphics[width=\linewidth]{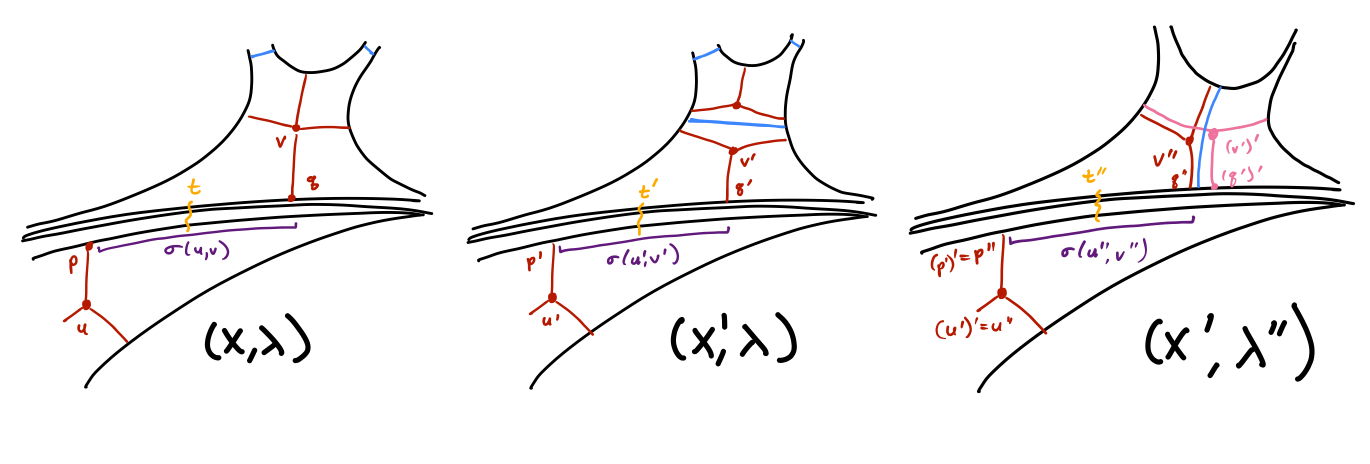}
    \caption{Corresponding hexagons and basepoints in the proof of Theorem \ref{thm:contcell}.}
    \label{fig:refinements}
\end{figure}

Using the main result of \cite{shshI}, shrinking $B_{\ref{prop:persistent}}(\zeta)$ as necessary, we can ensure that 
\begin{equation}\label{eqn:contest1}
| \sigl(X)(u,v) - \sigl(X')(u', v')| \le \zeta^a
\end{equation}
for every $X' \in B_{\ref{prop:persistent}}(\zeta)$.
Choose a tie $t'$ of $\tau(X', \lambda, \delta)$ that meets the tie $t''$ of $\tau(X', \lambda', \delta)$.
As described above, we can write
\[\sigl(X')(u', v') = \vec{d}(p', t' \cap g') + \vec{d}(t'\cap h', q')\]
where $\vec{d}$ denotes the signed distance.
Similarly, we can express 
\[\sigma_{\lambda'}(X')(u'', v'') = \vec{d}(p'', t'' \cap g'') + \vec{d}(t'' \cap h'', q'').\]
By Proposition \ref{prop:tiestable}, we know that the Hausdorff distance between $t'$ and $t''$ is $O_s(\zeta^a)$.
By Corollary \ref{cor:triplebasepoints}, we know that $p'$ and $(p')'$ are $O_s(\zeta^2)$ close, and as discussed earlier,  $d(p'', (p')')$ is $O(\zeta^{3/2})$. 

The triangle inequality now implies that
\begin{equation}\label{eqn:contest2}
|\sigl(X')(u', v') - \sigma_{\lambda'}(X')(u'', v'')| = O_s(\zeta^\alpha) + O_s(\zeta^2) + O(\zeta^{3/2}).
\end{equation}
Although the implicit constant in $O(\zeta^{3/2})$ depends on $X$ and $\lambda$, 
we can take $\zeta_{\ref{thm:contcell}}$ small enough (depending on $X$ and $\lambda$) to ensure that $\zeta^{3/2}$ is much smaller than $\zeta^{a}$ for any $\zeta < \zeta_{\ref{thm:contcell}}$ and $a\in (0,1)$.
Thus for small enough $\zeta$ the quantity in \eqref{eqn:contest2} is $O_s(\zeta^a)$.
The main result now follows by combining \eqref{eqn:contest1} and \eqref{eqn:contest2}.
\end{proof}

One direction of Theorem \ref{maintheorem:O and Oinv continuous} is an immediate corollary.

\begin{proof}[Proof of the continuity of $\cO$]
Suppose that $(X, \lambda)$ and $(X', \lambda')$ are close in $\PoM_g$ and that $\lambda$ and $\lambda'$ are Hausdorff-close on $X$.
Proposition \ref{prop:cellstable} ensures that there is a single period coordinate chart containing the corresponding quadratic differentials $q$ and $q'$.
Theorem \ref{thm:contcell} gives an estimate on the difference between the real parts of each period, while the imaginary parts are computed by the weights deposited by $\lambda$ and  $\lambda'$ on the branches of $\tau$. Since $\lambda$ and $\lambda'$ are close in $\ML$ these weights are close, completing the proof.
\end{proof}

\begin{remark}\label{rmk:cont_unif}
As with Proposition \ref{prop:persistent}, it is possible to get a uniform threshold in Theorem \ref{thm:contcell} for all $s$-thick surfaces, but to do so one must allow for a much wider variety of changes to the structure of the train track (or the dual cellulation).
These issues can all be ameliorated by only remembering a sub-cellulation corresponding to long branches and more carefully tracking how cellulations collapse to each other, but we have chosen to forgo such discussions as they require a great deal of delicacy (significantly more than what appears here!) and are not needed for the main results of this paper.
\end{remark}

Thurston proved that the space of chain recurrent geodesic laminations with its Hausdorff metric coming from a hyperbolic metric $X$ has Hausdorff dimension 0 \cite{Th_stretch}.
Using Theorem \ref{thm:contcell} and the fact that an $a$-H\"older map can only increase Hausdorff dimension by a factor of $1/a$ we obtain the following.

\begin{corollary}\label{cor:Hdim O(X)}
For any $X \in \T_g$, the set 
\[ \{ \eta \in \MF \,: \, \eta = \Ol(X)
\text{ for some }
\lambda \in \GLcr\}\]
has Hausdorff dimension $0$.
\end{corollary}

In sharp contrast, the theorem of Hubbard and Masur \cite{HubMas} states that every measured foliation is the vertical foliation of some quadratic differential that is holomorphic on $X$.

\part{Continuity of the inverse}\label{part:inverse continuity}

We now prove continuity of $\cO^{-1}$, assuming Hausdorff converegence of the supports of the horizontal foliations (with respect to some hyperbolic structure).
The main theorem of \cite{shshI} establishes continuity along leaves of the unstable foliation; the goal of this part is therefore to control the geometry of $\cO^{-1}$ in the stable direction.
Complicating our proof is the fact that the product structures of $\PT_g$ and $\QT_g$ are not mapped to each other via $\cO$. However, given Hausdorff closeness of their horizontals, we can make quantitative estimates comparing them; see Theorem \ref{thm:cont_inverse_stable} below.

The other thing we must be careful about is that, in order to build estimates in a neighborhood of a point using the local product structure, we need estimates in the unstable leaves that are uniform as we vary in the stable leaf.
Obtaining this control requires retreading some of the arguments in \cite{shshI}, but now using the language developed in Part \ref{part:continuity} to make uniform, quantitative statements.

\section{Estimates on shape-shifting}
\label{sec:shapeshift_est}
For a pair $(X,\lambda)$, a {\bf combinatorial deformation}\label{ind:combdef} $\ac$ is an element of $T_{\sigma_\lambda(X)}\SH^+(\lambda)$ such that $\sigl(X) + \ac \in \SH^+(\lambda)$; this can also be thought of as a weight system on a train track carrying $\sigl(X)$ (see \S \ref{subsec:reference_tts} and \cite[\S 14]{shshI}).
In \cite[\S\S14-15]{shshI}, for small enough $\ac$ we constructed a \emph{shape-shifting cocycle}\label{ind:shapeshift}
$\varphi_\ac: \pi_1(X) \to \PSL_2 \RR$, which encodes a small deformation of the hyperbolic structure of $X$.
More precisely, if one chooses $x\in X\setminus \lambda$ and identifies $\pi_1(X,x)$ as a subgroup of $\PSL_2 \RR$ via the holonomy representation $\rho$,
then the rule 
\[\rho_\ac: \gamma \in \pi_1(X,x) \to \varphi_\ac(\gamma) \cdot \rho(\gamma)\]
defines a group homomorphism and a new complete hyperbolic structure $X_\ac = \widetilde X/\rho_\ac$ satisfying $\sigma_\lambda(X_{\ac}) = \sigma_\lambda(X) + \ac$.
In this section, we recall the construction of $\varphi_\ac$ and extract some estimates
bounding the size of the difference between $\rho$ and $\rho_{\ac}$.
The main result of this section is Theorem \ref{thm:shapeshift_distance_small}, which we state after establishing a few preliminaries and setting notation.

\subsection{Left-invariant metrics on $\PSL_2 \RR$}\label{subsec:metrics on PSL_2R}
Given a point $\hat x\in \HH^2$, there is a complete left invariant metric $d_{\hat x}$ on $\PSL_2 \RR$ given by the rule\label{ind:PSL2Rmetric}
\[d_{\hat x}(A, B ) = \sup_{x \in \HH^2} d(A(x),B(x)) e^{-d(\hat x, x)},\]
where $d$ is the hyperbolic metric on $\HH^2$.
For a fixed basepoint $\hat x$ and $A \in \PSL_2 \RR$, denote by $\|A\| : = d_{\hat x} (id, A)$.
This metric has the following nice geometric properties \cite[Lemma 3.14]{PapadopTheret:Teich_Thurston}:
for a parabolic element $P\in \PSL_2 \RR$, 
\begin{equation}\label{eqn:parabolic}
        \|P\|\le \ell(h)
\end{equation}
where $h$ is the horocyclic segment joining $\hat x$ and $P\hat x$ based at the parabolic fixed point of $P$.
For a different basepoint $\hat y$, we have the inequality 
\begin{equation}\label{eqn: PSL2R metric basepoint}
d_{\hat x} (A,B) \le d_{\hat y} (A,B) e^{d(\hat x, \hat y)}.    
\end{equation}
A short computation using the definition and the previous property therefore yields
\begin{equation}\label{eqn:transvection_estimate}
    \|T_g^\ell\| \le \ell e^{d(\hat x, g)},
\end{equation}
where $T_g^\ell$ is the translation distance $\ell$ along the (oriented) geodesic $g$.
\medskip

For a given finite generating set $\Gamma\subset \pi_1(S)$ 
and $X, X'\in \T(S)$, say that $d_\Gamma(X,X')<\epsilon$\label{ind:dGamma}
if there are holonomy representations $\rho, \rho' : \pi_1(S) \to \PSL_2\RR$ compatible with the markings such that
\[\max_{\gamma\in \Gamma} d(\rho(\gamma), \rho'(\gamma)) < \epsilon.\]
The topology for $\T(S)$ generated by sets of the form $B_\Gamma(X,\epsilon) = \{X' : d_\Gamma(X, X')<\epsilon\}$ is equivalent to the usual topology \cite[Proposition 10.2]{FarbMarg}.
It is also not difficult to see that if $X_n \to X$, then  $B_\Gamma(X_n,\epsilon) \subset B_\Gamma(X,2\epsilon)$ for all $\epsilon>0$ and $n$ large enough.

For any generating set $\Gamma$ of $\pi_1(S)$, define\label{ind:gensetlength}
\[L_{\rho}(\Gamma)= \max_{\gamma \in \Gamma}\{\|\rho(\gamma)\|\}\]
 and $L_X(\Gamma)$ as the infimum of $L_{\rho}(\Gamma)$ as $\rho$ varies in the $\PSL_2 \RR$ conjugacy class of holonomy representations for $X$.

\subsection{Reference train tracks}\label{subsec:reference_tts}
Let us fix a thickness parameter $s>0$ and choose $\delta\le \delta_{\ref{lem:equittswork}}(s)$.
Let $(X,\lambda)$ be an $s$-thick chain recurrent pair and set $\tau = \tau(X,\lambda,\delta)$ to be the $\delta$--equilateral train track. 
Here are the relevant arc systems that we will consider:
\begin{itemize}
    \item The visible arc system $\arc_\bullet = \arc_\bullet(X,\lambda,\delta)$
    \item The invisible arc system $\arc_\circ = \arc_\circ(X,\lambda,\delta)$
    \item A completion $\arc_*$\label{ind:arccompletion} of $\arc_\bullet$ to a maximal filling arc system in the complement of $\tau$.
\end{itemize}
Let $\taua$ denote a standard smoothing of $\tau\cup \arc_*$.
Then $W( \taua )$ identifies a linear fragment of $T_{\sigma_\lambda(X)}\SH^+(\lambda)$,
in that any $\ac\in T_{\sigma_\lambda(X)}\SH^+(\lambda)$ can be represented as a weight system on some such $\taua$.
For brevity and because our choice of $\taua$ is fixed throughout this section, we will abbreviate $\|\ac\|_{\taua}$ to $\|\ac\|$.

With all notation as in \S\ref{subsec:metrics on PSL_2R}, we can now state our main result in this section.

\begin{theorem}\label{thm:shapeshift_distance_small}
Given $s>0$ there is a $\xi_{\ref{thm:shapeshift_distance_small}}$ such that the following holds.
For any $s$-thick $(X, \lambda)$, any finite generating set $\Gamma$ for $\pi_1(S)$, any combinatorial deformation $\ac$ with $\|\ac\|\le\xi_{\ref{thm:shapeshift_distance_small}}$, and any $a \in (0,1)$,
we have 
\[d_\Gamma (X, X_\ac) = O(\|\ac\|^a),\]
where $\sigma_{\lambda}(X_\ac) = \sigma_{\lambda}(X)+\ac$ and the implicit constant depends on $s$, $\delta$, $a$, and $L_X(\Gamma)$.
\end{theorem}

\begin{proof}
    This follows directly from Proposition \ref{prop:shapeshift}, which for each $\gamma \in \Gamma$ gives a bound 
    \[d_{\hat x}(\rho(\gamma),\rho_\ac(\gamma)) = O(\|\ac\|^a) \]
    with implicit constant depending on the parameters $s$, $\delta$, $a$, and $\| \rho(\gamma)\|$. 
\end{proof}

The rest of the section is devoted to formulating and proving Proposition \ref{prop:shapeshift}.

\subsection{Shapeshifting in the proto-spikes}
We  work in the universal cover and lift all relevant objects, but as usual we do not decorate with tildes.\footnote{Also note that our notation is a slightly simplified from \cite{shshI}; the goal is to make this section more readable without going into too much technical detail.}
Let $b$ be a branch of $\tau$ 
and let $k_b$ be a geodesic oriented transversal,\label{ind:transversal} i.e., $k_b$ meets $b$ transversely and meets no other branch of $\taua$.
In this subsection, we explain how to associate a shape-shifting transformation $\varphi_{\ac}(k_b)$ to $k_b$.

There is a packet of geodesics of $\lambda$ collapsing to $b$. 
Let $(\mathcal V, \le )$\label{ind:branchprotospikes} denote the linearly oriented set of complementary proto-spikes of $\lambda$ encountered along $k_b$.
There is a finite subset of $\mathcal V$ corresponding to the proto-spikes bounded by an arc of $\arc_\circ$.
Each proto-spike $V\in \mathcal V$ is comprised of a pair of geodesics $g_-$ and $g_+$ which may or may not be connected by
an arc $\alpha \in \arc_\circ$.
To each $V$, we associate the elementary shaping deformation\label{ind:elemshaping}
\begin{equation}\label{eqn:shaping}
A(V) = \begin{cases}
T_{g_-}^{f_{X,\ac}(V)} \circ T_{g_+}^{-f_{X,\ac}(V)}, & \text{$V$ is a spike} \\
T_{g_-}^{f_{X,\ac}(V)} \circ T_{\alpha}^{\ell_\ac(\alpha)} \circ T_{g_+}^{-f_{X,\ac}(V)}, & \text{$V$ is bounded by an arc $\alpha\in \arc_\circ$},
\end{cases}
\end{equation}

The functions $f_{X,\ac}(V)$ relate the ``sharpness''\label{ind:sharpnessfunctions} of the proto-spike $V$ on $X$ and the deformed proto-spike on $X_\ac$ and are defined in \cite[\S14]{shshI} and the function $\ell_{\ac}(\alpha)$ records the difference in the lengths of the orthogeodesic representatives of $\alpha$ on $X$ and $X_{\ac}$.
The shaping deformations $A(V)$ are continuous on the space of proto-spikes:  as the length of an arc tends to zero, the shaping deformations converge to that of a spike. For details, consult  Lemmas 14.6 and 14.15 and Remark 14.17 of  \cite{shshI} as well as the surrounding discussion.

We remark that the following Lemma holds for all arcs of $\arc$, and we will use it twice: once for the invisible arcs (Lemma \ref{lem:basic_estimate_protospikes}) and once for the visible ones (Lemma \ref{lem:estimate_spine}).

\begin{lemma}\label{lem:uniform_spikes}
For any $s$-thick $(X, \lambda)$ and any $\xi>0$, if $\ac$ is a combinatorial deformation with $\|\ac\|\le\xi$, we have that 
\[|f_{X, \ac}(V)| = O(\|\ac\|) \text{ and } |\ell_{\ac}(\alpha)| = O(\|\ac\|)\]
for every proto-spike $V$ and every $\alpha \in \arc_*$, 
where the implicit constant depends on $\xi$, $s$, and $\delta$.
\end{lemma}

\begin{proof} 
The functions $f_{X,\ac}$ and $\ell_\ac$ are explicit analytic functions of the geometry of $X\setminus \lambda$ and $\ac$, as exhibited in \cite[Lemma 14.6 and Lemma 14.15]{shshI} and \cite[\S2.4]{Mondello}.
The lemma is thus a consequence of compactness of the geometry of $X\setminus \mathcal E_\delta(\lambda)$ as $(X,\lambda)$ varies over the $s$-thick part of moduli space and continuity of $V\mapsto f_{X,\ac}(V)$ in the geometric topology on basepointed proto-spikes $V$, which is also compact. 
Compare with the proof of Lemma \ref{lem:bounded_neighborhood}.
\end{proof}

To each proto-spike, we now associate the {\bf elementary shape-shift} \label{ind:elemshsh} deformation 
\[\varphi(V) = T_{g_-}^{\ac(k_V)}\circ A(V) \circ T_{g_+}^{-\ac(k_V)} \]
where $k_V$ is an initial subsegment of $k_b$  terminating in the interior of $V$.

The following bound uses Lemma \ref{lem:uniform_spikes}, but can otherwise be extracted from the proof of Lemma 14.10 of \cite{shshI}.
We emphasize that it is valid for both spikes and proto-spikes bounded by invisible arcs.

\begin{lemma}\label{lem:basic_estimate_protospikes}
Let all notation be as above and let $\hat x \in \HH^2$ be the initial point of $k_V$.
For the metric $d_{\hat x}$, there is a constant $B = B(\xi, s, \delta)$ such that 
\[\|\varphi(V)\| \le Be^{\ell(k_V)}\ell(k_b\cap V)\left( e^{|\ac(k_V)| + O(\|\ac\|)} - 1\right ). \]  
\end{lemma}

\begin{proof}
If $V$ is a spike, then $\varphi(V)$ is a parabolic element whose fixed point is the end of $V$, so by \eqref{eqn:parabolic} it suffices to bound the length of the horocyclic arc 
between $\hat x$ and $\varphi(V)\hat x$ based at the end of $V$.

Otherwise, $V$ is a proto-spike bounded by an arc $\alpha\in \arc_\circ$.  
As in the proof of Lemma \ref{lem:invisiblesubarc} we know that $\alpha$ must be short; let $h_+$ be the geodesic asymptotic to $g_-$ obtained from $g_+$ by applying a small rotation at $\alpha \cap g_+$.
Then $\varphi(V)$ is close in $\PSL_2 \RR$ to the parabolic isometry 
\[\varphi'(V) = T_{g_-}^{\ac (k_V)+f_{X,\ac}(V)} \circ T_{h_+}^{-\ac(k_V)-f_{X,\ac}(V)}. \] 
Indeed, using property \eqref{eqn: PSL2R metric basepoint}, continuity of the metric $d_{\hat x}$, and Lemma \ref{lem:uniform_spikes} bounding the size of $\ell_{\ac}$, there is a constant $B' = B'(s,\|\ac\|)$ such that $\|\varphi (V)\|$ is also bounded above by $B'$ times the length of a horocycle between $\hat x$ and $\varphi'(V)\hat x$ based at the fixed point of $\varphi'(V)$.
We have thus reduced the case of proto-spikes to a computation for spikes.

To bound the length of the horocycle between $\hat x$ and $\varphi(V)\hat x$ based at $V$, we compute in the upper half plane.
Let $t = \ac(k_V) +f_{X,\ac}(V)$, which by by Lemma \ref{lem:uniform_spikes} is $ \ac(k_V) +O(\|\ac\|)$.
Normalizing so that $\hat x = i$ and $g_\pm$ are geodesics from $a_+>a_- >0$ to $\infty$, 
we have
\begin{align*}
    \varphi(V) &= 
\begin{pmatrix} 1 & a_- \\
& 1
\end{pmatrix}
\begin{pmatrix} e^{t/2} & \\
& e^{-t/2}
\end{pmatrix}
\begin{pmatrix} 1 & a_+-a_- \\ & 1
\end{pmatrix}
\begin{pmatrix} e^{-t/2} & \\
& e^{t/2}
\end{pmatrix}
\begin{pmatrix} 
1& -a_+ \\
& 1
\end{pmatrix} \\
\\
&= \begin{pmatrix} 1 & (a_+-a_-) (e^{\ac(k_V) + O(\|\ac\|)} - 1) \\ & 1\end{pmatrix}, 
\end{align*}
and so the length of the horocyclic arc based at $\infty$ through $\hat x $ and $\varphi(V) \hat x$ is $(a_+-a_-)(e^{\ac(k_V) + O(\|\ac\|)} - 1)$.

We now estimate the size of $a_+-a_-$, which is the length of a horocyclic segment contained in $V$ through $\hat x$ and based at the tip of $V$.
There is a constant $B''$ depending on $\delta$ and the angles that $k_b$ makes with $\lambda$ such that $\ell(k_b\cap V)$ is bounded above the by the length of any horocycle in $V$ meeting $k_b$ times $B''$ and below by $1/B''$ times the same horocycle length.
Using our coordinates in the upper half plane, let $y = \sup \{\Im (k_b\cap V)\}$. 
Then $y\le e^{\ell(k_V)}$, so that $a_+-a_-\le  e^{\ell(k_V)}B''\ell(k_b\cap V)$.

Let $B = B'B''$; recall that $B'$ depends on $s$ and $\|\ac\|$, while $B''$ depends both on the angle that $k_b$ makes with $\lambda$ and on $\delta$.  
We are free to replace $k_b$ by a different geodesic transversal (as long as we only do so once, rather than change it for every proto-spike $V$).
Suppose $k_b'$ meets some leaf $g$ of $\lambda$ running through $b$ at a right angle in a point $p\in g$.
We claim that $k_b'$ meets every leaf $h$ of $\lambda$ running through $b$ and makes angle in $[\pi/2-\theta, \pi/2+\theta]$ for some $\theta\in (0,\pi/2)$.  
Indeed, by our choice of $\delta \le \delta_{\ref{lem:equittswork}}$, 
every $h$ running through $b$ meets the $1/2$-ball centered at $p$ (since the width of $\cE_\delta(\lambda)$ is at most $1/2$).
A short computation shows that $k_b'$ meets $h$ and makes angle with $h$ in an interval bounded away from $0$ and $\pi$.
This shows that $B$ depends only on $\delta$, $s$, and $\|\ac\|$.\footnote{In the next Lemma \ref{lem:basic_estimate_product}, we will restrict to $\|\ac\|\le \xi_{\ref{thm:shapeshift_distance_small}}$, which further removes the dependence of $B'$ on $\|\ac\|$.}
Combining everything proves the lemma.
\end{proof}

We can now define the shape-shifting transformation\label{ind:shapeshifttrans}
\[\varphi(k_b) = \left ( \prod _{V\in \mathcal V} \varphi (V)\right)\circ  T_{g}^{\ac(k_b)}, \]
where the product in $\PSL_2 \RR$ is taken according to the linear order $\le $ on spikes $\mathcal V$ in the universal cover, and $g$ is the last geodesic of $\lambda$ that $k_b$ crosses.
Convergence was proved in \cite[Lemma 14.10]{shshI} as long as $\|\ac\|< s/9|\chi(S)|$ using the ideas of \cite{Bon_SPB,Th_stretch}; see also \cite{PapadopTheret:Teich_Thurston}.

The following is the main technical estimate in this section and is where we obtain the threshold $\xi_{\ref{thm:shapeshift_distance_small}}$ appearing in Theorem \ref{thm:shapeshift_distance_small}.

\begin{lemma}\label{lem:basic_estimate_product}
There is a $\xi_{\ref{thm:shapeshift_distance_small}} >0$ such that if $(X,\lambda)$ is $s$-thick and $\|\ac\|\le \xi_{\ref{thm:shapeshift_distance_small}}$, then for any $a \in (0,1)$,
    \[\|\varphi(k_b)\| \le O(e^{d(\hat x, b)}\|\ac\|^a) \]
    where the implicit constant depends on $s$, $\delta$, and $a$.
\end{lemma}

\begin{proof}
    First we note that $\varphi(k_b)$ depends only on the geodesics that $k_b$ crosses (not the arc itself). 
    Our basepoint $x$ is boundedly far  from any leaf of $\lambda$  (Lemma \ref{lem:radius_bounded}) and the width of the $\delta$--equilateral neighborhood of $\lambda$ is $O(\delta)$ (Lemma \ref{lem:equittwidth}) with both bounds depending only on $s$.
    So, we can assume that $\ell(k_b) \le C_1$,  for some $C_1$ depending only on the thickness $s$ of $X$. 

    To each $V\in \mathcal V$, there is a trainpath of length $r(V)\in \mathbb N$\label{ind:depth} starting at a (proto-)spike $V_0$ of $\tau$ and ending at 
    $V\cap k_b$ 
    (see \cite[\S14.1]{shshI}).
    From \cite[Lemma 14.5]{shshI}, since $X$ is $s$-thick, we have
\begin{equation}\label{eqn:length_spike_upper_bound}
         \ell(k_b\cap V)\le
     \delta\exp\left({-\frac{s}{9|\chi|}r(V)}\right).
     \end{equation}
    
    We have similar length bounds from below as follows.
    The total geometric length of the train path of combinatorial length $r(V)$ is at most $Er(V)$, where $E $ is the total length of $\tau$ and is bounded by a topological constant times $D_\delta\approx \log(1/\delta)$.
    This gives
    \begin{equation}\label{eqn:length_spike_lower_bound}
          \ell(k_b\cap V)\ge \delta e^{-E r(V)} \ge c_2 e^{-r(V)},
     \end{equation}
     for a constant $c_2 $ depending on $s$, $\delta$ and the topology of $S$.
    
    For any positive $\varepsilon$, there is a finite collection $\mathcal V_{ \varepsilon}$\label{ind:protospike_eps} of proto-spikes $V$ satisfying 
    \[\varepsilon \le e^{\ell(k_V)}\ell(k_b\cap V).\]
    Using the triangle inequality for our left invariant metric $d_{\hat x}$, we have 
    \begin{equation}\label{eqn:spikes_triangle}
        \|\varphi(k_b)\| \le 
        \sum_{V\in \mathcal V \setminus \mathcal V_\varepsilon}\|\varphi(V)\| + 
        \sum_{V\in \mathcal V_\varepsilon} \|\varphi (V)\| + 
        \|T_g^{\ac(k_b)}\|.
    \end{equation} 
Our goal is to bound each term in this sum. We observe first that \eqref{eqn:transvection_estimate} gives a bound of 
\begin{equation}\label{eqn:lastterm}
\|T_g^{\ac(k_b)}\|\le e^{C_1}|\ac(k_b)| \le e^{C_1} \|\ac\|
\end{equation}
on the size of the last term.

\begin{claim}\label{claim:spikes_thin}
        If $\|\ac\|<D/2 = \frac{s}{18 |\chi(S)|}$, then
        \[\sum_{V\in \mathcal V\setminus \mathcal V_{\varepsilon}} \|\varphi(V)\|\le C_2 \varepsilon   \]
        where $C_2 = C_2 (s,\delta, S)$.
\end{claim}
\begin{proof}[Proof of the claim]
For a given proto-spike $V$ downstairs, the intersection with $k_b$ is an at most countable collection of arcs with exponentially decreasing length.
Upstairs, this gives us at most countably many lifts $V_1$, $V_2$,... of $V$ listed so that $\ell(k_b\cap V_i)>\ell(k_b\cap V_{i+1})$. 
There is some $I$ such that for $i \ge I$, then $V_i \not \in \mathcal V_\varepsilon$.  Using the lower length bounds \eqref{eqn:length_spike_lower_bound} we can deduce that 
\[r(V_i) \ge \log(c_2/\varepsilon) +m(V_i)\]
       for some $m(V_i) \ge 0$ that goes to infinity with $i$.
        
        The basic estimate (Lemma \ref{lem:basic_estimate_protospikes}) gives 
        \begin{align*}
            \|\varphi(V_i)\| &\le  Be^{\ell(k_{V_i})}\ell(k_b\cap V_i) e^{|\ac(k_{V_i})| +O(\|\ac\|)}\\
            &\le Be^{C_1} \exp\left( -Dr(V_i)+r(V_i)\|\ac\|+O(\|\ac\|)\right)\\
            & \le \varepsilon C_2' e^{-m(V_i)(D-\|\ac\|)}
        \end{align*}
        for $V_i\in \mathcal V\setminus \mathcal V_\varepsilon$, where $C_2' = C_2' (s, \delta, S)$. 

        The number of spikes of $\tau$ is bounded above in terms of the topology of $S$,
        so we have 
        \begin{align*}
            \sum_{V \in \mathcal V \setminus \mathcal V_\varepsilon}\|\varphi(V)\| \le \varepsilon C_2'' \sum_{m= 0}^\infty e^{-m(D-\|\ac\|)} \le \varepsilon C_2'' \frac{1}{1-e^{(\|\ac\| - D)}}, 
        \end{align*}
        where $C_2'' = C_2''(s,\delta, S)$.
        For $\|\ac\|< D/2$, the ratio $1/(1-e^{\|\ac\|-D})$ is bounded from above by a constant depending only on $s$ and $S$.  
        Taking $C_2$ accordingly proves the claim.
        \end{proof}

\begin{claim}
We have bounds \[|\mathcal V_{\varepsilon}| \le C_3 \log (\delta/\varepsilon), \]
and for $V\in \mathcal V_\varepsilon$,  \[r(V) \le C_3 \log (\delta/\varepsilon)\]
if $\varepsilon <\delta/2$, where $C_3=C_3(s,\delta,S)$.
\end{claim}
\begin{proof}[Proof of the claim]  
Consider a spike of $S \setminus \tau$ and again enumerate the intersections of the corresponding proto-spike with $k_b$ as $V_1, V_2, ...$ in order so that the length decreases.

By definition that $V\in \mathcal V_{\varepsilon}$, we have  $\varepsilon \le e^{2\ell(k_V)}\ell(k_b\cap V)$.  Using the upper length bound \eqref{eqn:length_spike_upper_bound} and taking $D = \frac{s}{9|\chi(S)|}$, we find that
     \[r(V_i) \le D\inverse(\log(\delta/\varepsilon)+ C_1).\]
     If $\varepsilon\le \delta/2$, then we can write
     \[r(V_i) \le D\inverse(1+C_1)\log(\delta/\varepsilon).\]

    There are at most  $6|\chi(S)|$ proto-spikes of $\tau$ and by our labeling convention, $i\le r(V_i)$.
  Taking $C_3=6|\chi(S)|D\inverse(1+C_1)$ proves the claim.
     \end{proof}

    Recall from \cite[Lemma 14.3]{shshI} or \cite[Lemma 6]{Bon_SPB} that \[|\ac(k_V)|\le r(V) \|\ac\|.\]

  \begin{claim}
    There is a
    $\xi'= \xi'(s,\delta)$ such that 
      for $V\in \mathcal V_{\|\ac\|}$, we have \[\|\varphi(V)\|\le C_4\log(1/\|\ac\|)\|\ac\|,\]
      where $C_4= C_4(s,\delta, S)$ as long as $\|\ac\|\le\xi'$.
  \end{claim}
  \begin{proof}[Proof of the claim]
     Applying Lemma \ref{lem:basic_estimate_protospikes}, for each (proto-)spike $V$, we have 
     \[\|\varphi(V)\| \le Be^{C_1}\delta \left( e^{(r(V)+1)O(\|\ac\|)}-1\right).\]
     Now for $V \in \mathcal V_{\|\ac\|}$, we have that $r(V)$ is bounded by $C_3 \log(\delta/\|\ac\|)$.
     Since $e^x-1 = O(|x|)$ for $x$ close to $0$, if  
     \[C_3\log(\delta/\|\ac\|) \|\ac\| \ll 1\] 
     is small enough (defining $\xi'$), we  get 
     \[\|\varphi(V)\|\le C_4 \log(1/\|\ac\|)\|\ac\|\]
    where $C_4$ depends on $s$, $\delta$, and $S$.  
    \end{proof}

The previous two claims give that, for $\varepsilon = \|\ac\| \le \zeta',$
\[
\sum_{V\in \mathcal V_\varepsilon}\|\varphi(V)\| 
\le
C_5 \log(1/\|\ac\|)^2\|\ac\|
\]
where $C_5 = C_5(s, \delta,  S)$.
Putting this together with Claim \ref{claim:spikes_thin} (for $\varepsilon = \|\ac\|$) and equation \eqref{eqn:lastterm} yields
\begin{align*}
\|\varphi(k_b) \| \le C_2 \|\ac\| + C_5 \log(1/\|\ac\|)^2 \|\ac\| + e^{C_1}\|\ac\|.
\end{align*}
Since $\log(1/\|\ac\|)^2 \|\ac\|$ is $a$-H\"older for any $a \in (0,1)$, combining the inequalities proves the lemma, as long as $\|\ac\|\le\xi_{\ref{thm:shapeshift_distance_small}}= \min\{D/2, \xi', \delta/2\}$.
\end{proof}

\subsection{Shapeshifting along the spine}
We have described and bounded the shape-shifting cocycle $\varphi(k_b)$ for transversals to branches of $\tau$. 
Let $k$ now denote a \emph{spinal path}. That is, $k$ is non-backtracking and only meets the branches of $\taua$ that correspond to the completion $\arc_*$ of the visible arc system $\arc_\bullet$.

For $\alpha \in \arc_*$, define a shaping transformation $A(\alpha)$ by the formula in \eqref{eqn:shaping}. \label{ind:shapingspine}
Whereas for proto-spikes this transformation was best thought of as approximating a parabolic, here we think of it as sliding the complementary components of $X \setminus \lambda$ along the shortest representative of $\alpha$ (even if it does not appear as a leaf of $\Ol(X)$) and the boundary geodesics adjacent to it. Compare Section 14.3 of \cite{shshI}.
In Section 14.4 of \cite{shshI}, we defined shape-shifting deformations along spinal paths \[\varphi(k)= A_1\circ A_2 \circ ... \circ A_{n(k)}. \]   
Each $A_i$ is of the form $A(V)$ for some spike $V$ of $\tau$, or $A(\alpha)$ for some $\alpha \in \arc_*$, or $T_g^{\ac(\alpha)}$, where $g$ is a boundary geodesic of the subsurface of $X\setminus( \lambda \cup \arc_\circ)$ corresponding to the component of $X\setminus \tau$ containing $k$.
Moreover, $n(k)$ is at most $3$ times the number of arcs that $k$ crosses.
Since we are only interested in the size of $\varphi(k)$, we will not go into further details about the construction.

\begin{lemma}\label{lem:estimate_spine}
    Given $s>0$ and $\|\ac\|\le \xi_{\ref{thm:shapeshift_distance_small}}$, the shape-shifting cocycle $\varphi(k)$ corresponding to a spinal path $k$ satisfies \[\|\varphi(k)\| \le  e^{\diam(\hat x \cup k)}|k|O(\|\ac\|)\]
    where $|k|$ is the number of arcs crossed and the implicit constant depends on $s$, $\delta$, and the topology of $S$.
\end{lemma}

\begin{proof}
Applying Lemma \ref{lem:uniform_spikes} with $\|\ac\|\le \xi_{\ref{thm:shapeshift_distance_small}}$,
we get that the functions $f_{X,\ac}$ and $\ell_\ac$ are all $O(\|\ac\|)$.
The Lemma then follows from the definition of $\varphi(k)$ using the triangle inequality, \eqref{eqn:transvection_estimate}, and Lemma \ref{lem:basic_estimate_protospikes}.
\end{proof}

\subsection{Uniform bounds on shape-shifting}

We have now defined and estimated the size of shape-shifting cocycles on both spinal paths and along transversals to branches of $\tau$.
For an arbitrary non-back-tracking path $k$ transverse to $\taua$, we have
\[k = k_1\cdot  ... \cdot k_{n(k)},\]
where $k_i$ is either a maximal spinal path or $k_i = k_b$ for some branch $b$ of $\tau$.
The shape-shifting cocycle $\varphi(k)$ is defined 
\[\varphi(k) = \varphi(k_1)\circ ... \circ \varphi(k_{n(k)}).\]

Now we can establish the uniform bound on shape-shifting needed to prove Theorem \ref{thm:shapeshift_distance_small}.

\begin{proposition}\label{prop:shapeshift}
With notation as in \S\ref{subsec:reference_tts}, 
given $\gamma \in \Gamma$ and $a \in (0,1)$, if
$\ac \in W(\taua)$
satisfies $\|\ac\| \le \xi_{\ref{thm:shapeshift_distance_small}}$,
then
\[d_{\hat x}(\rho(\gamma), \rho_{\ac}(\gamma)) = O(\|\ac\|^{a})\]
where the implicit constant in this statement depends only on $s$, $\delta$, $a$,  and $\|\rho(\gamma)\|$.
\end{proposition}

\begin{proof}
Let $\gamma\in \pi_1(X,x)$, 
and consider a (geodesic) arc $k_\gamma$ joining $x$ to $\rho(\gamma). x$.  
We may perturb the data so that $k_\gamma$ is transverse to $\taua$ without backtracking.

First we estimate the size of $\varphi(k_{\gamma})$, as defined above.
Using Lemmas \ref{lem:basic_estimate_product} and \ref{lem:estimate_spine}, we get 
\[\|\varphi(k_\gamma)\| \le  3|k_\gamma| e^{\diam (\hat x \cup k_\gamma)} O(\|\ac\|^a) \]
with implicit constant depending on $s$, $\delta$, $a$, and $S$ as long as $\|\ac\|\le \xi_{\ref{thm:shapeshift_distance_small}}$.

Generally, if $\|B\|<C$ and $\|A \| \le \varepsilon\ll 1/C$, then there is a $C'$ such that the size of the commutator $[A,B] = A\inverse B\inverse A B$ is bounded by $C'\varepsilon$ \cite[Theorem 4.1.6]{Th_book}.
By definition, $\rho_\ac(\gamma) = \varphi(k_\gamma)\rho(\gamma)$, so 
\begin{align*}
    d_{\hat x}(\rho(\gamma) ,  \varphi(k_\gamma)\rho(\gamma)) & =  d_{\hat x}\left(\rho(\gamma) ,  \rho(\gamma)\varphi(k_\gamma) [\varphi(k_\gamma), \rho(\gamma)]\right)\\
    & = d_{\hat x}\left(id,  \varphi(k_\gamma) [\varphi(k_\gamma), \rho(\gamma)]\right)\\
    & \le \|\varphi(k_\gamma)\|(1+C') \\
    &\le |k_\gamma| e^{\diam(\hat x \cup k_\gamma)}O(\|\ac\|^a)
\end{align*}
Notice that $|k_\gamma|e^{\diam (\hat x\cup k_\gamma)}$ bounds $\|\rho(\gamma)\|$ from above. 
This gives \emph{uniform} bounds on $d_{\hat x}(\rho(\gamma), \rho_{\ac}(\gamma))$ with multiplicative factor depending only on $s,\delta, a$ and $d_{\hat x}(id, \rho(\gamma))$, completing the proof of the Proposition. 
\end{proof}

\section{Proof of inverse continuity}\label{sec:inverse continuity}
Our main theorem in this section establishes one of the implications stated in Theorem \ref{maintheorem:O and Oinv continuous} and gives a quantitative comparison between the product structures of $\PT_g = \T_g \times \ML_g$ and $\QT_g = \MF_g \times \MF_g \setminus \Delta$.

Recall from \S\ref{subsec:metrics on PSL_2R} the quantity $d_\Gamma(X,X')$ measures the largest distance in $\PSL_2\RR$ between the images of generators $\Gamma$ for $\pi_1(S)$ under the holonomy representations of $X$ and $X'$.

\begin{theorem}\label{thm:cont_inverse}
For every $L >0$, $\epsilon>0$, and $q = \cO(X, \lambda)$, there is a constant $\zeta$ and an open neighborhood $U$ of $q$ in $\QT$ and such that if
\[q'=\cO(X',\lambda')\in U \text{ and }
d_X^H(\lambda, \lambda')<\zeta,\]
then $d_\Gamma(X,X')<\epsilon$ for any finite generating set $\Gamma$ of $\pi_1(S)$ with $L_X(\Gamma) \le L$.
\end{theorem}

\begin{remark}
Given $s>0$, there is an $L$ such that any $s$-thick $X$ has a generating set $\Gamma$ with $L_X(\Gamma) \le L$.
\end{remark}

The proof of this Theorem relies on the local product structure of $\QT_g$. In \cite{shshI}, we proved that $\cO\inverse$ is very well behaved (in fact, is piecewise real-analytic) on leaves of the unstable foliation.
We prove in Theorem \ref{thm:cont_inverse_stable} below that $\cO^{-1}$ is H{\"o}lder on leaves of the stable foliation with respect to the Hausdorff metric on the supports of the horizontal laminations.

\begin{theorem}\label{thm:cont_inverse_stable}
For any $L >0$, any $a \in (0,1)$, any $A\ge 1$, and any $s$-thick $q \in \QT_g$ with $1/A\le\Area(q)\le A$, there is an open neighborhood $U_{\ref{thm:cont_inverse_stable}}^{ss}$ of $q$ in $\Fol^{ss}(q)$ and a $\zeta_{\ref{thm:cont_inverse_stable}} > 0$ such that for any 
$q' = \cO(X', \lambda') \in U_{\ref{thm:cont_inverse_stable}}^{ss}$ with
$d_X^H(\lambda, \lambda')<\zeta_{\ref{thm:cont_inverse_stable}},$
we have 
\[d_\Gamma(X,X') = 
O \left( d_X^H(\lambda, \lambda')^a \right)\] for any generating set $\Gamma$ with $L_X(\Gamma) \le L$, where
the implicit constant depends only on $s$, $a$, $A$, and $L$.
\end{theorem}

The main obstacle in the proof of Theorem \ref{thm:cont_inverse} is that the product structures of $\mathcal{PT}_g = \T_g \times \ML$ and $\mathcal{QT}_g = \ML \times \ML \setminus \Delta$ do not match.
That is, deforming in the stable leaf of $q$ changes both the horizontal lamination and the corresponding hyperbolic surface. 
Along these lines, one can interpret Theorem \ref{thm:cont_inverse_stable} as an estimate of the difference between these two product structures, assuming Hausdorff closeness of the horizontal laminations; 
Theorem \ref{thm:cont_inverse} then follows by constructing a neighborhood on which all of our constructions and estimates work simultaneously.

\subsection{Hausdorff convergence of horizontal laminations}
Suppose that $q_n = \cO(X_n, \lambda_n)$ is a sequence of quadratic differentials that converges to $q=\cO(X, \lambda)$.  
The condition that the supports of $\lambda_n$ converge in the Hausdorff topology to the support of $\lambda$ constrains the direction in which $q_n \to q$.

\begin{proposition}\label{prop:stable_hausdorff_stratum}
For any $s >0$ and any $A\ge 1$ there is a $\zeta_{\ref{prop:stable_hausdorff_stratum}}>0$ such that every $s$-thick $q= \cO(X,\lambda) \in \QT_g$ with $1/A\le\Area(q)\le A$ has a neighborhood $U^{ss}_{\ref{prop:stable_hausdorff_stratum}}$ of $q$ in $\Fol^{ss}(q)$ such that whenever 
\[q'= \cO(X',\lambda') \in U^{ss}_{\ref{prop:stable_hausdorff_stratum}} 
\text{ and }
d_X^H(\lambda,\lambda')<\zeta_{\ref{prop:stable_hausdorff_stratum}},\]
then $q$ and $q'$ are in the same stratum component $\cQ$.
\end{proposition}
\begin{proof}
Using Corollary \ref{cor:thinparts}, there is an $s'$ depending only on $s$ and $A$ such that $(X,\lambda)$ is $s'$-thick.
Taking any $\delta\le \delta_{\ref{lem:equittswork}}(s')$, we can build an equilateral train track $\tau = \tau(X,\lambda,\delta)$.
Let $\arc_\bullet = \arc_\bullet(X,\lambda,\delta)$ be the visible arc system and $\taua$ the standard smoothing of $\tau\cup \arc_\bullet$.

The dual $\mathsf T$ of $\taua$ is a simply horizontally convex cellulation of $\cO(X,\lambda)$ by veering saddle connections (Proposition \ref{prop:dual cell veering} and Lemma \ref{lem:dual track is horizontally convex}). 
Note that $\tau = \tau(q,\mathsf T)$ by Lemma \ref{lem:dualtt_equitt}.
Using Lemma \ref{lem:cellspersist}, we can find a neighborhood $B_{\mathsf T}(q)$ of $q$ in $\QT_g$ such that if  $q'\in B_{\mathsf T}(q)$, then the following hold:
\begin{itemize}
    \item There is a horizontally convex refinement $\mathsf T'$ of $\mathsf T$ by saddle connections on $q'$.
    \item The dual track $\tau'=\tau(q',\mathsf T')$ fully carries the imaginary foliation of $q'$. 
    \item After removing certain branches of $\tau'$ and collapsing certain others, the resulting train track is isotopic to an extension of $\tau$.
\end{itemize}

We now explain why all $q'$ close enough to $q$ in $\Fol^{ss}(q)$ but not in the same stratum component have only short {\em vertical} saddle connections coming from breaking up higher order zeros of $q$. Compare with our discussion in Section \ref{subsec:nbhd_strata}.
Let $\eta$ denote the measured lamination equivalent to the measured foliation $\Ol(X)$.
Theorem 10.15 of \cite{shshI} gives a PIL homeomorphism
\[\operatorname{I}_{\eta}: \Fol^{ss}(\cO(X,\lambda))\to \SH^+(\eta) \]
parameterizing the \emph{stable} leaf through $q$ (rather than the unstable leaf), essentially by period coordinates with bases adapted to vertical saddle connections. 

By the structure theory for shear--shape space \cite[Proposition 8.5]{shshI}, we see that $\operatorname{I}_{\eta}(\mathcal Q\cap \Fol^{ss}(\cO(X,\lambda)))$ is a $\cH^+(\eta)$--bundle over a union of cells of $\mathscr B(S\setminus \eta)$; these cells encode the multiplicities of the zeros and the vertical saddle connections between them.
Furthermore, if  $\arcc$ is the filling arc system of  $S\setminus \eta$ describing the graph of vertical saddle connections for $q$, then a neighborhood of $q$ in $\mathcal Q \cap \Fol^{ss}(q)$
is mapped homeomorphically to a neighborhood of $\operatorname I_\eta (q)$ in $\SH^+(\eta;\arcc)$ via $\operatorname{I}_\eta$.  

In the same vein, we can find a  neighborhood $U$ of $\operatorname I_\eta(q)$ in $\SH^+(\eta)$ whose projection to $\mathcal B(S\setminus \eta)$ meets only the cells that correspond to those filling arc systems $\arc$ of $S\setminus \eta$ containing $\arcc$.  
Dually, all flat surfaces 
in $I_\eta^{-1}(U)$ have sets of zeros corresponding to higher order zeros of $q$ and which are joined by trees of short vertical saddle connections.
Set $U^{ss}_{\ref{prop:stable_hausdorff_stratum}} = B_{\mathsf T}(q)\cap I_\eta\inverse(U)$.
\medskip

Invoking Proposition \ref{prop:ttdefs_comp}, 
so long as 
$d_X^H(\lambda,\lambda')<\delta/w_{\ref{prop:ttdefs_comp}}(s') =: \zeta_{\ref{prop:stable_hausdorff_stratum}}$
then $\lambda'$ is carried by $\tau(X,\lambda,\delta)$.
Towards contradiction, suppose that
$q'=\cO(X',\lambda')\in U^{ss}_{\ref{prop:stable_hausdorff_stratum}}$ is not in $\cQ$
but that 
$d_X^H(\lambda,\lambda')<\zeta_{\ref{prop:stable_hausdorff_stratum}}$.
Since $q'$ is not in $\cQ$,
there is a collection of zeros of $q'$ corresponding to breaking up one zero of $q$.
The horizontally convex cellulation $\mathsf{T}'$ of $q'$ must contain short saddle connections between these zeros (see the proof of Lemma \ref{lem:cellspersist}), but by construction of $U^{ss}_{\ref{prop:stable_hausdorff_stratum}}$ 
all of these saddle connections must be vertical.
Thus, the subspace of $W(\tau')$ corresponding to $W(\tau)$ is a {\em proper} subspace, obtained by setting the weights on some number of the branches of $\tau'$ equal to $0$.
But now since $\tau(q',\mathsf T')$ carries $\lambda'$ {\em fully,} this means $\tau$ cannot carry $\lambda'$.
This is a contradiction, so we conclude that $q'$ is contained in $\mathcal Q$, completing the proof of the proposition.
\end{proof}

We obtain also the following corollary of the proof of Proposition \ref{prop:stable_hausdorff_stratum}, which we record for later use.

\begin{corollary}\label{cor:lamination_close_star_nbh_stratum}
For any $s$-thick $q\in \QT$ with $1/A \le \Area(q)\le A$, 
there is an open neighborhood $U_{\ref{cor:lamination_close_star_nbh_stratum}}$ of $q$ in $\QT_g$ such that if $q' \in U_{\ref{cor:lamination_close_star_nbh_stratum}}$ and $d_X^H(\lambda,\ \lambda') <\zeta_{\ref{prop:stable_hausdorff_stratum}}$,
then $q' \in \cQ^*$.
\end{corollary}
\begin{proof}
The proof follows the same lines as that of the previous Proposition, except we derive a contradiction with Lemma \ref{lem:ttsaroundstrata} instead of Lemma \ref{lem:cellspersist}.
The relevant observation is just that if $q'$ is close enough to $q$ but not in $\mathcal Q^*$, then any refinement $\mathsf T'$ of $\mathsf T$ on $q'$ constructed in the proof of Lemma \ref{lem:cellspersist} has some short non-horizontal saddle connection as an edge. 
\end{proof}

\begin{remark}
We note that this argument actually proves something even stronger: for every $q'$ satisfying the hypotheses of Corollary \ref{cor:lamination_close_star_nbh_stratum} and every horizontal saddle connection of $\mathsf{T}$ on $q$, the corresponding saddle connection(s) of $\mathsf{T}'$ on $q'$ must also be horizontal.
This is because if $\lambda'$ is to be carried on $\tau$, then the dual train track $\tau(q',\mathsf{T}')$ must not have any new branches coming from smoothing the arcs dual to horizontal saddle connections of $\mathsf{T}$. Hence those saddle connections must remain horizontal in $\mathsf{T}'$.
\end{remark}

\subsection{Proofs of continuity theorems}
By definition, any differential $q' \in \Fol^{ss}(q)$ has the same vertical foliation $\Ol(X)$, but the shear-shape cocycles for $q$ and $q'$ with respect to their horizontal laminations live in different spaces.
By finding a suitable train track that carries the horizontal laminations of both $q$ and $q'$, these shear-shape cocycles give rise to identical weight systems.
Applying Theorem \ref{thm:contcell} and the shape-shifting estimates of Section \ref{sec:shapeshift_est} then allows us to control the distance between the corresponding hyperbolic surfaces.

\begin{proof}[Proof of Theorem \ref{thm:cont_inverse_stable}]
Fix $a \in (0,1).$
As in the proof of Proposition \ref{prop:stable_hausdorff_stratum}, we observe that the conditions that $q$ is $s$-thick and $1/A\le \Area(q)\le A$ imply that $\cO^{-1}(q)$ is $s'$-thick.
Fix a defining parameter $\delta$ smaller than both $\delta_{\ref{lem:equittswork}}(s')$ and $\delta_{\ref{lem:change base ties still meet}}(s')$.
Let $\tau:= \tau(X,\lambda,\delta)$ denote the corresponding equilateral train track,
$\arc_\bullet$ the corresponding visible arc system, and set $\taua$ to be a standard smoothing of $\tau \cup \arc_\bullet$.
By Proposition \ref{prop:dual cell veering} and Lemma \ref{lem:ttwts_from_cell}, we can realize $\sigl(X)$ as a real-valued weight system on $\taua$.

Set $\zeta_{\ref{thm:cont_inverse_stable}}$
to be the minimum of $\zeta_{\ref{prop:stable_hausdorff_stratum}}$ and the threshold $\zeta_{\ref{thm:contcell}}$ from Theorem \ref{thm:contcell}, taken for the pair $(X, \lambda) := \cO^{-1}(q)$ and H{\"o}lder exponent $\sqrt{a}$.
Observe that the first cutoff is uniform over all surfaces with bounded systole, while the second depends on $q$.

Proposition \ref{prop:stable_hausdorff_stratum} builds a small neighborhood $U^{ss}_{\ref{prop:stable_hausdorff_stratum}}$ of $q$ in $\Fol^{ss}(q)$
such that any $q' \in U^{ss}_{\ref{prop:stable_hausdorff_stratum}}$ with $d_X^H(\lambda, \lambda') < \zeta_{\ref{prop:stable_hausdorff_stratum}}$
is also in the same stratum component $\cQ$ as $q$.
Taking $U^{ss} \subset U^{ss}_{\ref{prop:stable_hausdorff_stratum}}$ smaller as necessary, we can also ensure that any such $q'$ also lies in the neighborhood $B^{\cQ}_{\mathsf{T}}(q)$ from Lemma \ref{lem:simplehorizconvex_smooth}, where $\mathsf{T}$ is the cellulation of $q$ dual to $\taua$. 
Proposition \ref{prop:ttcharts_stratum} then implies that the intersection of $U^{ss}$ with the ambient stratum containing $q$ can be parametrized by weight systems on $\taua$; in particular, $\sigma_{\lambda'}(X')$ is the real part of the complex weight system $[\hol(q')]_+$.
Since $q$ and $q'$ are in the same stable leaf, they have the same real parts of periods, so we can say that
\[\sigl(X) = \sigma_{\lambda'}(X')\]
as weight systems on $\taua$.
In summary, we have chosen a neighborhood $U^{ss}$ of $q$ in its stable leaf so that for every $q'=\cO(X', \lambda') \in U^{ss}$ with 
$d_X^H(\lambda, \lambda') < \zeta_{\ref{prop:stable_hausdorff_stratum}}$, 
we can make sense of the equality
$\sigl(X) = \sigma_{\lambda'}(X')$.

Since $\lambda'$ is $\zeta_{\ref{thm:contcell}}$ close to $\lambda$ on $X$, we may invoke 
Corollary \ref{cor:augstable} to deduce that $\taua':=\taua(X, \lambda', \delta)$ is slide-equivalent to a smoothing of $\tau \cup  \arc_\bullet^{\tau}(X, \lambda')$, which in particular contains $W(\taua)$ as a subspace.
Thus, $\sigl(X) = \sigma_{\lambda'}(X')$ and $\sigma_{\lambda'}(X)$ can all be represented as weight systems on $\taua'$.
Theorem \ref{thm:contcell} 
applied with H{\"o}lder exponent $\sqrt{a}$
then lets us conclude that 
\[\| \sigl(X) - \sigma_{\lambda'}(X) \|_{\taua'} =
\| \sigma_{\lambda'}(X') - \sigma_{\lambda'}(X) \|_{\taua'} 
= O_s(\zeta^{\sqrt{a}}).\]

Now so long as $\zeta$ is taken small enough, this is less than the threshold $\xi_{\ref{thm:shapeshift_distance_small}}$, again taken for H{\"o}lder exponent $\sqrt{a}$.
Applying Theorem \ref{thm:shapeshift_distance_small} to the combinatorial deformation 
$\ac := \sigma_{\lambda'}(X') - \sigma_{\lambda'}(X)$ in shear-shape coordinates for $\lambda'$, we obtain
\begin{equation}\label{eqn:inverseest}
d_\Gamma(X,X')
=
O\left(
\| \ac \|_{\taua'}^{\sqrt{a}}
\right)
=
O\left(\zeta^{a} \right),
\end{equation}
completing the proof of the Theorem.
\end{proof}

\begin{remark}
If we were to restrict to only considering those $q$ with no short horizontal saddle connections, then we could use Corollary \ref{cor:persist_notallarcs} (and a slight generalization of Corollary \ref{cor:augstable})
to ensure that $\cO(X, \lambda')$ actually lives in the same stratum component $\cQ$, yielding a uniform threshold (analogous to $\zeta_{\ref{thm:cont_inverse_stable}}$) for all $q$ with no short horizontal saddles, bounded systole, and bounded area.
One can also prove that there is a uniform threshold for all $q$ with bounded systole and area, but this requires a uniform threshold for Theorem \ref{thm:contcell} (as detailed in Remark \ref{rmk:cont_unif}) and is therefore omitted.
\end{remark}

We now assemble the proof of Theorem \ref{thm:cont_inverse}. Theorem \ref{thm:cont_inverse_stable} provides continuity in a neighborhood $U^{ss}$ around $q$ in its stable leaf, and the uniform estimate on shapeshifting from Theorem \ref{thm:shapeshift_distance_small} yields continuity in uniform size neighborhoods around $q' \in U^{ss}$ in their respective unstable leaves. See Figure \ref{fig:localprod invcont}.

\begin{figure}
    \centering
    \includegraphics[scale=.15]{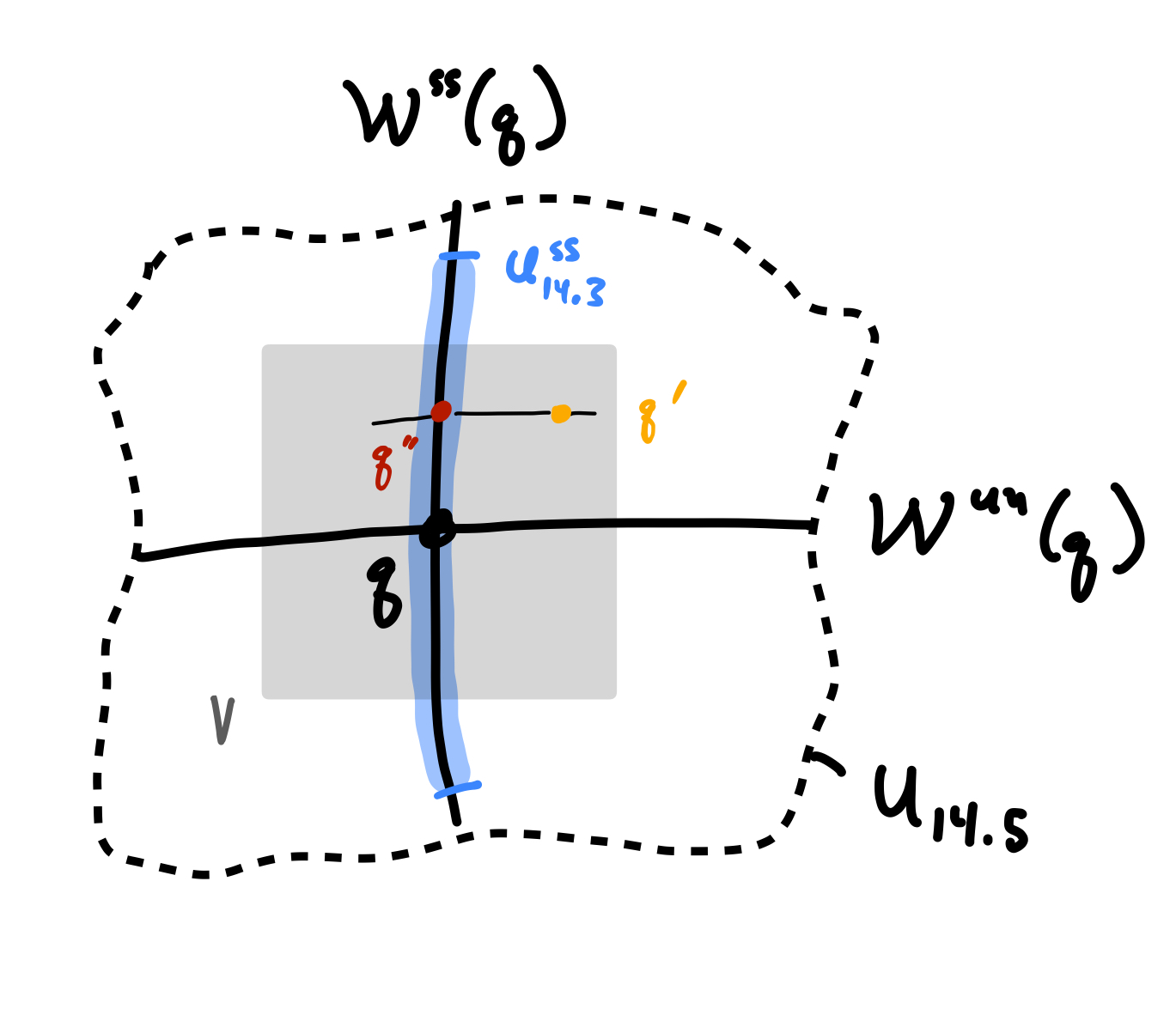}
    \caption{The local product structure around a point and the neighborhoods used in the proof of Theorem \ref{thm:cont_inverse}.}
    \label{fig:localprod invcont}
\end{figure}

\begin{proof}[Proof of Theorem \ref{thm:cont_inverse}]
Fix $a \in (0,1)$ and $\epsilon>0$, and pick any finite generating set $\Gamma$ with $L_X(\Gamma)\le L$.
Our base differential $q$ satisfies $\Area(q) \in (1/A,A)$ for suitable $A\ge 1$. We can apply all of the above results to differentials near $q$ with the same area bound.
Fix $\delta>0$ as in the proof of Theorem \ref{thm:cont_inverse_stable}
and consider the standard smoothing $\taua$ of the equilateral train track $\tau(X, \lambda, \delta)$ together with its visible arc system. Let $\mathsf{T}$ denote the (simply horizontally convex) cellulation of $q$ dual to $\taua$.

We begin with a small local product neighborhood $V$ of $q$ in $\QT_g$ such that 
\begin{enumerate}
    \item $V \cap \Fol^{ss}(q) \subset U^{ss}$, the neighborhood from Theorem \ref{thm:cont_inverse_stable}, and
    \item $V \subset U_{\ref{cor:lamination_close_star_nbh_stratum}}$.
\end{enumerate}
For any $q' \in V$, we can use the local product structure to find a unique $q'' = O(Z, \lambda')$ in $\Fol^{ss}(q)$.
By point (1) and Theorem \ref{thm:cont_inverse_stable}, we know that $d_\Gamma(X, Z) = O_s \left(d_X^H(\lambda, \lambda')^a \right)$; 
it therefore suffices to bound $d_\Gamma(Z, X')$ and take the cutoff $\zeta$ smaller than (a constant times) $\epsilon^{1/a}$.

As in the proof of Theorem \ref{thm:cont_inverse_stable}, 
point (1) and our threshold on $d_X^H(\lambda, \lambda')$ together imply that we can make sense of the equality $\sigl(X) = \sigma_{\lambda'}(Z)$ as weight systems on $\taua$.
On the other hand, point (2) and Corollary \ref{cor:lamination_close_star_nbh_stratum} ensure that $q' \in \cQ^*$, so Proposition \ref{prop:ttcharts_Q*} implies that 
$\sigma_{\lambda'}(X')$ can be represented as a weight system on a smoothing of $\tau \cup \arc'$, where $\arc'$ is some extension of $\arc$.
In particular, $\sigma_{\lambda'}(Z)$ and $\sigma_{\lambda'}(X')$ are both represented by weight systems on this smoothing.

We now shrink $V$ as necessary to ensure that the real parts of the periods of $q$ and $q'$ are $\xi_{\ref{thm:shapeshift_distance_small}}$-close; in particular, the real parts of the periods of $q''$ and $q'$ are also $\xi_{\ref{thm:shapeshift_distance_small}}$-close.
Since the difference in real parts of the periods is exactly the cocycle $\sigma_{\lambda'}(Z) - \sigma_{\lambda'}(X')$, 
Theorem \ref{thm:shapeshift_distance_small} 
implies that 
\[d_\Gamma(Z, X') =
O\left(
\|\sigma_{\lambda'}(Z) - \sigma_{\lambda'}(X')\|_{\tau_{\arc'}} ^a \right),\]
which is in particular arbitrarily small so long as the difference in periods is small enough.
\end{proof}

\part{Transporting measures}\label{part:measures}

In this Part, we analyze the measure-theoretic consequences of our continuity theorems. 
In particular, we show in Section \ref{sec:continuous a.e.} that $\cO$ and $\cO^{-1}$ are continuous almost everywhere for a large class of measures, once their domains are properly restricted.
This implies our main results about transporting convergence of measures between $\PM_g$ and $\QM_g$.

\section{Continuity almost everywhere}\label{sec:continuous a.e.}

In the previous two parts, we proved that both $\cO$ and $\cO^{-1}$ are continuous when we also assume Hausdorff convergence of laminations/horizontal foliations.
The point of this section is to understand how constraining the domains of these maps can enforce Hausdorff convergence along certain sequences.
As a consequence, we can prove that $\cO$ and $\cO^{-1}$ are continuous almost everywhere for many measures, proving Theorems \ref{mainthm:measure convergence strata} and \ref{mainthm:push measures from P to Q}.
See also Theorem \ref{thm:pull back special convergence} for the strongest statement we can prove along these lines.
Along the way, we also make explicit the fact that $\cO$ is a Borel isomorphism by decomposing $\PT_g$ and $\QT_g$ into countable unions of Borel sets, on each of which $\cO$ restricts to a homeomorphism (Theorem \ref{thm:Boreliso}).

\subsection{Measure and Hausdorff convergence}\label{subsec:measure vs hausdorff}
The main obstacle for the convergence of measured laminations $\lambda_n \to \lambda \in \ML_g$ to imply Hausdorff convergence of their supports is the problem of ``disappearing measure.''
More precisely, suppose that for a fixed transversal $t$ (e.g., a tie of a train track), the measure $\lambda_n(t)$ is positive but tends to 0. Any Hausdorff limit of the supports of the  $\lambda_n$ will then contain a leaf passing through $t$, but since $\lambda(t)=0$ the support of $\lambda$ does not meet $t$.

In the example above, complementary components of the $\lambda_n$ ``merge'' in the limit as the support vanishes.
One might hope that if $\lambda_n \to \lambda$ in measure and they all have homeomorphic complements, then Hausdorff convergence of supports follows.
This is still not quite right; if the complements of $\lambda_n$ degenerate while the measure of a transversal goes to 0, then the two phenomena can cancel out, resulting in a limit of the same topological type that winds around the surface differently. 
See Remark \ref{rmk:not homeo over orbit} below.
\medskip

On the other hand, if measure convergence is witnessed by a train track $\tau$ that is \emph{snug} for $\lambda$ (i.e. has the same topological type), then measure convergence implies Hausdorff convergence of supports (Lemma \ref{lem:full_convergence}).
Our main insight in this section is in the proof of Theorem \ref{thm:Boreliso}, where we show that if $\lambda_n \to \lambda$ in measure, the topology of $\lambda_n$ and $\lambda$ is constrained, and the geometry of the complementary components of $\lambda_n$ is also constrained, then the supports of $\lambda_n$ converge to the support of $\lambda$ in the Hausdorff topology.

The following is a consequence of \cite[Lemma 22]{BonZhu:HD} after noting that close-by measured laminations in the interior of the positive cone of weights on a train track $\tau$ are also carried on some $\tau'\prec \tau$ that is a much better approximation to the supports of both.
\begin{lemma}[\cite{BonZhu:HD}]\label{lem:full_convergence}
Let $\lambda \in \ML_g$ and let $\tau$ be a train track snugly carrying $\lambda$.
Suppose that $\lambda_n \in \ML_g$ are 
carried on $\tau$ and that $\lambda_n \to \lambda$ in the measure topology.
Then the supports of $\lambda_n$ converge to the support of $\lambda$ in the Hausdorff topology.
\end{lemma}

This Lemma identifies a combinatorial constraint that ensures measure convergence implies Hausdorff convergence; we now show that the sets on which these constraints hold are Borel.
Let $\tau$ be a bi-recurrent train track.  
In the weight space $W(\tau)$ of $\tau$, we wish to identity those (strictly positive) weights that correspond to measured laminations with the same topological type as $\tau$.  
Denote this subset of the positive cone as $W^{>0}_{\text{snug}}(\tau)$.\label{ind:snugwts}  By abuse of notation, this also defines a subset of $\ML_g$. 

\begin{lemma}\label{lem:snug_weights}
The space $W^{>0}_{\text{snug}}(\tau)$ is cut out of $W(\tau)$ by countably many integer linear equations and inequalities, and is therefore Borel.
\end{lemma}

\begin{proof}[Proof sketch]
If $\tau$ is not connected then its weight space is just the product of the weight spaces of its components, so we may as well restrict to the case of connected $\tau$.

Suppose $\lambda\in W^{>0}(\tau)$ is not snugly carried by $\tau$.
There is a (splitting) sequence of birecurrent train tracks $\tau\succ\tau_1\succ\tau_2 \succ ... $
such that the intersection of all of the weight spaces is exactly the cone of measures supported on $\lambda$
\cite[Proposition 3.3.2]{PennerHarer}.
Since $\tau$ is not snug, some $\tau_k$ is obtained from $\tau_{k-1}$ by colliding branches.
\footnote{We recall that a {\bf collision} happens when one rips apart a large branch and there is no branch going between the resulting sides
\cite[\S 2.1]{PennerHarer}.}

We claim that the dimension of $W(\tau_k)$ is strictly less than that of $W(\tau_{k-1})$. If both are non-orientable, then this follows because the Euler characteristic, which computes the dimension of the weight space, drops; see \cite[proof of Theorem 15]{Bon_THDGL} or \cite[\S2.1]{PennerHarer}.
Otherwise, $\tau_k$ may be orientable when $\tau_{k-1}$ is non-orientable, in which case the dimension of $W(\tau_k)$ is the Euler characteristic of $\tau_k$ plus $1$.
When this happens, multiple branches of $\tau_{k-1}$ must be collided, because one cannot add a single branch to an orientable train track to build a non-orientable and birecurrent one (compare \cite[Remark 9.6]{shshI}).
Thus the Euler characteristic of $\tau_k$ is at least {\em two} less than that of $\tau_{k-1}$, and so the dimension of the weight spaces still must drop.

Any carrying map $\tau_k\prec \tau$ induces a $\ZZ$-linear injective mapping of weight spaces $W(\tau_k)\to W(\tau)$, so $\lambda$ is in a (proper) subspace defined by integer equations. Moreover, any other weight in that subspace is clearly not snugly carried by $\tau$.
There are only countably many such subspaces, proving the lemma.
\end{proof}

\subsection{Borel isomorphism}
Since there are only countably many isotopy classes of train tracks, the above discussion gives a decomposition of $\ML_g$ into a countable union of Borel sets $W^{>0}_{\text{snug}}(\tau)$ on which measure convergence $\lambda_n \to \lambda$ implies Hausdorff convergence of their supports.
We can use this decomposition to induce a Borel decomposition of $\PT_g$: given a birecurrent $\tau$, define\label{ind:PTtau}
\[\PT_g(\tau)  = \T_g \times W^{>0}_{\text{snug}}(\tau).\]
This is clearly is a Borel subset of $\PT_g = \T_g\times \ML_g$.

Let us prove a corresponding statement for the corresponding decomposition of $\QT_g$.

\begin{lemma}
The set $\QT_g(\tau):= \cO\left(\PT_g(\tau) \right)$\label{ind:QTtau} is a Borel subset of $\QT_g$.
\end{lemma}
\begin{proof}
From the definition of $\cO$ and the main theorem of \cite{shshI}, we have 
\[\QT_g(\tau) = \bigcup_{\lambda \in W^{>0}_{\text{snug}}(\tau)} \MF(\lambda)\times \{\lambda\}.\]
Via the Gardiner--Masur theorem, we can also view this as $\MF_g\times W_{\text{snug}}^{>0}(\tau) \cap (\MF_g\times \MF_g\setminus \Delta)$, where $\Delta$ is the collection of non-binding pairs, i.e., 
\[\Delta = \{(\lambda,\eta) : \text{ there is a }\gamma \in \MF_g  \text{ such that } i(\gamma, \lambda) + i(\gamma, \eta) = 0\}.\]
Thus to deduce that $\QT(\tau)$ is a Borel subset of $\QT_g$, we need only show that $\Delta$ is a Borel set.

Enumerate the simple closed curves $\gamma_1, \gamma_2, ...$, and define 
\[I (\lambda, \eta) = \inf_{j\ge 0} i(\gamma_j,\lambda) + i(\gamma_j, \eta).\]
Since the projective classes of counting measures on simple closed curves are dense in $\mathcal{PML}_g$ and $\Delta$ is invariant under  scale, we have that $\Delta = I\inverse(0)$.
Then $I(\lambda, \eta)$ is the pointwise non-increasing limit of \[I^k(\lambda,\eta) = \min_{j\le k} i(\gamma_j,\lambda) + i(\gamma_j, \eta)\]
as $k\to \infty$, and $I^k$ is continuous, being a minimum of continuous functions.
This proves that $I$ is Borel (in fact upper semi-continuous), thus $\Delta$ is Borel and the proof is complete.
\end{proof}

We have now decomposed both $\PT_g$ and $\QT_g$ as countable unions of Borel sets mapped bijectively to one another.
On each these sets, measure convergence and Hausdorff convergence of supports coincide, yielding continuity of both $\cO$ and its inverse.
We therefore arrive at the following:

\begin{theorem}\label{thm:Boreliso}
The bijection $\cO : \PT_g \to \QT_g$ is a Borel-Borel isomorphism.
\end{theorem}

This result refines \cite[Theorem C]{shshI}; see also \cite[Remark 2.2]{shshI}.

\begin{proof}
Theorems \ref{thm:contcell} and \ref{thm:cont_inverse} together imply
that the map $\cO$ restricts to a homeomorphism
\[\PT_g(\tau) \cong \QT_g(\tau)\]
for any birecurrent $\tau$.
Indeed, suppose $(X_n, \lambda_n) \to (X,\lambda)$ all in $\PT_g(\tau)$. 
Since $\lambda_n$ and $\lambda$ are all snugly carried by $\tau$, Lemma \ref{lem:full_convergence} implies that the supports of $\lambda_n$ converge in the Hausdorff topology to the support of $\lambda$.
By Theorem \ref{thm:contcell}, we have that $\cO(X_n,\lambda_n) \to \cO(X,\lambda)$.
Similarly, if $\cO(X_n,\lambda_n)\to \cO(X,\lambda)$ all in $\QT_g(\tau)$, then again the supports of $\lambda_n$ converge to the support of $\lambda$, so we can apply Theorem \ref{thm:cont_inverse} to see that $X_n\to X$, and hence that $(X_n,\lambda_n)\to (X,\lambda)$.

To prove the theorem, we observe that since there are only countably many isotopy classes of train tracks, any Borel set $E\subset \PT_g$ can be written as the (not necessarily disjoint) union of countably many Borel sets
\[E = \bigcup_{\tau} E \cap \PT_g(\tau),\]
and each is mapped to a Borel subset of $\QT_g$ via $\cO$ by the argument above. 
Thus $\cO(E)$ is the union of countably many Borel sets, hence Borel.

A symmetric argument proves that $\cO\inverse(E)$ is Borel if $E\subset \QT_g(\tau)$ is Borel, so we are done.
\end{proof}

\begin{remark}\label{rmk:not homeo over orbit}
We note that the map $\cO$ does {\em not} yield a homeomorphism 
\[\cup_{\phi\in \Mod(S)}\PT_g(\phi(\tau))\leftrightarrow \cup_{\phi\in \Mod(S)}\QT_g(\phi(\tau)).\]
Indeed, find a pseudo-Anosov mapping class $\phi$ whose corresponding Teichm{\"u}ller geodesic axis does not lie in the principal stratum; for simplicity, we assume that we can find one for which the corresponding projectively invariant attracting measured lamination $\lambda$ has only complementary triangles except for one $4$-gon.
Let $\tau$ be a snug train track carrying $\lambda$ such that $\phi(\tau) \prec \tau$, and let $\sigma$ be another train track with the same type as $\tau$ but for which the interiors of their weight spaces are disjoint. For example, one can construct such a $\sigma$ by adding a diagonal to the $4$-gon in the complement of $\tau$ and removing a branch between two triangles somewhere else.

There is a common maximal train track $\tau'$ that carries both $\tau$ and $\sigma$, and up to passing to a power of $\phi$, 
we have that $\phi(\tau') \prec \tau'$.
Then the positive part of the weight spaces of train tracks $\phi^n(\sigma)$ will converge to $\RR_{>0}\lambda$ in $W^+(\tau')$ but need not ever intersect $W^+(\tau)$.
In particular, for any $\mu\in W_{\text{snug}}^{>0}(\sigma)$, there is a sequence $c_n \to \infty$ such that 
\[\lambda_n = \frac{\phi^n(\mu)}{c_n} \to \lambda\in \ML_g,
\]
but the Hausdorff limit of the supports of $\lambda_n$ converge to a diagonal completion of $\lambda$, since the $\lambda_n$ are all fully carried on $\tau'$.
This failure of Hausdorff convergence of supports causes a discontinuity in $\cO$.
\end{remark}

\subsection{Transporting measure convergence}
The same arguments as above also allow us to prove Theorems \ref{mainthm:measure convergence strata} and \ref{mainthm:push measures from P to Q}. We begin with the latter, which we remind the reader states that $\cO_*$ is continuous along sequences $\mu_n \to \mu$ of measures on $\PoM_g$ so long as the generic $(X, \lambda)$ with respect to $\mu$ is maximal.

\begin{proof}[Proof of Theorem \ref{mainthm:push measures from P to Q}]
We observe that if the support of $\lambda$ is maximal then any $(X, \lambda) \in \PoM_g$ is a point of continuity for $\cO$.
This follows from the fact that since $\lambda$ is maximal, any train track carrying it is snug. Thus for any $\lambda_n \to \lambda$, Lemma \ref{lem:full_convergence} ensures that the supports of $\lambda_n$ must converge to the support of $\lambda$. Theorem \ref{thm:contcell} gives us that $\cO$ is continuous at $(X, \lambda)$.

The Theorem now follows from standard facts about pushing forward measure convergence. Specifically, we can invoke Theorem 5.1 of \cite{B:measures},
\footnote{Billingsley calls this the ``Main Theorem.''}
which states that if $f:X\to Y$ is a Borel measurable map, $\mu_n$ is a sequence of Borel probability measures on $X$ converging weak-$*$ to a Borel measure $\mu$, and $\mu$-almost every point of $X$ is a point of continuity for $f$, then $f_*\nu_n\to f_*\nu$ weak-$*$.
We proved in Theorem \ref{thm:Boreliso} that $\cO$ was Borel, and our hypothesis on $\mu$ plus the observation in the first paragraph show that $\cO$ is continuous at $\mu$-almost every point of $\PoM_g$. This completes the proof of the Theorem.
\end{proof}

Let us now pivot to Theorem \ref{mainthm:measure convergence strata} and its generalizations.
For a stratum component $\cQ\subset \QT_g$, recall from Definition \ref{def:Q star} that $\cQ^*$ consists of those differentials $q$ obtained from $q'\in \cQ$ by pulling apart higher order zeros of $q'$ \emph{horizontally}.

\begin{theorem}\label{thm:continuity for only short horizontals}
    Let $\cQ$ be a component of a stratum of quadratic differentials and suppose that $q\in \cQ$ has no horizontal saddle connections.  Then $q$ is a point of continuity for $\cO\inverse$ restricted to ${\cQ^*}$.
\end{theorem}

\begin{proof}
Let $q= \cO(X,\lambda) \in \cQ$ have no horizontal saddle connections.
Take any simply horizontally convex cellulation $\mathsf{T}$ of $q$ (for example, the one dual to $\tau(X, \lambda, \delta)$ for small enough $\delta$) and let $\tau$ denote the corresponding dual track, as in Construction \ref{constr:dualtt}. 
Since $q$ has no horizontal saddle connections, it must be the case that $\tau$ is snug and $\arc(X, \lambda)$ is empty.

Applying Lemma \ref{lem:ttsaroundstrata}, there is a relatively open neighborhood $B^*_\mathsf{T}(q)$ of $q$ in $\cQ^*$ such that 
\begin{enumerate}
    \item any $q' \in B^*_\mathsf{T}(q)$ has a horizontally convex cellulation $\mathsf{T'}$ that refines $\mathsf{T}$ and
    \item the dual train track $\tau(q', \mathsf{T}')$ is isotopic to $\tau$.
\end{enumerate}
Now suppose $q_n=\cO(X_n,\lambda_n) \in B^*_\mathsf{T}(q)$ converge to $q$.
By (2) and the fact that dual train tracks carry horizontal foliations, we know that $\lambda_n \prec \tau$, so since $\tau$ is snug for $\lambda$, Lemma \ref{lem:full_convergence} implies that the supports of $\lambda_n$ converge to the support of $\lambda$.
Invoking Theorem \ref{thm:cont_inverse}, we get that $(X_n, \lambda_n) \to (X,\lambda) \in \PoM_g$.
\end{proof}

Theorem \ref{mainthm:measure convergence strata} is then an immediate consequence, as is the following more general formulation.

\begin{corollary}\label{cor:weak star convergence of measures if limit has only short horizontals}
Let $\cQ$ be a component of a stratum of quadratic differentials and let $\mu$ be a locally finite measure on $\cQ$ that gives zero measure to the set of differentials with a horizontal saddle connection. Then if $\mu_n$ is any sequence of locally finite Borel measures on $\cQ^*$ converging weak-$*$ to $\mu$, then $\cO^*\mu_n\to \cO^*\mu$ on $\PM_g$.
\end{corollary}

\begin{proof}[Proof of Theorem \ref{mainthm:measure convergence strata} and Corollary \ref{cor:weak star convergence of measures if limit has only short horizontals}]
The result follows directly from combining Theorems \ref{thm:Boreliso} and \ref{thm:continuity for only short horizontals} with \cite[Theorem 5.1]{B:measures}, as in the proof of Theorem \ref{mainthm:push measures from P to Q}.
\end{proof}

Using our discussion of train tracks dual to cellulations, we can give a much more general description of sequences $q_n \to q$ along which Hausdorff convergence of the horizontal foliations is enforced.

In what follows, let $\cQ$ be a stratum of quadratic differentials and let $q = \cO(X, \lambda) \in \cQ$.
Fix a simply horizontally convex cellulation $\mathsf{T}$ for $q$ containing all of its horizontal saddle connections, for example by taking $\delta$ small enough so that the equilateral train track $\tau(X, \lambda, \delta)$ is snug and then taking the dual cellulation to $\tau \cup \arc$ (Proposition \ref{prop:dual cell veering}).
For any completion $\arc'$ of $\arc$ to a maximal arc system and a standard smoothing $\taua'$ of $\tau \cup \arc'$, let $C(\taua') \subset W_{\CC}(\taua')$\label{ind:Ccone} be the $\RR$-linear cone cut out by requiring that the weights of the branches corresponding to the arcs of $\arc' \setminus \arc$ are real and non-negative.
Proposition \ref{prop:ttcharts_Q*} allows us to cover the relatively open neighborhood $B^*_\mathsf{T}(q)$ of $q$ in $\cQ^*$ with a union of such cones.

For each such $\arc'$, let $D(\taua')$\label{ind:Dcone} denote the sub-cone of $C(\taua')$ on which all weights of {\em all} arcs of $\arc'$ are real and non-negative.
A differential $q' \in B^*_\mathsf{T}(q)$ is in $C(\taua')$ if and only if the new saddle connections arising from breaking up the zero of $q$ all have real period, i.e., are all horizontal.
The differential $q'$ is in $D(\taua')$ if and only if moreover, every saddle connection of $q'$ realizing a horizontal saddle connection of $q$ is also horizontal.
We note that unless $q$ is in the principal stratum of $\QM_g$, it appears at the boundary of both $C(\taua')$ and $D(\taua')$.

Unraveling the identifications of Section \ref{sec:ttperiod coords}, we can also rephrase this data in terms of period coordinates.
Let $\Hor(q) < H_1(\widehat {q}, Z(\widehat q); \ZZ)$\label{ind:Horq} be the $\RR$-span of the horizontal saddle connections of $q$ (really, of its orientation cover); then the subspace\label{ind:Vq}
\[V(q) = \{ \phi | \Im(\phi(z))= 0 \text{ for all } z \in \Hor(q)\}
< H^1(\widehat{q}, Z(\widehat{q}); \CC)^-\]
is a local coordinate for those differentials in $\cQ$ near $q$ on which the horizontal saddle connections of $q$ remain horizontal.

The differentials $q' \in C(\taua')$ all share an isotopy class of collapse map to $q$ (see Lemma \ref{lem:collapsemaps}). Let $\overline{\text{Col}(q)}$ denote the kernel of the induced map on relative homology (with integer coefficients), equivalently, the span of those saddle connections that are collapsed.\label{ind:collapsepreimages}
Let $\overline{\Hor(q)}$ denote the preimage of $\Hor(q)$ in $H_1(\widehat {q}', Z(\widehat q'); \RR)$, equivalently, the span of the saddle connections realizing the horizontal ones of $q$ together with all collapsed saddle connections.
Then $C(\taua')$ and $D(\taua')$ can be identified as full-dimensional cones inside of the subspaces
\[\ImAnn(\overline{\text{Col}(q)}) < H^1(\widehat {q}', Z(\widehat q'); \CC)^-
\hspace{3ex}
\text{ and }
\hspace{3ex}
\ImAnn(\overline{\Hor(q)}) < H^1(\widehat {q}', Z(\widehat q'); \CC)^-,
\]
respectively, where $\ImAnn$ is defined as follows:\label{ind:ImAnn}
for any real vector space $H_1$ and any subspace $K$, one sets
\[\ImAnn(K) = \{ \phi \mid
\Im(\phi(k)) = 0 
\text{ for all } k \in K\}
< \Hom_{\RR}(H_1, \CC).\]

\begin{lemma}\label{lem:special Hausdconv}
With all notation as above, for any sequence $q_n = \cO(X_n, \lambda_n) \in D(\taua')$
such that $q_n \to q$, we have that the supports of $\lambda_n$ converge to the support of $\lambda$ in the Hausdorff topology.
\end{lemma}
\begin{proof}
Let $\tau$ denote the dual train track to our simply horizontally convex $\mathsf{T}$. Since $\mathsf{T}$ contains all of the horizontal saddle connections of $q$, it must be the case that $\lambda$ is carried snugly on $\tau$.

The proof of Lemma \ref{lem:ttsaroundstrata} and our choice of $D(\taua')$ ensure that for every $q_n \in D(\taua')$, the dual train track to the cellulation refining $\mathsf{T}$ is actually isotopic to $\tau$.
In particular, the horizontal foliations $\lambda_n$ are all carried on $\tau$,
so we can apply Lemma \ref{lem:full_convergence} to deduce that 
the supports of $\lambda_n$ converge to the support of $\lambda$.
\end{proof}

This broader continuity criterion allows us to pull back convergence to a broader class of measures on $\QM_g$ along a broader class of sequences.
We give a general definition below; see just after for examples of where these naturally arise in Teichm{\"u}ller dynamics.

\begin{definition}
Let $S$ be a surface with marked points $Z$.
Say that a $\RR$-linear subspace $V < H^1(S, Z; \CC)$ is {\bf special}\label{ind:special} if it is cut out by equations of the form
\begin{itemize}
    \item $\phi(z) = 0$ for $z \in H_1(S, Z; \RR)$ and
    \item $\Im(\phi(z)) = 0$ for $z \in H_1(S, Z; \RR)$.
\end{itemize}
Given a special subspace $V$, define $\Hor(V) < H_1(S, Z; \RR)$ to be the subspace\label{ind:HorV}
\[\Hor(V) := \{ z | \Im(\phi(z))= 0 \text{ for all } \phi \in V\}.\]
There is a natural $\RR$-linear map\label{ind:ev}
\[\begin{array}{rcl}
    ev: V & \to & \Hom_{\RR}(\Hor(V), \RR) \\
    \phi & \mapsto & (z \mapsto \phi(z)).
\end{array}\]
Say that a measure on $H^1(S, Z; \CC)$ is {\bf special} if its support is contained in a special subspace $V$ and, when disintegrated over $ev$, almost every fiberwise measure is absolutely continuous with respect to Lebesgue.

Now let $\cH$ be a component of a stratum of abelian differentials.
Say that a measure $\nu$ on $\cH$ is {\bf special} if, for any precompact period coordinate chart $\hol: U \subset \cH \to H^1(S, \text{Zeros}; \CC)$, we have that
$\hol_* (\nu|_U)$ is the restriction of a finite convex combination of special measures to $\hol(U)$.
Since strata of quadratic differentials can be thought of as affine invariant subvarieties of strata of abelian, we can extend this definition to special measures on stratum components $\cQ \subset \QM_g$.
\end{definition}

This definition bears some explaining. In period coordinates, the equations of the first form cut out an ambient affine submanifold (it is locally $\GL_2^+\RR$ invariant, but perhaps not globally), while the second should be thought of as enforcing that some saddle connections must be horizontal.
The (integral points of the) subspace $\Hor(V)$ record all of the saddle connections that must remain horizontal as one deforms within $V$.

Prototypical examples of special measures come from surgering differentials in families to ``push'' a $P$-invariant measure, resulting in a $U$-invariant one.
Moreover, every known ergodic $U$-invariant measure on $\QM_g$ is special.

\begin{example}\label{eg:break zero push}
Consider an abelian differential with a zero of order at least 2 (or more generally, a quadratic differential with a zero of even order at least $4$).
Breaking up the zero horizontally produces a differential in an adjacent stratum with two lower-order zeros and a horizontal connection connecting them (see \S\ref{subsec:nbhd_strata} and \cite{KZstrata}).
This surgery is an isometry outside of a small ball about the zero in question, and if one performs this surgery with defining parameter $t$, then it is always defined as long as there are no horizontal saddle connections of length at most $t$.

Given a stratum $\cQ$ of such differentials (or more generally any affine invariant subvariety in such a stratum), we can therefore perform this surgery in families outside of a measure $0$ set.
\footnote{Properly, this surgery also depends on a finite amount of combinatorial data that tells you how to break up the zero and in which direction.
The base family is really then a finite cover of $\cQ$ that also records this combinatorial information.}
This gives a way to push the Masur--Smillie--Veech measure of $\cQ$ (respectively, the associated affine measure) into an adjacent stratum, where it will end up supported on the set of surfaces which all have ``the same'' horizontal saddle connection of length $t$. 
It is helpful to think of this measure as living on $\cQ^*$ when $t$ is small, but one can perform this construction for arbitrarily large $t$.
Compare Theorem 6.5 of \cite{BSW:Ratnerhorocycle} and the surrounding discussion.

The resulting measure is clearly special in the sense defined above, with $\Hor(V)$ equal to the span of the horizontal saddle connection obtained by breaking up the zero, and the pushforward measure under $ev$ supported entirely over the constant function $(z \mapsto t)$.
\end{example}

\begin{example}
Consider a stratum $\cH$ of abelian differentials with multiple zeros. The {\bf isoperiodic} (or absolute period, or REL) foliation of $\cH$ has leaves consisting of those differentials with the same absolute periods (but different relative ones).
These leaves should be thought of as ``moving the zeros relative to one another,'' see \cite{McM_twists}.
This foliation further decomposes into real and imaginary parts (corresponding to moving zeros horizontally and vertically, respectively), and the real part is called the {\bf real REL} foliation.

Given an affine invariant subvariety $\AIS$ and a vector field defined $\nu_{\AIS}$ almost everywhere on $\AIS$ tangent to the real REL foliation, one can push the affine measure $\nu_{\AIS}$ by this vector field.
The flow of this vector field is not always well-defined, but since $\AIS$ is rotation-invariant, it is well-defined on a full measure subset. See \cite{BSW:Ratnerhorocycle} and \cite{CWY:weakRatner} (which also discuss the more general construction of pushing a $U$-invariant measure by a real REL vector field).

The resulting measure is special, and if $\AIS$ is cut out by equations of the form $\phi(z_i)=0$ for $z_i$, then the special subspace containing the support of the push of $\nu_{\AIS}$ has equations of the form $\phi(z_i)=0$ and $\Im \phi(z_j)=0$.
We note that if the equations defining $\AIS$ were not defined over $\mathbb{Q}$ then the subspace $\Hor(V)$ may have no integer points, and so the generic point with respect to the pushed measure may have no horizontal saddle connections.
\end{example}

The point of defining this class of measures is that a generic point has no more horizontal saddle connections than anticipated, so we can say along which directions $\cO^{-1}$ is continuous (generically).
Given a special measure $\nu$ on $\cQ$, and a point $q \in \cQ$, we can express (the pushforward of) $\nu$ restricted to a small period coordinate chart containing $q$ as a finite convex combination of special measures on the relevant cohomology group.
Enumerate the special subspaces containing the supports of each of these measures as $V_1, \ldots, V_n$; then say that a point $q$ has {\bf no extra horizontals}\label{ind:noextrahor} if there is some $V_i$ such that
every horizontal saddle connection of $q$ lies in $\Hor(V_i)$, i.e., $\Hor(q) \le \Hor(V_i)$.

\begin{lemma}\label{lem:special generic no saddles}
Let $\nu$ be a special measure on a stratum $\cQ$. Then $\nu$-almost every $q$ has no extra horizontals.
\end{lemma}

\begin{proof}
Since special measures, the subspaces $\Hor(V_i)$, etc., are all defined locally, let us fix a small precompact period coordinate chart for $\cQ$ and work there. The overall result can then be attained by patching together a countable union of such charts.
We can also reduce to the case where the pushforward of $\nu$ on this chart is supported inside of a single special subspace $V$ since we are simply taking a convex combination of them.

A horizontal saddle connection has a real period, and the condition that its period is real defines a codimension 1 $\RR$-subspace of $H^1(S, Z; \mathbb{C})$.
There are countably many possible saddle connections on any quadratic differential, giving rise to countably many such subspaces.
If the homology class of a saddle connection is not in $\Hor(V)$ then its corresponding subspace must meet $V$ transversely; thus, the set of points in $V$ that have more real periods than expected is contained in a countable union of hyperplanes.
These hyperplanes are all transverse to the fibers of $ev$, so in each fiber the same condition holds:
the set of points with more real periods than expected is a countable union of hyperplanes.

We now disintegrate $\nu$ over the evaluation map $ev$. Since the fiberwise measures are almost all absolutely continuous with respect to Lebesgue, this countable union of hyperplanes has fiberwise measure 0 in almost every fiber, hence has $\nu$ measure 0.
\end{proof}

Suppose that $\nu$ is a special measure on $\cQ$.
Let $U \subset \cQ^*$ be a precompact relatively open set in the domain of a period coordinate chart and let $V_1, \ldots, V_n$ denote the special subspaces containing support of $\hol_*(\nu|_U)$.
By definition, each $V_i$ is the intersection of $\ImAnn(\Hor(V_i))$ with the annihilator of $\ker(V_i)$;
as in the discussion before Lemma \ref{lem:special Hausdconv}, let $\overline{\Hor(V_i)}$ denote the preimage of $\Hor(V_i)$ in $H_1(\widehat {q}', Z(\widehat q'); \RR)$ and set\label{ind:specpreimage}
\[ \overline{V_i} := \ImAnn(\overline{\Hor(V_i)}) < H^1(\widehat {q}', Z(\widehat q'); \CC).\]
Restricting to our period coordinate chart $U$, this gives a subset of a neighborhood of $q$ in $\cQ^*$. 

If $\nu_n$ is a sequence of measures $\nu_n \to \nu$ such that $\hol_*(\nu_n|_U)$ is supported on $\bigcup \overline{V_i}$ for every such $U$, then we say that the sequence $\nu_n$ is {\bf specially convergent.}\label{ind:specconv}

\begin{example}
Consider a sequence of measures $\nu_n$ obtained by breaking up a zero along a horizontal saddle connection of length $1 - 1/n$, as in Example \ref{eg:break zero push}. Then the sequence $\nu_n$ converges specially to the measure obtained by breaking up the zero along a horizontal saddle connection of length $1$.
\end{example}

\begin{theorem}\label{thm:pull back special convergence}
Suppose that $\nu$ is a special measure on $\cQ$ and suppose that $\nu_n \to \nu$ is a specially convergent sequence of locally finite Borel measures on $\cQ^*$.
Then $\cO^* \nu_n \to \cO^* \nu$ on $\PM_g$.
\end{theorem}
\begin{proof}
Cover $\cQ^*$ with a countable set of compact period coordinate charts; we prove the Theorem on each.

By Lemma \ref{lem:special generic no saddles}, the $\nu$-generic point has no extra horizontal saddle connections.
That is, on a small period coordinate neighborhood on which $\nu$ looks like a convex combination of special measures supported in special subspaces $V_i$, the following holds:
for almost every $q$, we have that $\Hor(q) < \Hor(V_i)$ for one of these $V_i$.
In particular, we have that
\[\overline{V_i}
= \ImAnn(\overline{\Hor(V_i)})
\le \ImAnn(\overline{\Hor(q)})\]
and so Lemma \ref{lem:special Hausdconv} ensures that if $q_n \to q$ inside $\overline{V_i}$, the supports of the corresponding horizontal foliations must also converge.
Combining this with Theorem \ref{thm:cont_inverse}, we get that $\cO^{-1}$, restricted to the special subspaces containing the support of $\nu$, is continuous at $\nu$-almost every point.

Since all of the measures $\nu_n$ are supported on $\bigcup \overline{V_i}$, we can push them forward along $\cO^{-1}$ even with the domain restriction, and Theorem 5.1 of \cite{B:measures} yields the convergence of their pushforwards.
\end{proof}

We highlight one instance of special convergence; this plus \cite[Theorem 11.1]{BSW:Ratnerhorocycle}
proves Theorem \ref{thm:eigenform genericity}.
Recall that a point is generic for an ergodic flow $(\varphi^t, \mu)$ if ergodic averages along its $\varphi$-flowline converge to the average with respect to $\mu$.

\begin{corollary}
Suppose that $\nu$ is a $U$-invariant ergodic measure on $\cQ$ that is also special.
If $q = \cO(X, \lambda)$ in the support of $\nu$ is generic for $(U, \nu)$, then $(X, \lambda)$ is generic for $(\Eq, \cO^*\nu)$.
\end{corollary}

\addtocontents{toc}{\protect\setcounter{tocdepth}{0}}
\section*{Conflicts of Interest}
None.

\section*{Financial support}
AC acknowledges support from NSF grants
DGE-1122492, 
DMS-2005328, 
and DMS-2202703. 
JF acknowledges support from
NSF grant DMS-2005328 
and DFG grants 
427903332 
and 281071066 – TRR 191. 

\bibliography{references_SHSHII}{}
\bibliographystyle{amsalpha.bst}

\vfill \pagebreak
\part*{Index}

{\noindent Here, we collect many of the constants and notations used throughout the paper.}
\\

$s$, a lower bound on the systole

$\delta_{\ref{prop:ttdefs_comp}}$, 
$\delta_{\ref{lem:equittswork}}$,
$\delta_{\ref{lem:change base ties still meet}}$,
defining parameters for train track neighborhoods

$\zeta_{\ref{prop:persistent}}$,
$\zeta_{\ref{lem:fixX persistent}}$,
$\zeta_{\ref{cor:persist_notallarcs}}$,
$\zeta_{\ref{prop:tiestable}}$,
$\zeta_{\ref{lem:ttstablefixX}}$,
$\zeta_{\ref{thm:contcell}}$,
$\zeta_{\ref{thm:cont_inverse_stable}}$, 
$\zeta_{\ref{prop:stable_hausdorff_stratum}}$,
thresholds for Hausdorff-close laminations

$B_{\ref{prop:persistent}}(\zeta)$, 
$B_{\ref{lem:ttstablevarX}}(\ell)$,
$B_{\ref{thm:contcell}}(\zeta)$,
neighborhoods of a fixed $X \in \T_g$

$U_{\ref{thm:cont_inverse_stable}}^{ss}$,
$U^{ss}_{\ref{prop:stable_hausdorff_stratum}}$,
$U_{\ref{cor:lamination_close_star_nbh_stratum}}$,
neighborhoods of a fixed $q \in \QT_g$

$\xi_{\ref{thm:shapeshift_distance_small}}$,
a threshold for the size of a combinatorial deformation $\ac$

$w_{\ref{prop:ttdefs_comp}}$, comparing the width of uniform and equilateral neighborhoods

$W_{\ref{lem:equittwidth}}$, describing the width of an equilateral neighborhood

\begin{multicols}{2}

$\arc_*$ \pageref{ind:arccompletion}

$\arc(X,\lambda)$ \pageref{ind:arcsystems}

$\arc(q, {\mathsf{T}})$ \pageref{ind:ttdualtocell}

$\arc_\circ^\tau (\lambda)$ \pageref{ind:topinvis}

$\arc_\bullet^\tau(X,\lambda)$ \pageref{ind:topvis}

$\gvarc$ \pageref{ind:geominvis}

$\arc(\zeta)$ \pageref{ind:persistnotallarcs}

$\arcwt(X,\lambda)$ \pageref{ind:arcsystems}

$A(V)$ \pageref{ind:elemshaping}

$A(\alpha)$ \pageref{ind:shapingspine}

$|\Arcfill(S\setminus \lambda)|_\RR$ \pageref{ind:arccx}

Affine measures \pageref{ind:AIS}

Augmented equilateral train tracks \pageref{ind:augtt}
\\

$B_\mathsf{T}(q)$ \pageref{ind:BTq}

$B^*_\mathsf{T}(q)$ \pageref{ind:BTq*}

$B^{\cQ}_{\mathsf{T}}(q)$ \pageref{ind:BQT(q)}

$\mathscr B(S\setminus \lambda)$ \pageref{ind:base}

Basepoints of triples \pageref{ind:basepts}

Binding pairs \pageref{ind:bind}

Branches of a train track \pageref{ind:ttbasics}
\\

$c_\alpha$ \pageref{ind:arcsystems}

$C(\taua')$ \pageref{ind:Ccone}

$\overline{\text{Col}(q)}$ \pageref{ind:collapsepreimages}

Centers of triples \pageref{ind:center}

Chain-recurrence \pageref{ind:chainrec}

Collapse maps \pageref{ind:geomcollapse}, \pageref{ind:topcollapse}
\\

$\Delta_V(\delta)$\pageref{ind:DVd}

$d_\lambda(b)$ \pageref{ind:dlb}

$d_X^H$ \pageref{ind:Hausdmetric}

$d_{\hat x}(A, B )$ \pageref{ind:PSL2Rmetric}

$d_\Gamma(X,X')$ \pageref{ind:dGamma}

$D_\delta$ \pageref{ind:Ddelta}

$D(\taua')$ \pageref{ind:Dcone}

Dilation rays \pageref{ind:dilrays}
\\ 

$ev$ \pageref{ind:ev}

${\cE}_{\delta}(\lambda)$ \pageref{ind:equinbhd}

Equilateral neighborhoods \pageref{ind:equinbhd}

Equilateral train tracks \pageref{ind:equitt}

Extensions of train tracks \pageref{ind:extensions}
\\

$\varphi_\ac$ \pageref{ind:shapeshift}

$\varphi(k_b)$ \pageref{ind:shapeshifttrans}

$\varphi(k)$ \pageref{ind:shapingspine}

$\varphi(V)$ \pageref{ind:elemshsh}

$f_{X,\ac}(V)$ \pageref{ind:sharpnessfunctions}

Flipping smoothings \pageref{ind:flipping}

Fully carrying \pageref{ind:fully}
\\

$G_P(X,\lambda)$ \pageref{ind:config}

$\mathcal {GL}^{cr}(X)$ \pageref{ind:chainrec}
\\

$\Hor(q)$ \pageref{ind:Horq}

$\overline{\Hor(q)}$ \pageref{ind:collapsepreimages}

$\Hor(V)$ \pageref{ind:HorV}

$\overline{\Hor(V)}$ \pageref{ind:specpreimage}

$\cH(\lambda)$ \pageref{ind:transversecoc}

Hexagons \pageref{ind:hexagon}

Horizontally convex, simply horizontally convex \pageref{ind:horconv}

Horospherical measures \pageref{ind:horosphere}
\\

$\ImAnn$ \pageref{ind:ImAnn}

Injectivity radius $\injrad_q(x)$ \pageref{ind:flatthickthin}

Invariant subvarieties \pageref{ind:AIS}

Invisible arcs \pageref{ind:geominvis}, \pageref{ind:topinvis}
\\

$k_b$ \pageref{ind:transversal}
\\

$\widehat {\lambda}\cup\widehat{\arc}$ \pageref{ind:lamor}

$\ell(X, \lambda, \delta)$ \pageref{ind:lxdb}

$\ell_{\ac}(\alpha)$ \pageref{ind:sharpnessfunctions}

$L_{\rho}(\Gamma)$, $L_X(\Gamma)$ \pageref{ind:gensetlength}
\\

$\MF_g$ \pageref{ind:MF}

$\MF(\lambda)$ \pageref{ind:MFlambda}

$\ML_g$ \pageref{ind:ML}

Masur--Smillie--Veech measures \pageref{ind:MSV}

Mirzakhani measure \pageref{ind:Mirzmeasure}
\\

$\nu_{\cQ}$, $\nu_{\cQ}^1$ \pageref{ind:MSV}

$N_{\arc}$, $\widehat {N_{\arc}}$ \pageref{ind:snugtopnbhd}

$\cN_{\delta}(\lambda)$ \pageref{ind:unifnbhd}

No extra horizontals \pageref{ind:noextrahor}
\\

$O_s(f)$ \pageref{ind:Os}

$\cO$ \pageref{ind:O}

$\cO^*$ \pageref{ind:O*}

$\Ol$ \pageref{ind:orthofolmap}

$\cO_{\partial Y}(Y)$ \pageref{ind:orthofol}

Orientations of $\lambda\cup \arc$ \pageref{ind:lamor}

Orthogeodesic foliation \pageref{ind:orthofol}
\\

$\PM_g$, $\PoM_g$, $\PT_g$, $\PoT_g$ \pageref{ind:PMg}

$\PT_g(\tau)$ \pageref{ind:PTtau}

Period coordinates \pageref{ind:period coordinates}

Proto-spike \pageref{ind:proto-spike}
\\

$\widehat{q}$ \pageref{ind:period coordinates}

$\cQ$, $\cQ^1$ \pageref{ind:strata}

$\cQ^*$ \pageref{ind:Q*}

$\QM_g$, $\QoM_g$, $\QT_g$, $\QoT_g$ \pageref{ind:QMg}

$\QT_g(\tau)$ \pageref{ind:QTtau}
\\

$\rho_\ac$ \pageref{ind:shapeshift}

$r(V)$ \pageref{ind:depth}

Realizing a saddle connection \pageref{ind:rere}

Refining a cellulation \pageref{ind:rere}
\\

$\sigl(X)$ \pageref{ind:geomshsh}

$s$-thick, $s$-thin \pageref{ind:thickthin}

$\Sp$ \pageref{ind:spine}

$\ac$ \pageref{ind:combdef}

$\SH(\lambda)$ \pageref{ind:SHlambda}

$\SH^+(\lambda)$ \pageref{ind:SH+}

$\SH(\lambda;\arc)$ \pageref{ind:shshwitharc}

(SH0), (SH1), (SH2), (SH3) \pageref{ind:shshaxioms}

Shape-shifting cocycle \pageref{ind:shapeshift}

Shear-shape cocycles \pageref{ind:shshcoc}

Slide equivalence \pageref{ind:slideequiv}

Smoothings, standard smoothings \pageref{ind:smoothing}

Special convergence \pageref{ind:specconv}

Special subspaces, measures \pageref{ind:special}

Spine \pageref{ind:spine}

Standard transversals \pageref{ind:standardtrans}

Strata \pageref{ind:strata}

Stretchquakes \pageref{ind:stretchquakes}

Switches of a train track \pageref{ind:ttbasics}
\\

$\tau(q, {\mathsf{T}})$ \pageref{ind:ttdualtocell}

$\tau(X,\lambda, \delta)$ \pageref{ind:uniftt}, \pageref{ind:equitt}

$\mathsf{T}$ \pageref{ind:dualcelltott}

$\text{Thick}_\delta(Y)$, $\text{Thin}_\delta(Y)$ \pageref{ind:equithickthin}

$\Thick_t(q)$, $\Thin_t(q)$ \pageref{ind:flatthickthin}

Ties of a train track \pageref{ind:ttbasics}

Train paths \pageref{ind:ttbasics}

Transverse cocycles \pageref{ind:transversecoc}
\\

Uniform geometric train track \pageref{ind:uniftt}
\\

$V(q)$ \pageref{ind:Vq}

$\overline{V}$ \pageref{ind:specpreimage}

$(\mathcal V, \le )$ \pageref{ind:branchprotospikes}

$\mathcal V_{ \varepsilon}$ \pageref{ind:protospike_eps}

Veering saddle connections \pageref{ind:veering}

Visible arcs \pageref{ind:geominvis}, \pageref{ind:topvis}
\\

$\omega_{\SH}$ \pageref{ind:Thform}

$W(\tau)$, $W^{>0}(\tau)$, $W^{\ge0}(\tau)$, $W_\mathbb{C}(\tau)$ \pageref{ind:wtspaces}

$W^{>0}_{\text{snug}}(\tau)$ \pageref{ind:snugwts}

$\Fol^{uu}$, $\Fol^{ss}$ \pageref{ind:unstablefols}

Width of a train track \pageref{ind:ttbasics}
\\

$X_\ac$ \pageref{ind:shapeshift}
\\

$Y_q$ \pageref{ind:Yq}
\\

$\zeta$-equidistant \pageref{ind:zetaequid}

$[z]_+$ \pageref{ind:z+}
\\

$\preceq$ \pageref{ind:fully}

\end{multicols}
\Addresses
\end{document}